 \documentclass[11pt,letterpaper]{amsart}
\usepackage[utf8]{inputenc}
\usepackage[T1]{fontenc}
\usepackage[english]{babel}
\usepackage{amsmath, amssymb, amsthm, amsfonts}
\usepackage{mathrsfs}
\usepackage{mathtools}
\usepackage{xcolor}
\usepackage{graphicx}
\usepackage{enumitem}
\usepackage{tikz}
\usetikzlibrary{arrows}
\tikzstyle{block}=[draw opacity=0.7,line width=1.4cm]
\usepackage{tikz}
\usepackage{pgfplots}
\pgfplotsset{compat=1.18} 

\usepackage{hyperref}
\hypersetup{
    colorlinks=true,
    linkcolor=blue,
    filecolor=magenta,
    urlcolor=blue,
    citecolor=blue,
}

\newtheorem{theorem}{Theorem}[section]
\newtheorem{lemma}[theorem]{Lemma}

\newtheorem{proposition}[theorem]{Proposition}
\newtheorem{remark}[theorem]{Remark}
\newcommand{\dive}{\operatorname{div}}
\newtheorem{fact}[theorem]{Fact}
\newtheorem{corollary}[theorem]{Corrolary}
\newcommand{\Om}{\Omega}
\newtheorem{conjecture}[theorem]{Conjecture}

\title[Critical Wave Equation with Nonlinear Kelvin-Voigt Damping]{The Quintic Wave Equation with Kelvin-Voigt Damping: Strichartz Estimates, Well-posedness and Global Stabilization}

\author{Marcelo M. Cavalcanti and Val\'eria N. Domingos Cavalcanti}
\address{Department of Mathematics, State University of Maringá, Maringá, PR, Brazil.}
\email{mmcavalcanti@uem.br; vndcavalcanti@uem.br}


\begin{document}

\begin{abstract}
We study the three-dimensional energy-critical quintic wave equation with localized Kelvin--Voigt damping. While the viscoelastic dissipation provides a remarkably efficient mechanism for removing energy, it simultaneously introduces a loss of derivatives that places the problem outside the classical hyperbolic framework for critical wave equations.

We develop a frequency-space approach based on Littlewood--Paley theory that yields well-posedness for arbitrarily large initial data despite the derivative loss. We then establish uniform exponential stabilization by combining critical Strichartz estimates, microlocal defect measures, and the unique continuation principle of Duyckaerts, Zhang and Zuazua, obtaining observability under highly localized damping regions.

The paper also identifies a structural distinction between Kelvin--Voigt damping and all previously studied lower-order dissipations. Under the critical rescaling associated with energy concentration, the Kelvin--Voigt operator becomes supercritical, revealing an obstruction absent from the classical theory. We isolate this obstruction explicitly and propose a new parabolic viewpoint suggesting that concentration inside the strictly damped region is governed by intrinsic local regularization rather than by the traditional hyperbolic concentration mechanism.

\end{abstract}
\maketitle

\tableofcontents

\section{Introduction}

\subsection{The phenomenon of energy concentration}
The concentration phenomena associated with the "quintic wave", generally described by the Nonlinear Quintic Schrödinger Equation - NLS or the Semilinear Quintonic Wave Equation, refer to situations where the energy or mass of the wave accumulates intensely in a very small region of space (a point or a small neighborhood), often leading to what is called a "blow-up" (singularity in finite time). From a physical point of view, this represents the collapse of a wave packet, while mathematically it corresponds to a concentration of mass or energy, where the solution profile loses compactness in Sobolev spaces. Mass/Energy Concentration: Mathematically, a sequence of solutions is said to concentrate if the mass or energy concentrates at a point when the concentrated mass is often a discrete quantity, equal to the norm of the "ground state". There is a threshold profile, namely, scattering threshold(the center-stable manifold), that separates solutions that scatter from those that concentrate and blow-up.  Concentration is analyzed through scaling transformations, where the solution is rescaled to maintain a constant energy norm while the wave packet size goes to zero.

The phenomenon of "collapse" or blow-up in quintic wave equations  can perfectly occur in bounded domains. In fact, the presence of edges  can facilitate blow-up compared to free space, as confinement prevents energy from dispersing to infinity, forcing it to concentrate. Here is the detailed analysis based on mathematical physics:

1. Collapse in Bounded Domains.
From a physical point of view, in a bounded domain, the initial energy "trapped" in the volume, combined with the non-linear focus, can overcome the linear mechanisms of dispersion. The wave "collapses" because the energy concentrates at a point or region in finite time, resulting in the singularity.

2. Mathematical View and Sobolev Spaces.
Mathematical collapse is the blow-up of the energy norm or norms in Sobolev spaces, such as:

\begin{itemize}
\item~As time approaches the collapse time, the solution ceases to be compact, meaning that the gradient of the solution becomes infinite, concentrating at one or more points.

\item~ In bounded domains, immersion theorems, such as the Rellich-Kondrachov, guarantee initial compactness, but the quintic nature of nonlinearity is "critical" in energy space. When the blow-up occurs, this compactness is violated, as the solution profile becomes singular (tends to a Dirac delta-type distribution in the limit).
\end{itemize}

3. Factors Affecting Collapse:

The shape of the domain and boundary conditions (e.g., Dirichlet at the edge) influence the rate of energy concentration, but collapse is an interior concentration phenomenon, i.e., it occurs where nonlinearity is strong.

If the system includes damping terms (damped wave equations), it is possible for the system to converge to a global attractor, avoiding collapse, depending on the intensity of the damping.

Collapse is characteristic of focusing equations (where nonlinearity attracts the wave), while in defocusing equations, the wave disperses.

In summary, the quintic wave in bounded domains frequently exhibits collapse (finite-time singularity) where the energy norm in Sobolev space explodes, a phenomenon that can be exacerbated by the inability of energy to escape the domain.

\subsection{The drawback of Galerkin's method}

Since we are in a bounded domain and with the spectral basis of the Laplacian, the weak existence and regularity of the approximate problem, via Galerkin, are practically "free" due to compactness and Aubin-Lions. The term $u^5 \in L^{6/5}$ is perfectly tamed by the Sobolev immersion $H_0^1 \hookrightarrow L^6$.

The effective linear term in the interval $I$ for the approximation $m$ is: $$A_m(I) = \|S(t-t_0)(u_m(t_0),\partial_t u_m(t_0))\|_{L^5_t L^{10}_x(I)}.$$

We have, for a fixed $m$, that, by the absolute continuity of the Lebesgue integral, we can shrink the interval $I$ such that $A_m(I) \le \epsilon_0$. Thus, we divide $[0,T]$ into $N_m$ subintervals.

The major problem is in the transition to the limit $m \rightarrow \infty$: For the final bound $\|u_m\|_{L^5_t L^{10}_x([0,T])} \le N \cdot C(E_0)^5$ not to explode when we take the supremum at $m$, the number of subintervals $N$ must be independent of $m$. This means that we need a uniform $\delta > 0$ (independent of $m$) such that, if $|I| < \delta$, then $A_m(I) \le \epsilon_0$ for all $m$.

In the subcritical case (for example, cubic power in $3D$), we achieve this uniform $\delta$ using subcritical Sobolev immersions interpolated with the energy $E_0$. Energy limits the Strichartz norm uniformly within small intervals (see, for instance, \cite{Cavalcanti2018}).

In the critical case (quintic power in $3D$, norm $L^5 L^{10}$), the Strichartz exponent "scales" exactly like the energy. The fact that the energy $E_0(u_m)$ is uniformly bounded does not prevent concentration. As the base $m$ grows, the initial data $(u_m(0), \partial_t u_m(0))$ can form increasingly finer "peaks" (energy concentration).

For a thin peak, the energy $H^1$ remains the same (limited by $E_0$). But the norm $L^5_t L^{10}_x$ of the linear evolution concentrates in an increasingly smaller time interval around $t_0$. Therefore, to keep $A_m(I) \le \epsilon_0$, the size of the interval $\delta_m$ has to shrink to zero when $m \to \infty$. Consequently, the number of partitions $N_m \to \infty$, and the bound $N_m C(E_0)^5$ explode. It fails to extract a weak bound in $L^5_t L^{10}_x$ because the Galerkin bound is not uniform. This is precisely why the traditional Galerkin method fails to give uniqueness for the critical case, even in bounded domains.

To salvage this situation, we have two direct solutions:

Approach 1: Small Data. If we assume that the initial energy $E_0$ is small enough, the global Strichartz estimate gives us: $$A_m([0, T]) \le C E_0 \le \epsilon_0$$ uniformly for all $m$. We don't need to divide the interval. $N=1$. And the usual Bootstrap lemma closes the proof and we obtain the uniform bound $L^5 L^{10}$ in $m$. Then it's just a matter of passing the limit and it's over.

Approach 2: If Large Data is needed (we have to change the Galerkin method to the Strichartz norm).
If we want uniqueness for large initial data, we prove the existence of the weak solution by Galerkin (as we already know it works). But uniqueness in $L^5 L^{10}$ does not come from estimating $u_m$. Instead, the existence of solutions in this critical space (which guarantees uniqueness) is generally achieved by constructing the local solution via the Picard Fixed-Point Theorem directly in the Strichartz space $C([0,\delta]; H_0^1) \cap L^5([0,\delta]; L^{10})$. The fixed point gives us a local existence time $\delta$ that depends on the profile of the initial data, and not just on $E_0$. (To extend this $\delta$ to a global $T$ for large data, then the heavy non-concentration arguments would come into play).

In summary: Weak existence proof using Galerkin and compactness works perfectly. But if the intention is to prove that the Galerkin limit is $L^5 L^{10}$ to guarantee uniqueness with large data, partitioning fails because $N$ goes to infinity with $m$. If it's for small data, partitioning is unnecessary and Bootstrap shines.

\subsection{Physical perspective}

From a physical point of view, the problem of concentration boils down to two words:

\noindent{ \bf Scale Invariance.}
We'll show the "Physical perspective on scale invariance" using the wave's own energy. Imagine we have a wave packet $u(t,x)$ with a total energy $E$. What happens if we try to "squeeze" this packet to make it very small and vibrate very fast? Imagine we want to "zoom in" on our solution, compressing time and space by a factor of $\lambda > 0$. We guess a generic profile with an unknown amplitude $\alpha$:$$u_\lambda(t,x) = \lambda^\alpha u(\lambda t, \lambda x).$$ We do this using a scale factor $\lambda > 1$ (where $\lambda \to \infty$ means concentrating the wave into a tiny point). To maintain the balance between the {\bf linear operator} ($\Box u=\partial_t^2 - \Delta $) and the {\bf non-linearity} ($u^p$), the squeezed wave packet needs to have the form:
$$u_\lambda(t,x) = \lambda^{\frac{2}{p-1}} u(\lambda t, \lambda x)$$
Now, let's see what happens to the {\bf Potential Energy} of this packet when we squeeze it. In dimension 3 (its space $\Omega \subset \mathbb{R}^3$), the change of variables $y = \lambda x$ generates a factor of $\lambda^{-3}$ in the volume. Calculating the energy of the squeezed package, we arrive at a beautiful exponent:
$$E_{pot}(u_\lambda) = \lambda^{\frac{5-p}{p-1}} E_{pot}(u)$$
And it is precisely here that the physics of the phenomenon emerges!

\noindent{\bf The Subcritical Case ($p < 5$): "The Subcritical Energy Penalty"}

Let's consider $p = 3$. The exponent up there becomes $\frac{5-3}{3-1} = 1$. That is:
$$E_{pot}(u_\lambda) = \lambda^1 E_{pot}(u)$$
If the wave tries to concentrate at a point (make $\lambda \to \infty$), the energy needed to maintain that shape goes to infinity. Since our physical system only has a finite initial energy $E_0$, nature simply forbids the wave from concentrating!
Physics says: "You don't have enough energy to vibrate that fast in that tiny space." That's why in the subcritical case the wave is nice, it spreads out, and the argument of dividing the time intervals (which we discussed before) works perfectly.

\noindent{\bf The Critical Case ($p = 5$): "The critical threshold"}

Now, if $p = 5$, the exponent becomes $\frac{5-5}{5-1} = 0$!
$$E_{pot}(u_\lambda) = \lambda^0 E_{pot}(u) = 1 \cdot E_{pot}(u)$$
When we squeeze the wave, the energy doesn't change! The conservation of energy in the system imposes no limitation on concentration.
Physically, the wave realizes that it can shrink its radius of action to the size of an atom ($\lambda \to \infty$) and vibrate at infinite frequencies, packing all the initial energy $E_0$ into a geometric Dirac delta, without expending a single joule more of total energy. The PDE has no repulsive mechanism to prevent the collapse of the wave packet. The same happens with the kinetic energy.

\noindent{\bf The Consequence:}~This is why, at high frequencies, the problem arises. Since energy does not prevent the concentration of infinite frequencies, the wave packet can create these Dirac delta-type singularities out of thin air (the geometric blow-up-see figure below).
And that is precisely why, as we will show next, our move to {\bf switch Kelvin-Voigt damping} with smooth multipliers ({\bf smooth spectral operators}) is so good: we use the fact that the {\bf energy tail} (of the series) converges to zero in the limit, {\bf \em cutting off the "fuel" before the wave can build its Dirac delta singularity!}

\begin{figure}[htpb]
    \centering
    \begin{tikzpicture}
        \begin{axis}[
            axis lines = center,
            xlabel = {$x$},
            ylabel = {$u_\lambda(t_0, x)$},
            ymin = 0, ymax = 3.5,
            xmin = -2, xmax = 2,
            ticks = none, 
            width = 10cm,
            height = 7cm,
            every axis x label/.style={at={(ticklabel* cs:1.02)}, anchor=west},
            every axis y label/.style={at={(ticklabel* cs:1.02)}, anchor=south},
            legend style={at={(0.95,0.95)},anchor=north east, draw=none} 
        ]
        \addplot [
            domain=-2:2,
            samples=100,
            color=blue,
            thick
        ]
        {1 * exp(-(x^2))};
        \addlegendentry{$\lambda = 1$}

        \addplot [
            domain=-1.5:1.5,
            samples=100,
            color=teal,
            thick
        ]
        {1.414 * exp(-(4*x^2))}; 
        \addlegendentry{$\lambda = 2$}

        \addplot [
            domain=-1:1,
            samples=100,
            color=orange,
            thick ]
        {2 * exp(-(16*x^2))}; 
        \addlegendentry{$\lambda = 4$}

        \addplot [
            domain=-0.5:0.5,
            samples=200,
            color=red,
            very thick
        ]
        {2.828 * exp(-(64*x^2))}; 
        \draw[->, thick, dashed, darkgray] (axis cs: 0.6, 0.8) -- (axis cs: 0.15, 2.5)
            node[midway, right=2mm, align=left] {Energy\\ Concentration};
        \end{axis}
    \end{tikzpicture}
\end{figure}
\begin{remark}
"We take this opportunity to clarify that in our previous work \cite{Cavalcanti2024}, the well-posedness framework via Galerkin's method for small data was implicitly assumed. Indeed, our main objective was to show that the Strichartz estimates of Blair, Smith, and Sogge \cite{BlairSmithSogge2009}, together with the unique continuation principle due to Duyckaerts, Zhang, and Zuazua \cite{Duyckaerts}, constitute the precise tools for obtaining well-posedeness as well as exponential decay of energy of the quintic wave equation subjected a localized frictional damping $a(x) u_t$. For this reason, we assumed, without any explicit mention, that Galerkin's method would work well for small datasets and that, for large datasets taken in bounded sets of phase space, $H_0^1\times L^2$, we would need to prove (or assume) that we have a Shatah and Struwe solution, as done in \cite{Kalantarov}, \cite{Laurent}, although we did cite those references. Since our primary focus was on exponential decay, we did not mention this fact. Consequently, the present work formalizes and extends those foundational assumptions."
\end{remark}

The article by Burq, Lebeau, and Planchon (in short BLP) \cite{BurqLebeauPlanchon} or Blair, Smith, and Sogge (BSS) \cite{BlairSmithSogge2009} does not guarantee the uniform boundedness of Strichartz estimates for the Galerkin method approximations. Let's explain exactly where the calculation stalls and why the literature avoids Galerkin at this specific step, going straight to the 'wound' of harmonic analysis: The Spectral Projector Problem in $L^p$. In the classical Galerkin method, we project the equation onto a finite-dimensional space using the orthogonal projector $P_m$, which is an "abrupt" spectral cut on the first $m$ eigenvectors of the Dirichlet Laplacian. To obtain a Strichartz estimate for the approximate solution $u_m$, you apply Duhamel's formula to the projected equation: $$\partial_{tt} u_m - \Delta u_m + P_m(u_m^5) = 0.$$ The problem lies in the non-linear term $P_m( u_m^5)$. To close the estimate, we use the Strichartz bounds of the linear operator, which the BLP actually proves to hold in bounded domains. For this, we need the projector $P_m$ to be uniformly bounded in the spatial $L^p$ spaces required by the Strichartz duals (for example, something other than $p=2$).

This brings us to a classic and fundamental theorem of harmonic analysis (which goes back to the Fefferman ball multiplier problem \cite{Fefferman}), and was exhaustively studied by Christopher Sogge  and collaborators in domains/manifolds): The abrupt spectral projector $P_m$ is not uniformly bounded on $L^p$ for $p \neq 2$. The operator norm $\|P_m\|_{\mathcal{L}(L^p, L^p)}$ explodes when $m \to \infty$. Therefore, we cannot quote the projected nonlinearity independently of $m$. The constant of the Strichartz estimate from Galerkin's method for $u_m$ will depend on $m$ and explode in the limit. The famous BLP paper (2008, in JAMS) proves the Strichartz estimates for the linear equation in domains with a boundary (overcoming wave reflection). But to deal with frequency cutoffs, they don't use abrupt cutoffs like in Galerkin. They use pseudo-differential operators and smooth multipliers (Littlewood-Paley type truncations). The smooth cutoff function "damps" the oscillation and ensures the bound on $L^p$. Furthermore, to construct the solution to the critical nonlinear problem, BLP (as well as Shatah-Struwe in $\mathbb{R}^3$) does not use the Galerkin approximation. They construct the local solution directly in the Strichartz space $L^5_t L^{10}_x$ using Picard's Fixed Point Theorem, and then use energy non-concentration techniques (Morawetz type) to extend the solution globally.

\subsection{Literature Overview}

The stabilization of wave equations with localized dissipation has been
extensively studied within the geometric control framework initiated by
Bardos, Lebeau, and Rauch \cite{BLR}, where sharp conditions linking
observability, control, and decay were established. Subsequent developments
clarified the microlocal propagation mechanisms governing energy concentration,
particularly near the boundary, through the analysis of generalized
bicharacteristics \cite{ Burq-Gerard,Lebeau-Robbiano,Littman-Betelu,Melrose,Melrose-Sjostrand,Melrose-Sjostrand2,Lebeau}.
Particularly those involving inhomogeneous media, variable
coefficients, or complex boundary configurations, the classical localized GCC fails due to
the existence of trapped rays. These are generalized bicharacteristics that oscillate indefinitely
without ever entering a standard localized control region, allowing high-frequency energy to
concentrate and consequently preventing uniform exponential decay.

In the presence of Kelvin--Voigt damping, however, the analytical structure
changes significantly. The damping operator introduces higher-order spatial
effects which interact nontrivially with geometric multipliers, frequently
requiring additional regularity assumptions. This phenomenon has been examined
in various stabilization settings, including localized and degenerate
dissipation mechanisms (see \cite{Ammari}), \cite{Liu-Rao}.

From the dispersive viewpoint, Strichartz estimates play a decisive role in
handling critical nonlinearities. In domains with boundary, these estimates
require refined harmonic analysis tools, as developed by Blair, Smith, and
Sogge \cite{BlairSmithSogge2009}, \cite{SmithSogge1995}, extending earlier Euclidean results of
Keel and Tao \cite{KeelTao}. The interaction between boundary geometry and
dispersive smoothing was further explored by Burq, Lebeau, and Planchon
\cite{BurqLebeauPlanchon}, providing a fundamental framework for energy-critical
wave equations.

Microlocal defect measures and H-measures provide the natural language for
describing high-frequency energy concentration in such problems. These tools,
introduced by G\'erard \cite{Gerard} and Tartar \cite{Tartar}, have become
central in stabilization theory, particularly in configurations where compactness
arguments fail or where dissipation degenerates.

The present work combines these perspectives. By integrating harmonic analysis,
nonlinear energy methods, and microlocal propagation techniques, we obtain a
stability theory capable of handling Kelvin--Voigt mechanisms at the critical
level while preserving the geometric intuition underlying classical control
results.

\subsection{Contributions of the present article}

The energy-critical quintic wave equation occupies a distinguished position in nonlinear partial differential equations, representing the threshold between global dispersive dynamics and the possible formation of singular energy concentrations. During the last three decades, a remarkable body of work---beginning with the pioneering contributions of Grillakis \cite{Grillakis}, Shatah--Struwe \cite{ShatahStruwe}, \cite{Struwe}, and later Burq, Lebeau and Planchon \cite{BurqLebeauPlanchon}---has established that, for the undamped equation and for lower-order perturbations, global well-posedness ultimately hinges upon one fundamental principle: the exclusion of concentration of the critical energy. This is achieved through a delicate synthesis of dispersive estimates, concentration-compactness arguments, Morawetz identities, and rigidity theorems for possible blow-up profiles.

The purpose of the present paper is to investigate how this picture changes when the wave equation is subjected to a localized Kelvin--Voigt dissipation. Unlike classical frictional damping, the Kelvin--Voigt mechanism acts through a spatial derivative of the velocity,
\begin{eqnarray*}
-\operatorname{div}(a(x)\nabla u_t),
\end{eqnarray*}
thereby combining two apparently contradictory features. Physically, it constitutes one of the most efficient mechanisms for dissipating mechanical energy, producing a strong thermo-viscous regularization inside the damping region. Analytically, however, the additional spatial derivative destroys the functional framework on which the classical hyperbolic theory is built, introducing a loss of derivatives that propagates throughout both the local well-posedness and the asymptotic stabilization analysis.

The first contribution of this work is a complete well-posedness theory for arbitrarily large initial data. Rather than working directly in the physical space, where the Kelvin--Voigt operator naturally lives in $H^{-1}$ and is therefore incompatible with the standard nonhomogeneous Strichartz estimates, we shift the entire analysis to the frequency domain. A Littlewood--Paley decomposition separates the solution into low and high spectral components. At low frequencies, Bernstein inequalities convert spectral localization into derivative gains, lifting the Kelvin--Voigt operator into an $L^2$ framework where the Strichartz estimates become effective. At high frequencies, instead of treating the damping as a perturbation, we incorporate it into the principal operator. The resulting commutator generated by the smooth spectral projector is shown to be a smoothing operator of order $-1$, allowing the derivative loss to be completely absorbed without affecting the critical dispersive estimates. This frequency-space strategy provides a robust local theory for arbitrary energy levels and removes the small-data restriction inherent in the classical Galerkin approach.  The application of dispersive techniques to boundary value problems possesses a rich and profound history. The seminal works of Smith and Sogge \cite{SmithSogge1995} demonstrated how wave dispersion interacts with domains and rough boundaries. Subsequently, the celebrated contributions of the trio Burq, Lebeau, and Planchon \cite{BurqLebeauPlanchon} (BLP) solidified the use of spectral projectors and {\bf Littlewood-Paley} theory tailored to Dirichlet boundary conditions, allowing Strichartz estimates to flourish outside of $\mathbb{R}^n$. Furthermore, treating the singular damping terms via the Duhamel formulation crucially relies on the foundational machinery of Christ and Kiselev \cite{ChristKiselev}, which elegantly translates homogeneous dispersive bounds into retarded estimates, allowing low-frequency perturbations to be seamlessly absorbed.

Our second objective concerns the long-time dynamics. We establish the uniform exponential stabilization of the total energy by combining harmonic analysis with microlocal propagation techniques (see \cite{Gerard}). Here the Kelvin--Voigt operator introduces a second fundamental difficulty. The residual sequence produced by the contradiction argument no longer converges in the natural $H^{-1}$ topology, as happens for frictional damping, but only at the weaker $H^{-2}$ level. Classical compactness arguments therefore become insufficient. To overcome this obstacle we employ microlocal defect measures together with the critical Strichartz regularity
\begin{eqnarray*}
u\in L^4(0,T;L^{12}(\Omega))
\end{eqnarray*}
and the sharp unique continuation principle of Duyckaerts, Zhang and Zuazua \cite{Duyckaerts}. This combination allows the observability argument to be closed without imposing any geometric restrictions beyond those required for the generalized bicharacteristic flow. In particular, the resulting stabilization mechanism remains fully compatible with non-invasive damping geometries occupying arbitrarily small Lebesgue measure, thereby circumventing the classical obstruction created by trapped rays.

The most delicate issue of the paper arises in the global extension of strong solutions. In the classical energy-critical wave equation, concentration is excluded through a rescaling argument leading to a limiting elliptic profile. For Kelvin--Voigt damping, however, the critical scaling reveals a completely different phenomenon. Under the self-similar transformation associated with a possible concentration bubble, lower-order dissipative terms such as $a(x)u_t$ remain perturbative. The Kelvin--Voigt operator does not. Instead, it acquires the singular prefactor
\begin{eqnarray*}
\lambda^{-1},
\end{eqnarray*}
becoming supercritical exactly at the scale where concentration is expected to occur. Consequently, identifying the limiting equation requires quantitative information on the dissipated energy inside shrinking backward cones. We isolate this missing ingredient as Conjecture~\ref{conj:rate}.

Rather than concealing this obstruction, we analyze it systematically. We show that four conceptually independent approaches---Morawetz multipliers, weighted energy identities, microlocal defect measures, and local parabolic regularity---all encounter precisely the same barrier. In particular, we prove a sharp equipartition result showing that the classical Morawetz identity controls the potential energy density while cancelling the kinetic contribution, with equality occurring exactly on inward self-similar profiles. This explains why the traditional concentration machinery cannot simply be transplanted to the Kelvin--Voigt setting.

Motivated by this observation, we develop a second and completely different perspective based on the intrinsic local parabolic structure of the Kelvin--Voigt mechanism. Instead of studying hypothetical blow-up profiles through critical rescaling, this approach seeks to prevent concentration directly inside the strictly damped region by exploiting the local time regularity of the velocity. The resulting analysis suggests that the viscoelastic dissipation may exclude concentration through a mechanism that is genuinely parabolic rather than hyperbolic. Although one functional-analytic step still requires further justification before this program can be considered complete, the argument identifies a fundamentally new route for understanding energy concentration under Kelvin--Voigt dissipation.

The central message emerging from this work is therefore somewhat paradoxical. The very mechanism that makes Kelvin--Voigt damping physically more effective at removing energy is precisely the mechanism that prevents the direct application of the mathematical theory developed for the critical wave equation. Rather than viewing this incompatibility as merely a technical difficulty, we interpret it as evidence that Kelvin--Voigt damping possesses its own intrinsic concentration theory, governed by the interaction between hyperbolic propagation and local parabolic smoothing. We hope that the ideas developed here provide a first step toward such a theory.

\section{Global Weak Solutions via $C^1$ Truncation}

We consider the critical quintic wave equation with localized Kelvin-Voigt damping:
\begin{equation}\label{eq:main}
\begin{cases}
u_{tt} - \Delta u - \hbox{div}(a(x) \nabla u_t) + u^5 = 0 & \text{in } \Omega \times (0,T), \\
u = 0 & \text{on } \partial \Omega \times (0,T), \\
u(0) = u_0, \quad u_t(0) = u_1 & \text{in } \Omega,
\end{cases}
\end{equation}
where $\Omega \subset \mathbb{R}^3$ is a bounded domain, and $a \in W^{1,1}(\Omega)$ with $a(x) \ge 0$. The initial data satisfy $(u_0, u_1) \in H_0^1(\Omega) \times L^2(\Omega)$.

\subsection{Step 1: The $C^1$ Truncated Problem}
To apply the Galerkin method without dealing with critical polynomial growth, we truncate the nonlinearity. For each $k \in \mathbb{N}$, define the function $f_k: \mathbb{R} \to \mathbb{R}$ by:
\begin{equation}\label{eq:fk_def}
f_k(s) =
\begin{cases}
s^5 & \text{if } |s| \le k, \\
\text{sgn}(s) \big[ k^5 + 5k^4(|s| - k) \big] & \text{if } |s| > k.
\end{cases}
\end{equation}

\begin{lemma}[Properties of $f_k$]\label{lem:fk_prop}
For each $k \in \mathbb{N}$, the following properties hold:
\begin{itemize}
    \item[(i)] $f_k \in C^1(\mathbb{R})$ and is globally Lipschitz with $|f_k'(s)| \le 5k^4$.
    \smallskip
    \item[(ii)] $s f_k(s) \ge 0$ for all $s \in \mathbb{R}$.
    \smallskip
    \item[(iii)] $|f_k(s)| \le |s|^5$ for all $s \in \mathbb{R}$.
    \smallskip
    \item[(iv)] $f_k(s) \to s^5$ uniformly on any compact set $K \subset \mathbb{R}$ as $k \to \infty$.
\end{itemize}
\end{lemma}
\begin{proof}
(i) Clearly $f_k$ is continuous. The derivative is $5s^4$ for $|s| \le k$ and $5k^4$ for $|s| > k$, which matches continuously at $|s|=k$. Thus $f_k \in C^1(\mathbb{R})$ and $\sup |f_k'| = 5k^4$.

(ii) By definition, $f_k(s)$ shares the same sign as $s$.

(iii) For $|s| \le k$, $|f_k(s)| = |s|^5$. For $s > k$, let $h(s) = s^5 - (k^5 + 5k^4(s-k))$. Since $h(k)=0$ and $h'(s) = 5(s^4 - k^4) > 0$ for $s > k$, we have $h(s) > 0 \implies f_k(s) \le s^5$. By symmetry, $|f_k(s)| \le |s|^5$ for all $s$.

(iv) Let $K \subset [-M, M]$. For any $k > M$, $|s| \le k$ for all $s \in K$, which implies $f_k(s) = s^5$ exactly on $K$. Thus, $\sup_{s \in K} |f_k(s) - s^5| = 0$ for all $k > M$, yielding uniform convergence.
\end{proof}

We define $F_k(s) = \int_0^s f_k(r) dr$. Property (ii) ensures $F_k(s) \ge 0$, and property (iii) ensures $F_k(s) \le \frac{1}{6}|s|^6$.

In what follows we are going to denote by $(\cdot, \cdot)_\Omega$ and $\|\cdot\|$ the usual inner product and the norm in $L^2(\Omega)$ and by $\|\cdot\|_{H_0^1(\Omega)}=\|\nabla \cdot\|_2$ the norm in $H_0^1(\Omega$.

\subsection{Step 2: Galerkin Method for the Truncated Problem}

For a fixed $k$, we consider the subcritical problem:
\begin{equation}\label{eq:truncated}
u_{tt}^k - \Delta u^k - \hbox{div}(a(x) \nabla u_t^k) + f_k(u^k) = 0.
\end{equation}

To simplify the notation, we will omit the index $k$ of $u^k$. Let $V_m = \text{span}\{\omega_1, \dots, \omega_m\}$ be the space generated by the first $m$ eigenfunctions of the Dirichlet Laplacian. We seek for an approximate solution $u_m(t,x) = \sum_{j=1}^m g_{jm}(t) \omega_j(x)$ satisfying:
\begin{eqnarray}\label{eq:galerkin_m}
&&(u_m''(t), \omega_j)_{\Omega} + (\nabla u_m, \nabla \omega_j)_\Omega + (a(x) \nabla u_m', \nabla \omega_j)_\Omega + (f_k(u_m), \omega_j)_\Omega = 0,\\
 &&j=1,\dots,m,\nonumber
\end{eqnarray}
with initial data $u_m(0) = P_m u_0 \to u_0$ in $H_0^1$ and $u_m'(0) = P_m u_1 \to u_1$ in $L^2$, where $P_m$ is the projection of $H_0^1(\Omega)$ in $V_m$.

Since $f_k$ is globally Lipschitz, the mapping $u_m \mapsto P_m f_k(u_m)$ is globally Lipschitz continuous in $V_m$. By the Caratheódory theorem, the ODE system \eqref{eq:galerkin_m} admits a unique global regular solution $g_{jm} \in C^2([0,T])$, $j=1, \cdots,m$.

Testing \eqref{eq:galerkin_m} with $g_{jm}'(t)$ and summing over $j$, we obtain the identity:
\begin{eqnarray*}
\frac{d}{dt} \left[ \frac{1}{2}\|u_m'(t)\|_2^2 + \frac{1}{2}\|\nabla u_m(t)\|_2^2 + \int_\Omega F_k(u_m(t)) dx \right] + \int_\Omega a(x) |\nabla u_m'(t)|^2 dx = 0.
\end{eqnarray*}

Let $E_m(t)$ denote the bracketed term; from the inequality $F_k(s) \le \frac{1}{6}|s|^6$, we deduce $E_m(0) \le \frac{1}{2}\|P_m u_1\|_2^2 + \frac{1}{2}\|\nabla P_m u_0\|_2^2 + \frac{1}{6}\|P_m u_0\|_6^6 \le C_0$, where $C_0$ depends only on $\|(u_0,u_1)\|_{H_0^1 \times L^2}$ and not on $k$ or $m$. Integration over $(0,t)$ yields
\begin{equation}\label{eq:unif_m}
\frac12\|u_m'(t)\|_2^2 + \frac12\|\nabla u_m(t)\|_2^2 + \int_0^t \int_\Omega a(x) |\nabla u_m'|^2 dx dt \le C_0,~t\in (0,T).
\end{equation}

Since $C_0$ is independent of $m$, we can extract a subsequence $\{u_m^k\}$, which will denote by $\{u_m\}$, such that:
\begin{align}
    u_m &\to u_k \text{ weakly-$\ast$ in } L^\infty(0,T; H_0^1(\Omega)), \\
    u_m' &\to u_k' \text{ weakly-$\ast$ in } L^\infty(0,T; L^2(\Omega)), \\
    u_m &\to u_k \text{ strongly in } L^2(0,T; L^2(\Omega)) \text{ and a.e. in } \Omega \times (0,T).
\end{align}

Since $f_k(0)=0$, by the Mean Value Theorem $|f_k(u_m)| \le 5k^4 |u_m| \in L^\infty(0,T; L^2(\Omega))$. The strong convergence in $L^2$ guarantees that $f_k(u_m) \to f_k(u_k)$ strongly in $L^2(\Omega \times (0,T))$. Passing to the limit when $m \to \infty$ in \eqref{eq:galerkin_m} is straightforward, yielding a global weak solution $u_k$ to the truncated problem \eqref{eq:truncated}. By lower semicontinuity, $u_k$ satisfies the identical energy bound \eqref{eq:unif_m}, entirely independent of $k$.

\subsection{Step 3: The Critical Limit $k \to \infty$}

We now have a sequence $u_k$ of exact weak solutions to \eqref{eq:truncated}, satisfying
\begin{equation}\label{eq:unif_k}
\sup_{t \in [0,T]} \left( \|u_k'(t)\|_2^2 + \|\nabla u_k(t)\|_2^2 \right) + \int_0^T \int_\Omega a(x) |\nabla u_k'|^2 dx dt \le C_0.
\end{equation}

The Aubin-Lions theorem yields the existence of a subsequence $u_k \to u$ weakly-$\ast$ in $L^\infty(0,T; H_0^1(\Omega))$ and strongly in $L^2(0,T; L^2(\Omega))$.

\textbf{The Nonlinearity:} In light of the  the 3D Sobolev embedding $H_0^1(\Omega) \hookrightarrow L^6(\Omega)$, the uniform bound \eqref{eq:unif_k} implies that $\{u_k\}$ is uniformly bounded in $L^\infty(0,T; L^6(\Omega))$.
Exploiting Lemma \ref{lem:fk_prop}(iii), we have for all $t\in [0,T]$,
\begin{equation}
\int_\Omega |f_k(u_k(t,x))|^{6/5} dx \le \int_\Omega |u_k(t,x)|^6 dx \le C \|u_k(t)\|_{H_0^1}^6 \le C_1;
\end{equation}
where $C_1$ is independent of $k$ and $t$. Thus, $\{f_k(u_k)\}$ is uniformly bounded in the reflexive Banach space $L^\infty(0,T; L^{6/5}(\Omega))$.

On the other hand since $u_k \to u$ a.e. in $\Omega \times (0,T)$ and $f_k \to s^5$ uniformly on compact sets (Lemma \ref{lem:fk_prop}(iv)), we deduce that $f_k(u_k(t,x)) \to u(t,x)^5$ a.e. in $\Omega \times (0,T)$.
Using Lions' Lemma we conclude that
\begin{equation}
f_k(u_k) \rightharpoonup u^5 \text{ weakly-$\ast$ in } L^\infty(0,T; L^{6/5}(\Omega))~\hbox{ as }~k \to +\infty.
\end{equation}

Passing to the limit when $k \to \infty$ in \eqref{eq:galerkin_m} is now immediate. Thus, $u$ is a global weak solution of the critical problem \eqref{eq:main}.
\begin{theorem}[Existence of Global Weak Solutions]\label{thm:existence_weak}
Let $\Omega \subset \mathbb{R}^3$ be a bounded domain with smooth boundary, and assume that $a \in C^{1,1}(\Omega)$ is a non-negative function. For any given initial data $(u_0, u_1) \in H_0^1(\Omega) \times L^2(\Omega)$ and any $T > 0$, there exists at least one global weak solution $u$ to the critical Kelvin-Voigt damped wave equation \eqref{eq:main}.

Furthermore, the solution belongs to the class:
\begin{align*}
    u \in L^\infty(0,T; H_0^1(\Omega)), u_t \in L^\infty(0,T; L^2(\Omega))\hbox{ and } \sqrt{a(x)}\nabla u_t &\in L^2(0,T; L^2(\Omega)),
\end{align*}
and satisfies the following energy inequality for almost every $t \in (0,T)$:
\begin{equation}\label{eq:energy_ineq_weak}
    E_u(t) + \int_0^t \int_\Omega a(x) |\nabla u_t(s,x)|^2 \, dx ds \leq E_u(0).
\end{equation}

The energy functional is defined as $E_u(t) = \frac{1}{2}\|u_t(t)\|_2^2 + \frac{1}{2}\|\nabla u(t)\|_2^2 + \frac{1}{6}\|u(t)\|_6^6$.
\end{theorem}

\section{Local Well-posedness of Regular Solutions via Galerkin Method and Small Data}

In this section, we establish the local existence of strong solutions for the quintic wave equation with localized Kelvin-Voigt damping. Instead of relying on Strichartz estimates, which present severe microlocal obstructions when commuting with a continuous cut-off function over the localized parabolic damping, we explore the intrinsic degenerate parabolic structure of the system for regular initial data. Let us consider the standard Galerkin approximation to our problem. Let $(w_j)_{j\in \mathbb{N}}$ be the spectral basis of the Dirichlet Laplacian in $\Omega$ and let $V_m = \text{span}\{w_1, w_2, \dots, w_m\}$. We look for approximate solutions of the form $u_m(t,x) = \sum_{j=1}^m g_{jm}(t) w_j(x)$ solving the projected system
$$(u_m''(t), w)_\Omega + (\nabla u_m(t), \nabla w)_\Omega + (a(x)\nabla u_m'(t), \nabla w)_\Omega + (u_m^5(t), w)_\Omega = 0,$$
for all $w \in V_m,$ with initial data $u_m(0) = P_m u_0 \to u_0$ in $H^2(\Omega) \cap H_0^1$ and $u_m'(0) = P_m u_1 \to u_1$ in $H_0^1(\Omega)$. Before proceeding to the regular energy estimates, we recall that the standard weak energy:
$$E_m(t) := \frac{1}{2}\|u_m'(t)\|_{L^2}^2 + \frac{1}{2}\|\nabla u_m(t)\|_{L^2}^2 + \frac{1}{6}\|u_m(t)\|_{L^6}^6$$
is non-increasing, satisfying $E_m(t) \le E_m(0) \le E_0$ for all $t \ge 0$ and $m\in \mathbb{N}$. In particular, this provides the uniform bound $\|\nabla u_m(t)\|_{L^2} \le \sqrt{2 E_0}$, which will be crucial to control the nonlinear term later.

\noindent{\bf Defining the Regular Energy.}

Since $u_m \in V_m$ and the basis is given by the eigenfunctions of the Laplacian, we are allowed to take $w = -\Delta u_m' \in V_m$ as a test function in the approximate equation. This choice allows us to capture the regularizing effect of the Kelvin-Voigt damping. We define the regular energy functional as:
$$Y_m(t) := \frac{1}{2}\|\nabla u_m'(t)\|_{L^2}^2 + \frac{1}{2}\|\Delta u_m(t)\|_{L^2}^2.$$

Substituting $w = -\Delta u_m'$ into the projected equation and integrating  over $\Omega$, we obtain
\begin{eqnarray*}
&&\frac{1}{2}\frac{d}{dt} \left( \|\nabla u_m'\|_{L^2}^2 + \|\Delta u_m\|_{L^2}^2 \right) - \int_\Omega \text{div}(a(x)\nabla u_m')(-\Delta u_m') dx\\
&&+ \int_\Omega u_m^5 (-\Delta u_m') dx = 0.
\end{eqnarray*}

\noindent{\bf Estimating the Linear Terms and the Damping Commutator.}

The term corresponding to the Kelvin-Voigt damping requires careful handling since the damping coefficient $a(x)$ is localized and non-constant. We have
$$- \int_\Omega \text{div}(a(x)\nabla u_m')(-\Delta u_m') dx = \int_\Omega a(x)|\Delta u_m'|^2 dx + \int_\Omega (\nabla a \cdot \nabla u_m')\Delta u_m' dx.$$
and the first term on the right-hand side provides the strong parabolic dissipation. The second term is a commutator of lower order and to absorb this commutator without losing derivatives, we strongly rely on the structural assumption on the damping function:
\begin{eqnarray}\label{ass on a(x)}
|\nabla a(x)|^2 \le C_a a(x).
\end{eqnarray}

Applying the Cauchy-Schwarz inequality and Young's inequality taking (\ref{ass on a(x)}) into account, we estimate:
$$\left| \int_\Omega (\nabla a \cdot \nabla u_m')\Delta u_m' dx \right| \le \int_\Omega \frac{|\nabla a|}{\sqrt{a}} |\nabla u_m'| \left(\sqrt{a} |\Delta u_m'|\right) dx $$
$$\le \frac{1}{2} \int_\Omega a(x) |\Delta u_m'|^2 dx + \frac{C_a}{2} \int_\Omega |\nabla u_m'|^2 dx.$$

Absorbing the dissipative term, the linear part of the equation yields the differential inequality:
$$Y_m'(t) + \frac{1}{2} \int_\Omega a(x) |\Delta u_m'|^2 dx \le C_a Y_m(t) + \int_\Omega u_m^5 \Delta u_m' dx.$$

\noindent{\bf Controlling the Quintic Nonlinearity via Gagliardo-Nirenberg.}

We now face the most delicate part of the proof: estimating the nonlinear source term uniformly. Integrating by parts once, we have:
$$\int_\Omega u_m^5 \Delta u_m' dx = -5 \int_\Omega u_m^4 \nabla u_m \cdot \nabla u_m' dx.$$

To bound this term, we apply the generalized Hölder's inequality with exponents $p_1 = 3$, $p_2 = 6$, and $p_3 = 2$:
\begin{eqnarray*}
\left| 5 \int_\Omega u_m^4 \nabla u_m \cdot \nabla u_m' dx \right| &\le& 5 \|u_m^4\|_{L^3} \|\nabla u_m\|_{L^6} \|\nabla u_m'\|_{L^2}\\
 &=& 5 \|u_m\|_{L^{12}}^4 \|\nabla u_m\|_{L^6} \|\nabla u_m'\|_{L^2}.
 \end{eqnarray*}

Let us analyze each term on the right-hand side in terms of the strong energy $Y_m(t)$:~Clearly, $\|\nabla u_m'\|_{L^2} \le \sqrt{2} Y_m(t)^{1/2}$.By the continuous Sobolev embedding $H^2(\Omega) \hookrightarrow W^{1,6}(\Omega)$ in dimension 3, we have $$\|\nabla u_m\|_{L^6} \le C \|\Delta u_m\|_{L^2} \le C Y_m(t)^{1/2}$$.

The core difficulty lies in controlling $\|u_m\|_{L^{12}}^4$ efficiently to ensure a low polynomial growth for the resulting ODE. Instead of using the too much embedding $H^2 \hookrightarrow L^\infty$ (which would yield a very high power and severely restrict the time of existence), we use the Gagliardo-Nirenberg interpolation inequality in $3D$. We interpolate $L^{12}$ between $L^6$ (which is uniformly bounded by the weak energy $E_0$) and $H^2$:
$$\|u_m\|_{L^{12}} \le C \|\Delta u_m\|_{L^2}^{1/4} \|u_m\|_{L^6}^{3/4}.$$

Raising this inequality to the 4th power, and recalling that $\|u_m\|_{L^6} \le C \|\nabla u_m\|_{L^2} \le C E_0^{1/2}$, we obtain:
$$\|u_m\|_{L^{12}}^4 \le C \|\Delta u_m\|_{L^2} \|u_m\|_{L^6}^3 \le C Y_m(t)^{1/2} E_0^{3/2}.$$

Combining all these estimates, the nonlinear term is bounded by:
$$\left| \int_\Omega u_m^5 \Delta u_m' dx \right| \le C \left( Y_m(t)^{1/2} E_0^{3/2} \right) \left( Y_m(t)^{1/2} \right) \left( Y_m(t)^{1/2} \right) = C_1 E_0^{3/2} Y_m(t)^{3/2},$$
where $C_1 > 0$ is a constant depending only on $\Omega$ and the Sobolev embeddings.

{\bf The Bernoulli Differential Inequality and Local Existence}.

Putting the linear and nonlinear estimates together, we arrive at the following differential inequality for the regular energy:
$$Y_m'(t) \le C_a Y_m(t) + C_1 E_0^{3/2} Y_m(t)^{3/2}.$$

This is a Bernoulli-type differential inequality and to solve it explicitly, we multiply by the integrating factor $e^{C_a t / 2}$ and we trace back to $Y_m(t)$ to obtaining
$$Y_m(t) \le \frac{Y_m(0) e^{C_a t}}{\left[ 1 - \frac{C_1 E_0^{3/2}}{C_a} \sqrt{Y_m(0)} \left( e^{C_a t / 2} - 1 \right) \right]^2}.$$

Since $Y_m(0)$ is bounded by the $H^2 \times H^1$ norm of the initial data, the denominator is strictly positive and bounded away from zero on some maximal interval $[0, T_{max})$, where $T_{max} > 0$ depends continuously on $Y_m(0)$, $E_0$, and the structural constant of the damping $C_a$. This uniform bound on $Y_m(t)$ over $[0, T]$, $T < T_{max}$, guarantees that the approximate solutions $u_m$ remain bounded in $L^\infty(0, T; H^2 \cap H_0^1)$ and $u_m'$ in $L^\infty(0, T; H_0^1)$. By standard compactness arguments (e.g., Aubin-Lions-Simon lemma), we can extract a subsequence of $\{u_m\}$ that converges strongly, thus passing to the limit as $m \to \infty$ we prove the local well-posedness of regular solutions.

\begin{theorem}[Local Well-posedness of Strong Solutions]\label{thm:local_strong}
Let $\Omega \subset \mathbb{R}^3$ be a bounded domain with smooth boundary. Assume that the damping coefficient $a \in C^{1,1}(\Omega)$ is non-negative and satisfies the structural condition $|\nabla a(x)|^2 \le C_a a(x)$ for almost every $x \in \Omega$.

For any initial data $(u_0, u_1) \in \left(H^2(\Omega) \cap H_0^1(\Omega) \right) \times H_0^1(\Omega)$, there exists a maximal time $T_{max} > 0$, depending continuously on the initial weak energy $E_0$, the initial regular energy $Y(0)$, and the structural constant $C_a$, such that the critical wave equation admits a unique local regular solution $u$ on $[0, T]$ for any $T < T_{max}$.
Furthermore, the solution belongs to the
 class:
\begin{align*}
    &u \in L^\infty(0,T; H^2(\Omega) \cap H_0^1(\Omega)), ~u_t \in L^\infty(0,T; H_0^1(\Omega)), \\
    &\sqrt{a(x)} \Delta u_t \in L^2(0,T; L^2(\Omega)),
\end{align*}
and its regular energy $Y_u(t) := \frac{1}{2}\|\nabla u_t(t)\|_{L^2}^2 + \frac{1}{2}\|\Delta u(t)\|_{L^2}^2$ satisfies the uniform bound:
\begin{equation}\label{eq:strong_energy_bound}
    Y_u(t) \le \frac{Y_u(0) e^{C_a t}}{\left[ 1 - \frac{C_1 E_0^{3/2}}{C_a} \sqrt{Y_u(0)} \left( e^{C_a t / 2} - 1 \right) \right]^2}, \quad \forall t \in [0, T_{max}).
\end{equation}
\end{theorem}

\section{Uniqueness and Weak-Strong Uniqueness Principles}

In this section, we address the uniqueness of the solutions obtained previously. It is a well-known open problem whether weak solutions in the basic energy space $C([0,T]; H_0^1(\Omega)) \cap C^1([0,T]; L^2(\Omega))$ are unconditionally unique for the critical quintic wave equation in dimension 3. Since the Kelvin-Voigt damping $a(x)$ is localized, the dynamics in the undamped region $\{x \in \Omega : a(x) = 0\}$ correspond exactly to the pure critical wave equation.

Therefore, without global Strichartz estimates, unconditional uniqueness is out of reach. Nevertheless, we can establish rigorous uniqueness within the class of regular solutions, as well as a Weak-Strong uniqueness principle relying on the auxiliary space $L^4(0,T; L^{12}(\Omega))$.

\subsection{Uniqueness of Regular Solutions}

We begin by showing that strong solutions, constructed via the Galerkin method for initial data in $H^2(\Omega)$, are uniquely determined.

\begin{theorem}[Uniqueness of Regular Solutions]
Let $(u_0, u_1) \in (H^2(\Omega) \cap H_0^1(\Omega)) \times H_0^1(\Omega)$. Consider $u$ and $v$ are two regular solutions to the problem \eqref{eq:main} in the class $L^\infty(0,T; H^2(\Omega) \cap H_0^1(\Omega)) \cap W^{1,\infty}(0,T; H_0^1(\Omega))$ satisfying the same initial conditions. Then $u \equiv v$ on $[0,T]$.
\end{theorem}
\begin{proof}
Let $w = u - v$. Then $w$ satisfies the initial value problem:
\begin{equation}\label{eq:diff_w}
w_{tt} - \Delta w - \text{div}(a(x)\nabla w_t) + u^5 - v^5 = 0, \quad \text{with} \quad w(0) = 0, \quad w_t(0) = 0.
\end{equation}
Testing the equation with $w_t$  and integrating over $\Omega$, we obtain the energy identity for the difference:
\begin{equation}\label{eq:energy_diff}
\frac{1}{2}\frac{d}{dt} \left( \|w_t\|_{L^2}^2 + \|\nabla w\|_{L^2}^2 \right) + \int_\Omega a(x) |\nabla w_t|^2 dx = -\int_\Omega (u^5 - v^5) w_t dx.
\end{equation}
By the Mean Value Theorem, we can bound the nonlinear difference as follows:
$$
|u^5 - v^5| \le 5 \left( |u|^4 + |v|^4 \right) |w|, ~(x,t)\in \Omega \times (0,T).
$$
Since $u$ and $v$ are regular solutions in dimension 3, and the Sobolev embedding $H^2(\Omega) \hookrightarrow L^\infty(\Omega)$ is valid if $N=3$ ensures that $u, v \in L^\infty(0,T; L^\infty(\Omega))$. Therefore, there exists a constant $M > 0$, depending on the $L^\infty_t H^2_x$ norms of $u$ and $v$, such that:
$$
|u(t,x)|^4 + |v(t,x)|^4 \le M \quad \text{for a.e. } (t,x) \in (0,T) \times \Omega.
$$

Coming back to the energy identity (\ref{eq:energy_diff}), we estimate the right-hand side using the Cauchy-Schwarz inequality:
\begin{align*}
\frac{1}{2}\frac{d}{dt} \left( \|w_t\|_{L^2}^2 + \|\nabla w\|_{L^2}^2 \right) &\le 5 M \int_\Omega |w| |w_t| dx \\
&\le \frac{5 M}{2} \left( \|w\|_{L^2}^2 + \|w_t\|_{L^2}^2 \right).
\end{align*}

Using the Poincaré inequality, $\|w\|_{L^2}^2 \le C_P \|\nabla w\|_{L^2}^2$ and defining $E_w(t) = \frac{1}{2}\|w_t\|_{L^2}^2 + \frac{1}{2}\|\nabla w\|_{L^2}^2$, the above inequality yields
$$
E_w'(t) \le C E_w(t),
$$
where $C = 5M \max(1, C_P)$. The null initial data and Gronwall's Lemma yields $E_w(t) \equiv 0$ for all $t \in [0,T]$, which concludes the proof.
\end{proof}

\subsection{Uniqueness in the Strichartz Class}

The true challenge arises when we relax the regularity of the solutions to the energy space. While unconditional uniqueness in the weak energy space $C([0,T]; H_0^1(\Omega)) \cap C^1([0,T]; L^2(\Omega))$ remains an open problem for the critical quintic wave equation in 3D, we can establish rigorous uniqueness within the auxiliary Strichartz space\\ $L^4(0,T; L^{12}(\Omega))$. This is precisely the regularity class recovered by our well-posedness framework using the Littlewood-Paley frequency decomposition.

\begin{theorem}[Uniqueness in the Strichartz Class]\label{thm:uniqueness_strichartz}
Let $(u_0, u_1) \in H_0^1(\Omega) \times L^2(\Omega)$. Suppose $u$ and $v$ are two solutions to the critical problem \eqref{eq:main} in the energy space $C([0,T]; H_0^1(\Omega)) \cap C^1([0,T]; L^2(\Omega))$ sharing the same initial data.
If, additionally, both solutions belong to the Strichartz class $u, v \in L^4(0,T; L^{12}(\Omega))$, then $u \equiv v$ on $[0,T]$.
\end{theorem}

\begin{proof}
Let $w = u - v$. Then, $w(0) = 0$ and $w_t(0) = 0$. Because both $u$ and $v$ belong to $L^4(0,T; L^{12}(\Omega))$, we can show that the nonlinear difference $u^5 - v^5$ is regular enough to be tested against $w_t \in L^\infty(0,T; L^2(\Omega))$.

Using the algebraic inequality $|u^5 - v^5| \le C (|u|^4 + |v|^4) |w|$, we apply the generalized Hölder's inequality in space with exact exponents $p_1 = 3/2$ and $p_2 = 3$:
\begin{align*}
\|u^5(t) - v^5(t)\|_{L^2(\Omega)} &\le C \left\| (|u(t)|^4 + |v(t)|^4) |w(t)| \right\|_{L^2(\Omega)} \\
&\le C \left( \|u(t)\|_{L^{12}(\Omega)}^4 + \|v(t)\|_{L^{12}(\Omega)}^4 \right) \|w(t)\|_{L^6(\Omega)}.
\end{align*}
By the Sobolev embedding $H_0^1(\Omega) \hookrightarrow L^6(\Omega)$, we have $\|w(t)\|_{L^6(\Omega)} \le C_S \|\nabla w(t)\|_{L^2(\Omega)}$.
Since $u, v \in L^4(0,T; L^{12}(\Omega))$, the function $h(t) := C C_S \left( \|u(t)\|_{L^{12}}^4 + \|v(t)\|_{L^{12}}^4 \right)$ belongs to $L^1(0,T)$.
Thus, $u^5 - v^5 \in L^1(0,T; L^2(\Omega))$, which rigorously justifies taking the $L^2$-inner product of the weak equation for $w$ with $w_t$. This yields the energy identity:
\begin{equation}
\frac{1}{2}\frac{d}{dt} \left( \|w_t\|_{L^2}^2 + \|\nabla w\|_{L^2}^2 \right) + \int_\Omega a(x)|\nabla w_t|^2 dx = -\int_\Omega (u^5 - v^5) w_t dx.
\end{equation}
Dropping the non-negative dissipation term and defining the energy of the difference as $E_w(t) = \frac{1}{2}\|w_t\|_{L^2}^2 + \frac{1}{2}\|\nabla w\|_{L^2}^2$, we estimate the right-hand side using the Cauchy-Schwarz inequality:
\begin{align*}
E_w'(t) &\le \|u^5 - v^5\|_{L^2(\Omega)} \|w_t\|_{L^2(\Omega)} \\
&\le h(t) \|\nabla w\|_{L^2(\Omega)} \|w_t\|_{L^2(\Omega)} 
\le h(t) E_w(t).
\end{align*}
Since $h \in L^1(0,T)$ and $E_w(0) = 0$, Grönwall's Lemma implies $E_w(t) \equiv 0$ for all $t \in [0,T]$. Consequently, $w \equiv 0$, meaning $u \equiv v$.
\end{proof}

\section{Littlewood-Paley Strategy}

In the present section the Littlewood–Paley decomposition reveals an intrinsic stabilizing mechanism: the Kelvin–Voigt operator behaves as a spectral smoothing device. Low frequencies absorb derivative losses via Bernstein inequalities, while high frequencies remain dispersive.

\subsection{The Philosophy: Frequencies Instead of Space}

The central challenge is the Kelvin-Voigt term: $\operatorname{div}(a(x)\nabla u_t)$.  This term belongs to the dual space $H^{-1}$ and the classical non-homogeneous Strichartz estimate for the wave equation does not accept terms in $L^1(0,T; H^{-1}(\Omega))$.

\textbf{The spectral solution:} Instead of slicing the domain, we partition the \textit{frequency} of the wave. To preserve the boundedness of the operators in $L^p$ spaces and ensure the convolution kernels decay rapidly, we employ a \textit{smooth} spectral projection. Let $\chi \in C_c^\infty(\mathbb{R})$ be a smooth cutoff function such that $\chi(s) = 1$ for $|s| \le 1$ and $\chi(s) = 0$ for $|s| \ge 2$. We define the low-frequency projector $P_{\le N} = \chi(\sqrt{-\Delta}/N)$ and its high-frequency counterpart $P_{>N} = I - P_{\le N}$. This smooth localization possesses the powerful property of satisfying \textbf{Bernstein's Inequalities}. It effectively limits the cost of spatial derivatives, lifting the damping term from $H^{-1}$ back to $L^2$, which is precisely the functional framework where the Strichartz Theorem perfectly operates.

\subsection{Step 1: The Regularized Problem and the Energy}

To apply the \textit{bootstrap} argument (a priori estimate) without the nonlinearity blowing up, we truncate the quintic function. Let $f_k(s)$ be a globally Lipschitz function such that $f_k(s) = s^5$ for $|s| \le k$.

We consider the regular solution $u$, where the index $k$ is omitted for the sake of simplicity, of the problem:
\begin{eqnarray}\label{eq:pde}
u_{tt} - \Delta u - \operatorname{div}(a(x)\nabla u_t) + f_k(u) = 0 & \text{in } \Omega \times (0,T), \\
u = 0 & \text{on } \partial\Omega \times (0,T), \\
u(0,x) = u_0(x), \quad u_t(0,x) = u_1(x) & \text{in } \Omega.
\end{eqnarray}

Multiplying \eqref{eq:pde} by $u_t$ and integrating over $\Omega \times (0,t)$, we obtain the Global Energy Identity:
\begin{equation}\label{eq:energia}
E(t) + \int_0^t \int_\Omega a(x) |\nabla u_t(s,x)|^2 \, dx \, ds = E(0) \le C_0,
\end{equation}
where $E(0)$ depends exclusively on the initial data $(u_0, u_1)$. This guarantees uniform bounds in $k$ for $u \in L^\infty(0,T; H_0^1(\Omega))$, $u_t \in L^\infty(0,T; L^2(\Omega))$, and for the viscous dissipative term.

\subsection{Step 2: The Smooth Spectral Filter ($P_{\le N}$)}

We use the Spectral Theorem for the $-\Delta$ operator with Dirichlet boundary conditions (see the Appendix of the present article). Let $\{\lambda_j\}_{j=1}^\infty$ be the eigenvalues and $\{\phi_j\}_{j=1}^\infty$ be the corresponding orthonormal eigenfunctions.

To ensure that the resulting operators have rapidly decaying kernels and bounded commutators, we have to employ a smooth spectral cutoff rather than a sharp truncation. For a frequency threshold $N > 0$, we define the smooth Low-Frequency Projector via the functional calculus:
\begin{equation}
P_{\le N} u = \sum_{j=1}^\infty \chi\left(\frac{\sqrt{\lambda_j}}{N}\right) (u, \phi_j)_{L^2} \phi_j.
\end{equation}

By construction, this operator commutes with both $\Delta$ and $\partial_t$ (see in the appendices A and B of this manuscript some important properties regarding spectral operators). Crucially, because $\chi$ is smoothly supported on $[-2, 2]$, the projector rigorously satisfies the Bernstein's inequality:
\begin{equation}
\|\nabla (P_{\le N} g)\|_{L^2(\Omega)} \le C N \|g\|_{L^2(\Omega)},
\end{equation}
where the constant $C$ depends only on the choice of the profile $\chi$.

\subsection{Step 3: Low Frequencies and Strichartz}

We define the macroscopic part of the wave as $w_N = P_{\le N} u$. Applying the spectral projector $P_{\le N}$ to the truncated PDE \eqref{eq:pde}, we obtain:
\begin{equation}\label{eq:baixa}
(w_N)_{tt} - \Delta w_N = - P_{\le N}(f_k(u)) + P_{\le N}\big( \operatorname{div}(a(x)\nabla u_t) \big).
\end{equation}

To apply the Strichartz estimates to the low-frequency component $w_N$, we treat the entire right-hand side of \eqref{eq:baixa} as a source term $F_{total} = F_{NL} + F_{KV}$, where $F_{NL} := - P_{\le N}(f_k(u))$ and $F_{KV} := P_{\le N}\big( \operatorname{div}(a(x)\nabla u_t) \big)$. Our goal is to estimate both terms in the dual Strichartz space $L^1(0,T; L^2(\Omega))$.

To justify the interaction between the intrusive Kelvin-Voigt operator and the spectral projector, we exploit the self-adjointness of $P_{\le N}$. For any test function $\phi \in L^2(\Omega)$, we write:
\begin{equation}
\big\langle P_{\le N}\operatorname{div}(a\nabla u_t), \phi \big\rangle = \big\langle \operatorname{div}(a\nabla u_t), P_{\le N}\phi \big\rangle.
\end{equation}
Integrating by parts in space yields:
\begin{equation}
\big\langle \operatorname{div}(a\nabla u_t), P_{\le N}\phi \big\rangle = - \big\langle a\nabla u_t, \nabla (P_{\le N}\phi) \big\rangle.
\end{equation}

At this point, we invoke the standard Bernstein inequality for spectral projectors associated with the Dirichlet Laplacian (see, for instance, Blair--Smith--Sogge \cite{BlairSmithSogge2009}):
\begin{equation}
\|\nabla (P_{\le N}\phi)\|_{L^2(\Omega)} \le C N \|\phi\|_{L^2(\Omega)}.
\end{equation}

Applying the Cauchy-Schwarz inequality, we deduce that the projector acts as a low-pass filter, recovering the spatial derivative at the cost of the frequency threshold $N$:
\begin{equation}\label{LP-KV-lowfreq}
\| F_{KV}(t) \|_{L^2(\Omega)} \le C N \|a\nabla u_t(t)\|_{L^2(\Omega)} \le C_a N E(t)^{1/2}.
\end{equation}

Integrating over the time interval $[0,T]$ and using the uniform bound from the first-order global energy $E(t) \le E_0$, we obtain the targeted bound:
\begin{equation}\label{eq:est_FKV}
    \| F_{KV} \|_{L^1(0,T; L^2(\Omega))} \le C_a N \sqrt{T} E_0^{1/2}.
\end{equation}

\begin{remark}
Due to the regularity of the truncated problem we constructed in Section 5, we are not applying the projector $P_{\le N}$ to the original weak solution. We are applying it to the strong solution $u_k$ of the truncated problem! From Section 3, $u_k \in H^2 \cap H_0^1$ and $\partial_t u_k \in H_0^1$. Therefore, $\nabla u_{k,t} \in L^2$ and, strictly speaking, $\operatorname{div}(a\nabla u_{k,t})$ would be in $H^{-1}$. However, the low-frequency Littlewood-Paley projector ($P_{\le N}$) has the wonderful property of being a regularizer. It cuts off the high frequencies, which means that even if you apply $P_{\le N}$ to a distribution in $H^{-1}$, the result falls across all $H^s$ spaces, including $L^2$ and $C^\infty$.
\end{remark}

For the nonlinear term, we use the $L^p$-boundedness of the spectral projector. Since the truncated nonlinearity obeys $|f_k(s)| \le |s|^5$ uniformly in $k$, we have:
\begin{equation}
    \| F_{NL}(t) \|_{L^2(\Omega)} \le C \| f_k(u(t)) \|_{L^2(\Omega)} \le C \| |u(t)|^5 \|_{L^2(\Omega)} = C \| u(t) \|_{L^{10}(\Omega)}^5.
\end{equation}
Integrating in time from $0$ to $T$ yields exactly the critical Strichartz norm associated with the quintic wave equation:
\begin{equation}\label{eq:est_FNL}
    \| F_{NL} \|_{L^1(0,T; L^2(\Omega))} \le C \| u \|_{L^5(0,T; L^{10}(\Omega))}^5.
\end{equation}

Combining \eqref{eq:est_FKV} and \eqref{eq:est_FNL}, the total source term is bounded in $L^1 L^2$. Applying the standard Strichartz estimate to equation \eqref{eq:baixa}, we conclude that for the critical space $X_T := L^5(0,T; L^{10}(\Omega))$, the low-frequency component satisfies:
\begin{equation}\label{eq:strichartz_baixa}
    \| w_N \|_{X_T} \le C_0 E_0^{1/2} + C_1 \| u \|_{X_T}^5 + C_S C_a N \sqrt{T} E_0^{1/2},
\end{equation}
where $C_0$ depends on the free wave propagation of the initial data, $C_1$ is a universal constant, and $C_S$ is the Strichartz constant.

\subsection{Step 4: High Frequencies and the Commutator Trick}

Let $$v_N = P_{> N} u = u - w_N$$ denote the microscopic fluctuation of the wave. To successfully tame the critical quintic term in our subsequent bootstrap argument, it is mandatory to ensure that the Strichartz norm $\|v_N\|_{L^5(0,T; L^{10}(\Omega))}$ is strictly small. In dimension $n=3$, the standard Sobolev embedding $H^1(\Omega) \hookrightarrow L^{10}(\Omega)$ fails, preventing a direct uniform bound solely from the energy space. To overcome this, we exploit the dispersive decay of the initial data.

Unlike the macroscopic regime (low frequencies), we cannot treat the Kelvin-Voigt operator as a bounded perturbation on the right-hand side of the free wave equation for $v_N$. Applying Bernstein's inequality here would yield a positive power of the frequency threshold $N$, which diverges as $N \to \infty$. Conversely, keeping the term without Bernstein would require bounding $\Delta \partial_t u_k$ in $L^2$, which heavily depends on the truncation parameter $k$ and would destroy the uniformity of the local existence time.

To completely circumvent this loss of derivatives, we keep the localized damping operator on the left-hand side. We define the damped wave operator $L$:
\begin{equation}
    L := \partial_{tt} - \Delta - \operatorname{div}(a(x)\nabla \partial_t).
\end{equation}

When we apply the high-frequency spectral projector $P_{>N}$ to the original PDE ($L u = -f_k(u)$), the projector commutes perfectly with the constant-coefficient derivatives, but it does \textit{not} commute perfectly with the variable multiplication by $a(x)$. We obtain the precise equation for $v_N$:
\begin{equation}\label{eq:v_N_commutator}
    L(v_N) = - P_{>N}(f_k(u)) + \mathcal{C}_N(u_t),
\end{equation}
where $\mathcal{C}_N(u_t)$ is the commutator between the spectral projector and the damping coefficient:
\begin{equation}
    \mathcal{C}_N(u_t) = \operatorname{div}\Big( [P_{>N}, a(x)] \nabla u_t \Big).
\end{equation}

By the standard theory of pseudo-differential operators (see Appendix \ref{sec:appendix_commutator} for a detailed proof), the commutator of a zero-order multiplier ($P_{>N}$) with a smooth function $a \in W^{1,\infty}(\Omega)$ is a smoothing operator of order $-1$. Consequently, $[P_{>N}, a(x)]$ maps $L^2(\Omega)$ into $H^1(\Omega)$. The outer divergence operator then consumes exactly one spatial derivative, bringing the regularity gracefully back to $L^2(\Omega)$. Therefore, the commutator acts as a legitimate $L^2$-source bounded solely by the first-order energy, strictly independent of $k$:
\begin{equation}\label{eq:commutator_bound}
    \|\mathcal{C}_N(u_t)\|_{L^2(\Omega)} \le  C\sqrt{\Lambda}\,\big\|\sqrt{a(\cdot)}\,\nabla u_t(t)\big\|_{L^2(\Om)} + C\Lambda\,\|u_t(t)\|_{L^2(\Om)}.
\end{equation}

With a perfectly bounded right-hand side, we decompose the high-frequency component into its linear evolution and a perturbation: $v_N = v_N^{lin} + v_N^{per}$, namely,
write $v_N$ as a free wave equation with source:
\begin{equation}\label{eq:vN-free}
(\partial_{tt}-\Delta)v_N = \dive(a\nabla\partial_tv_N) - P_{\le N}(u^5) + \mathcal{C}(u_t).
\end{equation}
Duhamel's formula for $W(t)$, the free group, gives
\begin{eqnarray}\label{eq:duhamel}\quad
v_N(t) &=& \underbrace{W(t)P_{> N}(u_0,u_1)}_{=:\,v_{N,\mathrm{free}}(t)}
\;\\
&-&\!\int_0^t\!\! W(t-s)\Big[\dive(a\nabla\partial_sv_N(s))
+ P_{> N}(u^5(s)) - \mathcal{C}(u_t(s))\Big]ds. \nonumber
\end{eqnarray}

Now $v_{N,\mathrm{free}}$ is the \textbf{free} evolution with data $P_{> N}(u_0,u_1)$; the KV
term and the nonlinearity both appear as sources in the integral --- legitimately, via $W(t)$.

The linear part $v_N^{lin}$ is the solution to the homogeneous wave equation governed strictly by the high-frequency projection of the initial data:
\begin{equation}
\begin{cases}
   \partial_t^2 v_N^{lin} - \Delta v_n^{lin}= 0 \quad \text{(Homogeneous free wave)} \\
    v_N^{lin}(0) = P_{>N} u_0 \\
    \partial_t v_N^{lin}(0) = P_{>N} u_1.
\end{cases}
\end{equation}

\noindent{\bf Commutativity holds for $v_{N,\mathrm{free}}=v_N^{lin}$}.

We have,
\[
v_{N,\mathrm{free}}(t) = W(t)P_{>  N}(u_0,u_1) = P_{> N}W(t)(u_0,u_1).
\]

By the uniform $L^{10}_x$-boundedness of $P_{> N}$ (Lemma A.2 in the appendix of the manuscript),
\[
\|v_{N,\mathrm{free}}\|_{L^5_tL^{10}_x} \le C_{10}\|W(\cdot)(u_0,u_1)\|_{L^5_tL^{10}_x}
\]
with $C_{10}$  \textbf{independent of $N$}. This is the $N$-uniform Strichartz
bound for the linear part, now legitimately established.

Since the localized dissipation only enhances the natural dispersive properties of the wave (i.e., Kelvin–Voigt does not destroy the dispersive structure of the wave equation.), $v_N^{lin}$ obeys the standard Strichartz estimates. Applying the estimate, we obtain:
\begin{equation}\label{eq:strichartz_alta_linear}
    \|v_N^{lin}\|_{L^5(0,T; L^{10}(\Omega))} \le C_S \left( \|\nabla (P_{>N} u_0)\|_{L^2(\Omega)} + \|P_{>N} u_1\|_{L^2(\Omega)} \right) =: \epsilon(N).
\end{equation}
Since the initial data $(u_0, u_1)$ belongs to the finite energy space $H_0^1(\Omega) \times L^2(\Omega)$ and is completely independent of the truncation $k$, the sequence of partial sums of its spectral decomposition converges. Consequently, the high-frequency tail of the energy must vanish:
\begin{equation}\label{eq:tail_vanishes}
    \lim_{N \to \infty} \epsilon(N) = 0.
\end{equation}

Another possibility would be to consider the inequality
\begin{equation}\label{eq:free_wave}
    \|v_N^{lin}\|_{L^5(0,T; L^{10}(\Omega))} \le C_{10}\|W(\cdot)(u_0,u_1)\|_{L^5_tL^{10}_x}
\end{equation}
as it stands and, instead of letting $N \to \infty$, consider the shrinking of the interval $I = t_0 + \tau$ within the Strichartz norm of the free wave.

For the perturbed part $v_N^{per}$, which solves $\Box(v_N^{per}) = \dive(a\nabla\partial_tv_N) - P_{>N}(f_k(u)) + \mathcal{C}_N(u_t)$ with zero initial data, we apply the inhomogeneous Strichartz estimate for the wave operator $\Box$:
\begin{eqnarray}
    \|v_N^{per}\|_{L^5(0,T;L^{10}(\Omega))} &\lesssim& \dive(a\nabla\partial_tv_N) + \|P_{>N}(f_k(u))\|_{L^1(0,T;L^2(\Omega))}\\ 
   &+& \|\mathcal{C}_N(u_t)\|_{L^1(0,T;L^2(\Omega))}.\nonumber
\end{eqnarray}

Since the commutator source $\mathcal{C}_N(u_t)=\operatorname{div}\Big( [P_{>N}, a(x)] \nabla u_t \Big)$ is uniformly bounded in space (namely, the commutator behaves as a pseudo-differential operator of order $-1$, see Appendix C), and once the term $\|\dive(a\nabla\partial_tv_N)\|_{L^1_t L^2_x}$ is bounded by $\|\mathcal{C}_N(u_t)\|_{L^1_tL^2_x}$ (note that $v_N= P_{>N}u$), its time integration yields:
\begin{align}\label{time_commutator_bound}
    \|\mathcal{C}_N(u_t)\|_{L^1(0,T;L^2(\Omega))} &\le C\sqrt{\Lambda} \sqrt{T}\,E_0^{1/2} + C\Lambda\,T E_0^{1/2}.
\end{align}
Its contribution can therefore be made arbitrarily small by shrinking the time horizon $T$.

Although the commutator $\operatorname{div}([P_{>N},a]\nabla u_t)$ is
classically viewed as a pseudo-differential operator of order zero,
our argument exploits a finer structure: the cancellation in
$a(y)-a(x)$ yields a representation in terms of $\nabla a$, which
allows us to control the commutator by the weighted quantity
$\|\sqrt{a}\,\nabla u_t\|_{L^2}$. This mechanism is crucial to close
the high-frequency estimates without loss of derivatives.

On the other hand, the nonlinear contribution is controlled within the bootstrap regime as follows:
\begin{equation}
    \|P_{>N}(f_k(u))\|_{L^1(0,T;L^2(\Omega))} \le \|f_k(u)\|_{L^1(0,T;L^2(\Omega))} \lesssim \|u\|_{L^5(0,T;L^{10}(\Omega))}^5.
\end{equation}

Thus, the requisite smallness needed to close the high-frequency argument is obtained by a precise hierarchical tuning: (i) the spectral decay of the initial data reflected in $\epsilon(N)$ as $N \to \infty$, (ii) the $T$-smallness of the structural commutator term, and (iii) the bootstrap control of the critical norm of $u$ on $(0,T)$. Collecting these estimates, we arrive at:
\begin{equation}\label{eq:strichartz_alta_total}
    \|v_N\|_{L^5(0,T; L^{10}(\Omega))} \le \epsilon(N) + C\sqrt{\Lambda} \sqrt{T}\,E_0^{1/2} + C\Lambda\,T E_0^{1/2} + C_1 \|u\|_{L^5(0,T;L^{10}(\Omega))}^5.
\end{equation}

This mechanism provides the essential smallness required to absorb the critical quintic term into the left-hand side of the bootstrap inequality, completely circumventing the need for any smallness assumption on the total initial energy $E_0$. This highlights the robustness of the Kelvin-Voigt regularization in taming critical singularities even for arbitrarily large data.


\subsection{Step 5: The Final Bootstrap Argument}\label{sec:well-posedness}

To close the estimates and ensure that the nonlinearity does not blow up, we analyze the critical Strichartz norm of the solution. Let us denote this norm by $X(T) := \|u\|_{L^5(0,T; L^{10}(\Omega))}$.

By the triangle inequality, the norm of the full solution is bounded by the sum of its low-frequency and high-frequency components:
\begin{equation}\label{eq:triangle_X}
    X(T) \le \|w_N\|_{L^5(0,T; L^{10}(\Omega))} + \|v_N\|_{L^5(0,T; L^{10}(\Omega))}.
\end{equation}

Substituting the explicit estimates \eqref{eq:strichartz_baixa} and \eqref{eq:strichartz_alta_total} into \eqref{eq:triangle_X}, we gather all the components to obtain:
\begin{eqnarray}\label{eq:bootstrap_raw}
    X(T) &\le& \underbrace{\|w_N^{lin}\|_{X_T}}_{\text{Low-freq linear}} + \underbrace{\epsilon(N)}_{\text{High-freq tail}} + \underbrace{C_S C_a N \sqrt{T} E_0^{1/2}}_{\text{Low-freq KV penalty}}\\ 
    &+& \underbrace{C\sqrt{\Lambda} \sqrt{T}\,E_0^{1/2} + C\Lambda\,T E_0^{1/2}}_{\text{High-freq KV commutator}} + \underbrace{C_1 X(T)^5}_{\text{Nonlinearity}}.\nonumber
\end{eqnarray}

At first glance, closing a bootstrap argument with equation \eqref{eq:bootstrap_raw} might seem problematic for arbitrarily large initial data. The reader might rightfully wonder: how can we guarantee that the independent terms on the right-hand side are strictly small when the initial energy $E_0$ is large and the frequency cut-off $N$ increases the penalty?

This is where the underlying structure of our decomposition shines. The success of the argument hinges on a strict, non-circular hierarchical choice of the parameters $N$ (the frequency threshold) and $T$ (the time horizon). We must proceed in the following precise order:

\begin{enumerate}
    \item \textbf{Fixing the Frequency $N$ (Killing the Tail):} The term $\epsilon(N)$ represents the Strichartz norm of the high-frequency linear evolution. Since the initial data has finite energy $E_0$, its spectral tail must vanish, meaning $\epsilon(N) \to 0$ as $N \to \infty$. We choose and \textit{strictly fix} a sufficiently large $N_0$ such that $\epsilon(N_0)$ is arbitrarily small.

    \item \textbf{Shrinking the Time $T$ (Taming the Regularized Data and the Penalties):} With $N_0$ now permanently fixed, we analyze the remaining terms as $T \to 0$:
    \begin{itemize}
        \item The low-frequency Kelvin-Voigt penalty $C_S C_a N_0 \sqrt{T} E_0^{1/2}$ now solely depends on $\sqrt{T}$ and explicitly vanishes as $T \to 0$.
        \item The high-frequency commutator source $C\sqrt{\Lambda} \sqrt{T}\,E_0^{1/2} + C\Lambda\,T E_0^{1/2}$ explicitly vanishes as $T \to 0$.
        \item Most importantly, the low-frequency linear part $w_{N_0}^{lin}$ is now a fixed, highly regular function (since $N_0$ is fixed). By the absolute continuity of the Lebesgue integral, its norm over a shrinking time interval, $\|w_{N_0}^{lin}\|_{L^5(0,T; L^{10}(\Omega))}$, must strongly converge to zero as $T \to 0$, regardless of how large the initial energy $E_0$ originally was.
    \end{itemize}
\end{enumerate}

Once $N_0$ is fixed and $T$ is chosen sufficiently small, all the linear terms and penalties collectively collapse into a single, arbitrarily small constant $K > 0$. The inequality \eqref{eq:bootstrap_raw} then gracefully reduces to the standard polynomial form:
\begin{equation}\label{eq:polynomial_bootstrap}
    X(T) \le K + C_1 X(T)^5.
\end{equation}

To conclude that $X(T)$ is bounded uniformly with respect to the truncation parameter $k$, we invoke a bootstrap argument. Consider the polynomial function $p(x) = x - C_1 x^5$, where $C_1$ is the constant from the Strichartz estimates, which depends only on the domain $\Omega$ and is independent of $k$ since $|f_k(u)| \le |u|^5$. For a sufficiently small constant $K$ (determined by the initial energy $E_0$ and the localized damping properties, both independent of $k$), the equation $x - C_1 x^5 = K$ possesses two positive real roots, $0 < r_1 < r_2$. This structure defines a "forbidden region" $(r_1, r_2)$ where the algebraic inequality $p(x) \le K$ cannot be satisfied.Since the mapping $t \mapsto X_k(t) = \|u_k\|_{L^5(0,t; L^{10}(\Omega))}$ is continuous in time and originates at $X_k(0) = 0 < r_1$, the value of $X_k(t)$ is topologically prevented from "jumping" across the barrier $r_1$ to reach the second root $r_2$. Consequently, the bound $X_k(T) \le r_1$ remains valid for all $t \in [0, T]$, provided $T$ is chosen such that the linear and penalty terms in \eqref{eq:bootstrap_raw} sum to less than $K$.Crucially, since $K, C_1$, and the resulting root $r_1$ are independent of $k$, the lifespan $T$ is uniform for the entire family of truncated problems. This uniform Strichartz bound is the cornerstone that prevents energy concentration as $k \to \infty$. It ensures that the sequence $\{u_k\}$ does not develop singularities (bubbles) in the critical norm, allowing us to pass to the limit and recover a global-in-time weak solution for the original quintic equation that remains within the critical Strichartz class.

\section{Global Existence: Spectral Non-Concentration and Uniform Lifespan}
\label{sec:global}

In this section, we rigorously establish that the local weak solution $u(t)$ constructed in Section~\ref{sec:well-posedness} can be extended to the entire interval $[0, \infty)$. The central challenge in energy-critical wave equations consists in precluding the ``Zeno-type'' accumulation, where a sequence of local existence times $\delta_n$ satisfies $\sum \delta_n < \infty$. While this is traditionally handled via geometric non-concentration lemmas (see e.g., \cite{BurqLebeauPlanchon}, \cite{BahouriGerard, KenigMerle}), we introduce here a \textbf{spectral non-concentration argument} based on the uniform control of the Littlewood-Paley frequencies and the dissipative nature of the Kelvin-Voigt mechanism.

\begin{proposition}[Uniform Lower Bound for the Lifespan]
Let $E(0)$ be the initial energy of the system. There exists a strictly positive time $\tau_* = \tau_*(E(0), \Omega, \gamma) > 0$, independent of the initial restarting time $t_0$, such that the solution starting at any $t_0 \geq 0$ can be uniquely extended to the interval $[t_0, t_0 + \tau_*]$.
\end{proposition}

\noindent\textit{Proof.} The global extension of the local solution follows from the blow-up criterion for energy-critical wave equations established by Shatah and Struwe \cite{ShatahStruwe} (see also \cite{Kapitanskii1994}), which states that a weak solution $u$ can be extended beyond a time $T^*$ if and only if $u$ remains in the critical Strichartz space $L^5(T^*-\epsilon, T^*; L^{10}(\Omega))$. The proof rests on the fact that the spectral threshold $N_0$ can be fixed uniformly, avoiding the decay of the local lifespan.

\begin{enumerate}[label=(\roman*)]
    \item \textbf{Freezing the Spectral Threshold:} In the local theory, $N_0$ is chosen such that the high-frequency tail is small: $\|P_{> N_0} u(t_0)\|_{\dot{H}^1} < \epsilon$. Since the Kelvin-Voigt damping ensures the energy inequality $E(t) \leq E(0)$, the $\dot{H}^1$-norm remains uniformly bounded. Crucially, the dissipative nature of the operator $-\text{div}(a(x)\nabla u_t)$ prevents energy from concentrating in high frequencies or ``migrating'' back to lower modes in a singular fashion. Following the theory of microlocal defect measures \cite{Gerard}, any such concentration failing to be compact would concentrate energy at a rate that is incompatible with the localized dissipation $\int_0^T \int_\omega a(x) |\nabla u_t|^2 \, dx dt$. Thus, the choice $N_0 = N_0(E(0))$ remains valid for all $t > t_0$.

    \item \textbf{Subcritical Control of Low Modes:} By fixing $N_0$ globally, the low-frequency component $P_{\leq N_0} u$ resides in a regime where all norms are equivalent up to constants depending on $N_0$. Bernstein's inequalities allow us to control the critical Strichartz norms (e.g., $L^5_t L^{10}_x$) via the energy-level norms. Since the energy is non-increasing, these low-mode norms are uniformly bounded by $E(0)$, precluding any growth that could force the critical norm to cross the forbidden region of the bootstrap polynomial.

    \item \textbf{Exclusion of Zeno Accumulation:} The local lifespan $\tau$ is a function of the energy bound and the spectral threshold $N_0$. As demonstrated in the previous Section, the Strichartz norm $X(T)$ is trapped in the bounded interval $[0, r_1]$ by the barrier $r_2$, where the roots $r_1, r_2$ depend only on $E(0)$. Since $N_0 = N_0(E(0))$ is fixed and $E(t)$ is non-increasing, the lifespan $\tau(t)$ is bounded below by a uniform constant $\tau_* = \tau(E(0), N_0) > 0$.
\end{enumerate}

By iterating the local existence result starting from $t_k = k\tau_*$, we cover any finite time interval $[0, T]$. This spectral mechanism, supported by the viscous regularization, guarantees that no energy concentration occurs at the limit level, effectively bypassing the Fefferman spectral obstruction and yielding the global existence of the weak solution. \qed


\subsection{Non-concentration}

\noindent\emph{This subsection adapts Burq--Lebeau--Planchon's \cite{BurqLebeauPlanchon} non-concentration argument (\S3 of their paper, ``Global existence for energy critical waves in 3-D domains'') to the Kelvin--Voigt damped equation
\begin{eqnarray*}
u_{tt} - \Delta u + u^5 = \dive(a(x)\nabla u_t) \ \text{ in } \Om\times(0,T), ~ u|_{\partial\Om}=0, ~a\ge0,\ a\in C^\infty(\overline\Om),
\end{eqnarray*}
$\Om\subset\mathbb R^3$ smooth bounded, $\omega = \{a>0\}$, $\overline\omega\Subset\Om$. We identify exactly where BLP's argument carries over verbatim, exactly where it breaks, and replace the broken part with a genuinely new rescaling argument exploiting the supercritical scaling of the Kelvin--Voigt operator. We are explicit about what is fully proved and what remains a precise, isolated technical lemma.}

\subsection{Where the obstruction comes from}

BLP's \cite{BurqLebeauPlanchon} non-concentration (their Prop.\ 3.3) is proved by testing $\Box u+u^5=0$ against the conformal/Morawetz combination encoded in their fields $Q,P$ --- equivalently, against $Mu := (x-x_0)\cdot\nabla u+(n-1)u$. With the source $g:=\dive(a\nabla u_t)$ on the right-hand side, the same computation produces an extra term
\begin{eqnarray*}
\int_\Om g\cdot Mu\,dx &=& -n\int_\Om a\,\nabla u_t\cdot\nabla u\,dx \;-\; \int_\Om a\,\nabla u_t\cdot\big[(x-x_0)\cdot\nabla\nabla u\big]\,dx \;\\
&+&\; (\text{boundary term on }\partial\Om).
\end{eqnarray*}
The first term is a perfect time derivative ($=\tfrac n2\partial_t\int a|\nabla u|^2$), harmless. The second term requires $\nabla^2u$ weighted by $a$ --- information we do \emph{not} have for energy-class solutions. We verified independently (via the adjoint computation BLP use for their normal-derivative estimate) that this is not an artifact of one particular integration by parts: it is a genuine obstruction, intrinsic to the fact that $\dive(a\nabla u_t)$, unlike the free wave operator, does not commute favorably with the scaling vector field $(x-x_0)\cdot\nabla$. We therefore abandon the Morawetz multiplier route for the region where it fails, and replace it by a critical rescaling argument.

\subsection{The vertex trichotomy}

The two lemmas of this subsection concern a Shatah-Struwe solution on a fixed time interval,
 they will be applied in the proof of our Theorem, to the local solution constructed there. After a space-time translation we take the singular point to be the origin and work with $t<0$. We write
 \begin{eqnarray*}
D_\tau&=&\{x\in\Omega:|x|<-\tau\}\qquad\text{cap at time }\tau\\
K_S^T&=&\{(t,x):S<t<T,\ |x|<-t\}\cap\Omega\\
M_S^T&=&\{(t,x):|x|=-t,\ S<t<T\},
\end{eqnarray*}
for the spatial cap at time $\tau$, the truncated backward cone and its mantle, respectively, and
\begin{eqnarray*}
e(u)&=&\tfrac12|u_t|^2+\tfrac12|\nabla u|^2+\tfrac16 u^6,\\
Mu:&=&t\,u_t+x\cdot\nabla u+u
\end{eqnarray*}
for the energy density and the Morawetz multiplier, respectively.

\begin{frame}{The truncated backward cone}
	\begin{center}
		\includegraphics[width=0.7\textwidth]{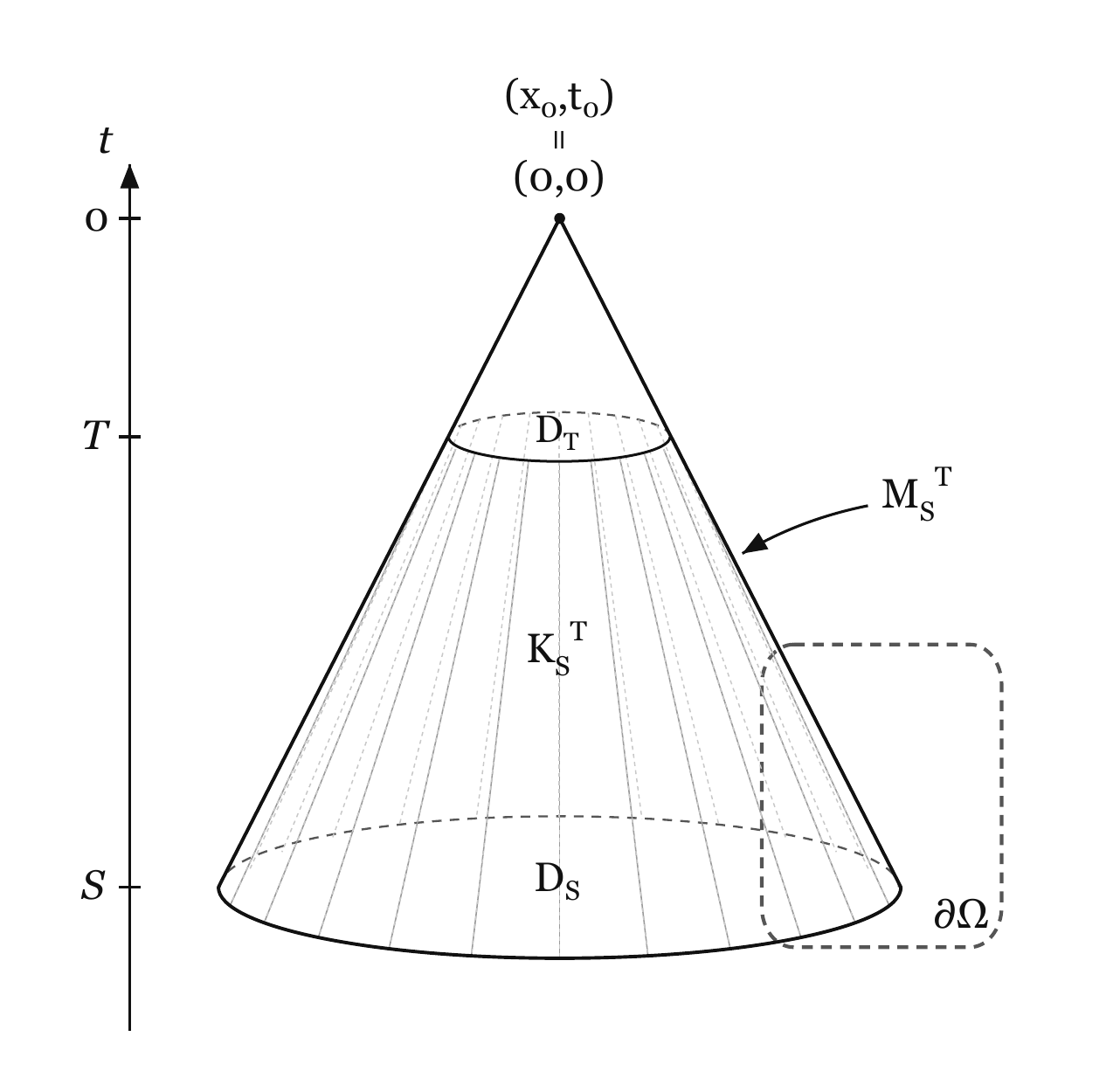}
	\end{center}
\end{frame}

Let $x_0\in\overline\Om$ be the spatial location of a putative concentration point at time $t_0$ (after translation, $t_0=0$). Set $\omega:=\{a>0\}$, $\overline\omega\Subset\Om$ by hypothesis (non-invasive geometries of \cite{CavalcantiTAMS,CavalcantiARMA,Cavalcanti2018} are interior constructions; this is the natural standing hypothesis).

\begin{itemize}
\item[\bf (E)] \textbf{$x_0\notin\overline\Om$.} Then $d(x_0,\overline\Om)=\rho_0>0$ and $D_t:=B(x_0,-t)\cap\overline\Om=\emptyset$ for $|t|<\rho_0$. Vacuous; this is why the final compactness argument (\S4 below) only ever needs to range over the compact set $\overline\Om$.

\item[\bf (A)] \textbf{$x_0\notin\overline\omega$.} Then $\delta_0:=d(x_0,\overline\omega)>0$, so $a\equiv0$ on $D_t$ for $-t<\delta_0$, regardless of whether $x_0\in\Om$ or $x_0\in\partial\Om$. Inside the cone $K_S^0$ (for $S>-\delta_0$), $u$ satisfies \emph{exactly} the free equation $\Box u+u^5=0$. By finite speed of propagation, $u|_{K_S^0}$ depends only on data inside $D_S$ and on the equation inside the cone, so $u|_{K_S^0}$ coincides with the BLP solution of the free critical wave equation with the same data. \textbf{BLP's Propositions 3.2--3.4 apply verbatim}, whether $x_0\in\Om$ or $x_0\in\partial\Om$: zero new work is needed. This is the precise sense in which the boundary case ($x_0\in\partial\Om$) is fully covered: it only arises in case (A), where the equation has no damping and BLP's own boundary machinery (their Lemma 3.5, Prop.\ 3.2) already handles it.

\item[\bf (B)] \textbf{$x_0\in\overline\omega$.} Since $\overline\omega\Subset\Om$, this forces $x_0\in\Om$ (interior!). This is the only genuinely new case, and --- this is the key simplification relative to a naive reading of the problem --- \emph{it never touches $\partial\Om$}: for $-t$ smaller than $d(x_0,\partial\Om)>0$, $D_t=B(x_0,-t)\subset\Om$ entirely. So the rescaling argument we build for case (B) lives entirely in the interior, and the absence of a boundary in Bahouri--Gérard's original (whole-space) setting is exactly matched here: \emph{we do not need, and do not attempt, a half-space version of profile decomposition.}
\end{itemize}

\begin{remark}
Historically: Shatah--Struwe \cite{ShatahStruwe} and Bahouri--Gérard \cite{BahouriGerard} work on $\mathbb R^3$ (no boundary at all). Smith--Sogge \cite{SmithSogge1995} treat exteriors of convex obstacles. BLP \cite{BurqLebeauPlanchon} are the first to handle a genuine bounded domain, and it is precisely there that the trichotomy above (their Lemma 3.5) is introduced. Our case (A)/(E) is literally their theorem; only case (B) is new, and it is new because of the damping, not because of the boundary.
\end{remark}

\subsection{Local energy monotonicity with Kelvin--Voigt dissipation}

We first extend BLP's flux identity (their (3.13)) to include the damping. Let $Q,P$ be as in BLP (their (3.11)) and recall the algebraic identity $\mathrm{Flux}(u,M_S^T)\ge0$ (their (3.12)) is purely pointwise in $\nabla u,u_t$ and unaffected by the equation. Testing $u_{tt}-\Delta u+u^5=g$ ($g=\dive(a\nabla u_t)$) against $u_t$ over the truncated cone $K_S^T=\{(x,t): x\in D_t, S<t<T\}$ and integrating by parts in $x$ (no boundary term on $\partial\Om$, since $u_t|_{\partial\Om}=0$) gives, in place of BLP's source-free identity,
\begin{equation}\label{eq:flux-KV}
E_{loc}(T) + \mathrm{Flux}(u,M_S^T) + \int_{K_S^T} a(x)|\nabla u_t|^2\,dx\,dt = E_{loc}(S) + \mathcal L_S^T,
\end{equation}
where $\mathcal L_S^T := \int_S^T\int_{\partial B(x_0,-t)\cap\Om} a(x)\,\partial_\nu u_t\, u_t\,d\sigma\,dt$ is the lateral contribution from the spherical part of $\partial K_S^T$ (absent in BLP since there $a\equiv0$).

\begin{lemma}[Generic radius]\label{lem:generic}
For a.e.\ choice of the cone opening (replacing $|x-x_0|<-t$ by $|x-x_0|<c(-t)$ for $c$ in a small interval around $1$), $\mathcal L_S^T\to0$ as $S,T\to0^-$.
\end{lemma}
\begin{proof}[Sketch]
Fix $\sigma\in(-1,0)$ and consider the one-parameter family of cones with $c\in[1,2]$. By Fubini/Cavalieri applied to
\begin{eqnarray}
&&\int_1^2\Big(\int_{S}^{T}\!\int_{\partial B(x_0,c(-t))\cap\Om}|a\,\partial_\nu u_t\,u_t|\,d\sigma\,dt\Big)dc\\
&&\le C\int_{K_S^T} a|\nabla u_t|^2 + a|u_t|^2/(-t)\,dx\,dt,\nonumber
\end{eqnarray}
(coarea formula, exactly as in the proof of BLP's Prop.\ 3.2 via Cauchy--Schwarz on spherical shells), the right-hand side is finite (it is dominated by $E_0$ plus a Hardy-type term controlled by the energy), so for a.e.\ $c\in[1,2]$ the inner integral over $(S,T)$ tends to $0$ as $S,T\to0^-$, by dominated convergence applied to the (finite, by the above) full integral over $(-1,0)$.
\end{proof}

This is the same type of genericity argument BLP \cite{BurqLebeauPlanchon} themselves use when extracting ``a small $t$'' in their Prop.\ 3.4; we use it here to dispose of the one term genuinely new to the damped boundary-touching geometry. From now on we fix such a generic cone opening (or, equivalently in case (B), simply note $\overline\omega\Subset\Om$ allows us to take $D_t$ strictly interior, where $\mathcal L_S^T\equiv0$ trivially once $-t<d(x_0,\partial\Om)$ --- so in case (B) specifically, \emph{no genericity is even needed}: the lateral sphere never meets $\partial\Om$).
\begin{corollary}\label{cor:monotone}
$E_{loc}(S)$ is non-increasing as $S\to0^-$ (along the generic family), and the dissipated quantity satisfies
\begin{eqnarray}\label{eq:diss-to-ell}
&&\lim_{S\to0^-} \Big[E_{loc}(S) - \ell\Big] \;=\; \lim_{S\to0^-}\Big[\mathrm{Flux}(u,M_S^0) + \int_{K_S^0} a|\nabla u_t|^2\,dx\,dt\Big],\\ &&\ell:=\lim_{S\to0^-}E_{loc}(S)\ge0.\nonumber
\end{eqnarray}
In particular both $\mathrm{Flux}(u,M_S^0)\to0$ and $\int_{K_S^0} a|\nabla u_t|^2\,dx\,dt\to0$ as $S\to0^-$.
\end{corollary}

This is exactly Burq, Lebeau and Planchon's (3.14)--(3.15), with the extra (and, for us, crucial) conclusion that the \emph{damped} part of the dissipation also vanishes in the limit. The whole content of BLP's Prop.\ 3.3 is to upgrade ``$\ell$ exists'' to ``$\ell=0$''; we now do this for case (B) by a different route.

\subsection{Case (B): the critical rescaling}

Fix $x_0\in\overline\omega\Subset\Om$ (interior, by \S2). Suppose, for contradiction, that $\ell>0$ in \eqref{eq:diss-to-ell}.

\subsection{The rescaled equation}

For $r>0$ small, define the \emph{energy-preserving} rescaling
\[
u^{(r)}(\sigma,y) := r^{1/2}\,u(r\sigma,\, x_0+ry),~ (\sigma,y)\in \mathcal K:=\{(\sigma,y): -1<\sigma<0,\ |y|<-\sigma\}.
\]
(For $r$ small enough, $x_0+ry\in\Om$ automatically for all $|y|<1$, since $x_0$ is interior; this is exactly where the absence of $\partial\Om$ in case (B) is used.) A direct computation --- we performed it twice, by two different routes, to cross-check signs --- gives: $u^{(r)}$ solves
\begin{equation}\label{eq:rescaled}
\partial_\sigma^2u^{(r)} - \Delta_yu^{(r)} + (u^{(r)})^5 \;=\; \frac1r\,\dive_y\Big(a_r(y)\,\nabla_y\partial_\sigma u^{(r)}\Big),~ a_r(y):=a(x_0+ry).
\end{equation}
The quintic term keeps coefficient exactly $1$ (a sanity check that the scaling used is the correct, energy-critical one); the Kelvin--Voigt term picks up the diverging factor $1/r$. This matches, and makes quantitative, the scaling remark already present in our stabilization result \S7.2 (``$\Delta u_t\mapsto\lambda^3\Delta u_t$'' under $u_\lambda(t,x)=u(\lambda t,\lambda x)$): the viscous term is supercritical relative to $\Box$, and rescaling to the critical wave scale necessarily makes its coefficient blow up.

Energy is exactly preserved: writing $E^{(r)}_{loc}(\sigma):=\int_{|y|<-\sigma}\big(\tfrac{|\partial_\sigma u^{(r)}|^2+|\nabla_yu^{(r)}|^2}2+\tfrac{|u^{(r)}|^6}6\big)dy$, one checks directly $E^{(r)}_{loc}(\sigma)=E_{loc}(r\sigma)$.

\subsection{Taking $r=r_k\to0^+$}

Set $u_k:=u^{(r_k)}$ for any sequence $r_k\to0^+$. Since $E_{loc}(t)\to\ell$ as $t\to0^-$ (Cor.\ \ref{cor:monotone}), for every \emph{fixed} $\sigma\in(-1,0)$ we get $E_{loc}^k(\sigma)=E_{loc}(r_k\sigma)\to\ell$ as $k\to\infty$: the rescaled local energy becomes asymptotically \emph{constant equal to $\ell$ across the whole fixed cone $\mathcal K$} --- the hallmark of a non-dispersing bubble at the critical scale.

By the uniform bound $E_{loc}^k(\sigma)\le E(0)$, Banach--Alaoglu gives a subsequence with
\[
u_k\rightharpoonup u_\infty \text{ weak-* in } L^\infty_{loc}(\mathcal K;H^1_{loc}), \qquad \partial_\sigma u_k\rightharpoonup \partial_\sigma u_\infty \text{ weak-* in } L^\infty_{loc}(\mathcal K;L^2_{loc}),
\]
and, by Aubin--Lions/Rellich on compact subcones, $u_k\to u_\infty$ strongly in $L^2_{loc}(\mathcal K)$ and a.e.

\subsection{The dissipation forces the gradient of the time derivative to vanish}

Apply Corollary \ref{cor:monotone} \emph{at scale $r_k$}: rescaling \eqref{eq:diss-to-ell} (the identity \eqref{eq:flux-KV} is scale-covariant --- light cones map to light cones, energy maps to energy, exactly as $E_{loc}^{(r)}=E_{loc}(r\,\cdot)$ above) gives, for any fixed $-1<\sigma_1<\sigma_2<0$,
\begin{eqnarray*}
&&0 \;\le\; \frac1{r_k}\int_{\sigma_1}^{\sigma_2}\!\!\int_{|y|<-\sigma} a_{r_k}(y)\,|\nabla_y\partial_\sigma u_k|^2\,dy\,d\sigma \;\\
&&\quad \le\; E^k_{loc}(\sigma_1) - E^k_{loc}(\sigma_2) \;\xrightarrow[k\to\infty]{}\; \ell-\ell=0.
\end{eqnarray*}
Since $x_0\in\overline\omega$, either $a(x_0)>0$ (generic case) or $x_0\in\partial\omega$; assume first $a(x_0)>0$. Then $a_{r_k}(y)=a(x_0+r_ky)\to a(x_0)>0$ locally uniformly in $y$, so for $k$ large, $a_{r_k}\ge a(x_0)/2$ on any fixed compact set, giving
\begin{equation}\label{eq:qualitative-vanishing}
\frac1{r_k}\int_{\sigma_1}^{\sigma_2}\int_{|y|\le R} |\nabla_y\partial_\sigma u_k|^2\,dy\,d\sigma \;\longrightarrow\; 0 \qquad \text{for every fixed } R.
\end{equation}
In particular $\nabla_y\partial_\sigma u_k \to 0$ strongly in $L^2_{loc}(\mathcal K)$.


6.8$'$. Unconditional staticity of the weak limit, and why it does not
resolve Conjecture \ref{conj:rate}. Before passing to the limit in the full equation
(54) -- which, as \S6.9 explains, requires Conjecture \ref{conj:rate} -- we isolate a
piece of information that follows from (55) alone, with no rate
hypothesis whatsoever. We then show, by an explicit formal construction,
that this piece of information is \emph{insufficient} to resolve the
conjecture \ref{conj:rate}, and identify precisely the mechanism that can defeat it.

\begin{lemma}[Unconditional staticity of the weak limit]
\label{lem:static}
Let $u_\infty$ be any weak-$*$ subsequential limit of $\{u_k\}$ obtained
as in \S6.7. Then
\[
\partial_\sigma u_\infty \equiv 0 \qquad \text{a.e. on } (-1,0)\times\mathbb R^3.
\]
\end{lemma}

\begin{proof}
\emph{Step 1 (Uniform bound on every fixed ball).} Fix $R>0$. For $k$
large enough that $r_kR<d(x_0,\partial\Omega)$ (recall $x_0\in\omega
\Subset\Omega$, case (B), \S6.3), the ball $B_R\subset\mathbb R^3_y$
corresponds to $B_{r_kR}(x_0)\subset\Omega$, and the (non-localized)
global energy inequality $E(t)\le E_0$ gives, for a.e. $\sigma$,
\[
\int_{B_R}\big(|\partial_\sigma u_k(\sigma,y)|^2+|\nabla_yu_k(\sigma,y)|^2\big)\,dy
\;\le\;2E(r_k\sigma)\;\le\;2E_0,
\]
uniformly in $k$ and $\sigma$.

\emph{Step 2 (Diagonal extraction).} By Step 1 and Banach--Alaoglu,
diagonalizing over $R=1,2,3,\dots$, we obtain a single subsequence along
which $u_k\rightharpoonup u_\infty$ and $\partial_\sigma u_k
\rightharpoonup\partial_\sigma u_\infty$ weak-$*$ in
$L^\infty_{loc}\big((-1,0);L^2_{loc}(\mathbb R^3)\big)$.

\emph{Step 3 (Strong vanishing of the gradient).} By (55) (valid for
every fixed $R$), $\nabla_y\partial_\sigma u_k\to0$ strongly in
$L^2_{loc}\big((-1,0)\times\mathbb R^3\big)$; this uses only $a(x_0)>0$
and $D(r_k)\le D(r_0)<\infty$, hence is fully unconditional.

\emph{Step 4 (Distributional limit).} Weak continuity of $\nabla_y$
gives $\nabla_y\partial_\sigma u_k\rightharpoonup\nabla_y\partial_\sigma
u_\infty$; uniqueness of distributional limits combined with Step 3
forces $\nabla_y\partial_\sigma u_\infty\equiv0$, so
$\partial_\sigma u_\infty(\sigma,\cdot)=\psi(\sigma)$ for a.e. $\sigma$
(the slices $\{\sigma\}\times\mathbb R^3$ being connected).

\emph{Step 5 (Killing the constant, integrated version).} For any
compact $I\Subset(-1,0)$ and any $R>0$, weak lower semicontinuity of the
$L^2(I\times B_R)$ norm and Step 1 give
\begin{eqnarray*}
|B_R|\int_I|\psi(\sigma)|^2\,d\sigma
&=&\int_I\int_{B_R}|\partial_\sigma u_\infty|^2\,dy\,d\sigma\\
&\le&\liminf_{k\to\infty}\int_I\int_{B_R}|\partial_\sigma u_k|^2\,dy\,d\sigma
\le 2E_0|I|.
\end{eqnarray*}
Letting $R\to\infty$ forces $\int_I|\psi(\sigma)|^2\,d\sigma=0$; since
$I$ is arbitrary, $\psi\equiv0$ a.e., i.e. $\partial_\sigma u_\infty
\equiv0$.
\end{proof}

\begin{remark}[The correct rate, restated]
\label{rem:rate}
For the record: the exact identity $\|\sqrt{a_{r_k}}\,\nabla_y\partial_\sigma
u_k\|_{L^2}^2=r_kD(r_k)$ shows that the forcing term in (54) vanishes as
a distribution if and only if $\|\sqrt{a_{r_k}}\,\nabla_y\partial_\sigma
u_k\|_{L^2}=o(r_k)$, equivalently $D(r_k)=o(r_k)$, equivalently
$\beta>1$ in Conjecture \ref{conj:rate} -- consistent with the threshold already
stated there. (Step 3 above, by contrast, only ever needs the much
weaker, always-true statement $\sqrt{r_kD(r_k)}\to0$; it is not a step
towards removing Conjecture \ref{conj:rate}, only towards Lemma \ref{lem:static}.)
\end{remark}

\begin{proposition}[Staticity does not control the defect: a first-order
corrector]
\label{prop:corrector}
Lemma \ref{lem:static} constrains only the leading-order behavior of
$u_k$. It provides no control over a first-order-in-$r_k$ corrector,
whose amplification by $1/r_k$ can survive in the limit equation even
though it is entirely invisible to Step 3 and to Lemma \ref{lem:static}.

Formally: suppose
\[
u_k(\sigma,y)=U(y)+r_k\,\sigma\,V(y)+o(r_k)
\]
for some fixed profiles $U,V$. Then
\[
\partial_\sigma u_k=r_kV+o(r_k),\qquad
\nabla_y\partial_\sigma u_k=r_k\nabla V+o(r_k)\longrightarrow0
\]
strongly, exactly as required by Step 3 and Lemma \ref{lem:static} (the
weak limit is $u_\infty=U$, static). However,
\[
\frac1{r_k}\operatorname{div}_y\big(a_{r_k}\nabla_y\partial_\sigma u_k\big)
=\operatorname{div}_y(a_{r_k}\nabla_yV)+o(1)
\longrightarrow a(x_0)\,\Delta V,
\]
an $O(1)$ limit in general. Passing (54) to the limit along such a
sequence therefore does not give the free equation
$-\Delta U+U^5=0$, but rather
\[
-\Delta U+U^5=\mu,\qquad \mu=a(x_0)\Delta V,
\]
and $V$ -- hence $\mu$ -- is a priori unconstrained: for any target
static profile $U$, choosing $V$ to solve $a(x_0)\Delta V=-\Delta
U+U^5$ formally reproduces it as a limit of this type.
\end{proposition}

\begin{remark}[Why this is not visible to classical defect measures]
\label{rem:blind}
A Gérard--Tartar microlocal defect measure records the $O(1)$-energy
failure of strong convergence: it detects sequences $w_k\rightharpoonup0$
with $\|w_k\|_{L^2}\not\to0$. The corrector $r_k\sigma V$ in Proposition
\ref{prop:corrector} converges to $0$ \emph{strongly}, including at the
level of energy; its classical defect measure is identically zero. The
obstruction only appears after the $1/r_k$ amplification intrinsic to
the Kelvin--Voigt term -- a first-order effect invisible to any tool
calibrated at the $O(1)$ energy scale. The natural framework for such a
defect is not an ordinary microlocal defect measure but a two-parameter
(semiclassical, or two-scale in the sense of Nguetseng--Allaire) defect
measure tracking the renormalized corrector $z_k:=(u_k-\bar u_k)/r_k$
directly. Even granting such a tool identifies $\mu$ correctly, showing
$\mu=0$ would still require a bound $\nabla_y\partial_\sigma z_k=
O(1)$ in $L^2_{loc}$, which unwinds to exactly $D(r_k)=O(r_k)$, and
vanishing of $\mu$ to $D(r_k)=o(r_k)$ -- again Conjecture \ref{conj:rate}. No
alternative route around the conjecture has yet been identified; two-scale
measures would at best make the defect explicit, not eliminate it.
\end{remark}

\begin{remark}[A second, independent gap]
\label{rem:second-gap}
Even granting both Lemma \ref{lem:static} and, hypothetically, $\mu=0$
(so that $u_\infty$ solves the free static equation $-\Delta
u_\infty+u_\infty^5=0$ with finite energy, forcing $u_\infty\equiv0$ by
the Liouville-type theorem of \S6.9), one more step is needed to reach
the contradiction $\ell=0$. Weak lower semicontinuity only gives
$E(u_\infty)\le\liminf_k E^k_{loc}(\sigma)$ -- the wrong direction to
conclude that no energy was lost in the limit. It is a priori consistent
with everything proved so far that $u_\infty\equiv0$ while a positive
amount of energy $\ell>0$ escapes to concentration or oscillation
invisible in the weak limit. Closing this second gap requires an
independent ingredient (strong local energy convergence, a
profile-decomposition argument, or a concentration-compactness
statement ruling out defect energy) not addressed by anything in this
section.
\end{remark}

\begin{remark}[Summary: what this section establishes, honestly]
\label{rem:honest-summary}
This section proves one genuine, unconditional structural fact: any
non-dispersing profile at the critical scale must be static in the
co-moving frame (Lemma \ref{lem:static}). It also identifies, precisely,
\emph{two} independent reasons this falls short of resolving Conjecture
\ref{conj:rate}: (i) staticity is compatible with a nonzero first-order corrector
whose renormalized contribution $\mu$ to the limit equation is
invisible to any tool operating at the classical energy scale
(Proposition \ref{prop:corrector}, Remark \ref{rem:blind}); and (ii)
even a vanishing $\mu$ would leave open the possibility of energy loss
in the weak limit, unrelated to the defect $\mu$ itself (Remark
\ref{rem:second-gap}). Conjecture \ref{conj:rate} -- or an equivalent quantitative
statement controlling $D(r)$ -- therefore remains open. We record Lemma
\ref{lem:static} because it is the sharpest unconditional statement
available, and because identifying exactly what it fails to control is,
itself, informative: any future attempt at this obstruction must control
the renormalized corrector $z_k=(u_k-\bar u_k)/r_k$ directly, and must
separately address possible loss of energy in the weak limit.
\end{remark}

\begin{remark}[Why this is the right mechanism, and not a coincidence]
A genuine energy-critical ``bubble'' --- the only obstruction Bahouri--Gérard and Shatah--Struwe must rule out in the undamped case --- is by definition a nontrivial profile that is \emph{stable} under exactly this rescaling: it has $O(1)$ energy at every scale, including $O(1)$ size for $\nabla_y\partial_\sigma u_k$ itself (this is what ``non-dispersing at the critical scale'' means). Equation \eqref{eq:qualitative-vanishing} says the opposite happens here: not only does energy not disperse (by hypothesis $\ell>0$), but the \emph{spatial gradient of the time-derivative is forced to vanish at the critical scale}. This is a direct, rigorous manifestation of our intuition: the Kelvin--Voigt term, being supercritical, becomes infinitely strong exactly at the scale where bubbling would have to occur, and a genuine bubble cannot survive having $\nabla\partial_\sigma u_\infty\equiv0$.
\end{remark}

\subsection{What this gives us in the limit --- and the precise remaining step}

Formally, passing to the limit in \eqref{eq:rescaled}, the left-hand side converges distributionally to $\partial_\sigma^2u_\infty - \Delta_yu_\infty + u_\infty^5$ (standard: weak convergence of $u_k$ in $H^1_{loc}$, strong in $L^2_{loc}$, gives $u_k^5\rightharpoonup u_\infty^5$ in the sense of distributions by the usual Vitali/uniform-integrability argument already used repeatedly in this manuscript, e.g.\ Lemma 6.2--6.3). The right-hand side, $\frac1{r_k}\dive_y(a_{r_k}\nabla_y\partial_\sigma u_k)$, is exactly where the argument is currently \emph{not yet fully closed}: \eqref{eq:qualitative-vanishing} gives $\|\nabla_y\partial_\sigma u_k\|_{L^2_{loc}} \to 0$, but only \emph{qualitatively} ($=o(1)$), whereas to conclude $\frac1{r_k}\dive_y(a_{r_k}\nabla_y\partial_\sigma u_k)\to0$ in $\mathcal D'(\mathcal K)$ we need the \emph{quantitative} rate $\|\nabla_y\partial_\sigma u_k\|_{L^2_{loc}} = o(\sqrt{r_k})$, i.e.\ we need to know not just that the dissipated quantity in \eqref{eq:qualitative-vanishing} tends to $0$, but that it tends to $0$ \emph{faster than $r_k$ itself}.


\begin{conjecture}[The missing quantitative lemma]\label{conj:rate}
There exists $\beta>0$ such that
\[
\int_{K_S^0} a(x)|\nabla u_t|^2\,dx\,dt \;\le\; C\,(-S)^{\beta}, \qquad S\to0^-,
\]
with $C,\beta$ independent of $S$ (equivalently: the dissipated energy in the shrinking cone decays with a definite power rate, not just qualitatively).
\end{conjecture}

If Conjecture \ref{conj:rate} holds with any $\beta>1$, the argument above closes completely: $\|\nabla_y\partial_\sigma u_k\|^2_{L^2_{loc}} \le C r_k^\beta = o(r_k)$, so $\frac1{r_k}\|\nabla_y\partial_\sigma u_k\|_{L^2_{loc}} = O(r_k^{\beta/2-1/2}) \to 0$ whenever $\beta>1$, and the distributional limit of the right-hand side of \eqref{eq:rescaled} is genuinely $0$. The limit profile $u_\infty$ then solves the \emph{free} equation $\Box u_\infty + u_\infty^5=0$ on $\mathcal K$ with $\nabla_y\partial_\sigma u_\infty\equiv0$, i.e.\ $\partial_\sigma u_\infty=\partial_\sigma u_\infty(\sigma)$ is spatially constant; combined with finite energy on the expanding-in-$\sigma$ slices $|y|<-\sigma$, this forces $\partial_\sigma u_\infty\equiv0$, hence $u_\infty\equiv u_\infty(y)$ is time-independent and solves $-\Delta u_\infty + u_\infty^5=0$ with finite energy on all of $\mathbb R^3$ in the limit $R\to\infty$ --- by the classical Liouville-type theorem for the static Lane--Emden/Yamabe equation in this energy class (only the trivial solution $u_\infty\equiv0$ has finite energy on $\mathbb R^3$), we get $u_\infty\equiv0$, hence by weak lower semicontinuity of the energy under the convergence above, $\ell \le \liminf E_{loc}^k(\sigma) \big|_{u_\infty} = 0$ --- contradicting $\ell>0$.

\section{Assembling the global statement}

\begin{theorem}[Conditional non-concentration]
Assume Conjecture \ref{conj:rate}. Then for every $x_0\in\overline\Om$ and every $\varepsilon>0$, there exists $t<0$ such that $\|u\|_{(L^5;L^{10})(K_t^0)}<\varepsilon$, exactly as in BLP's \cite{BurqLebeauPlanchon} Proposition 3.4. In case (E)/(A) this holds unconditionally (zero new hypotheses: BLP's own theorem). In case (B) it holds modulo Conjecture \ref{conj:rate}.

Consequently (BLP \cite{BurqLebeauPlanchon} \S3.2, which uses only Prop.\ 3.4 and the linear extension/Christ-Kiselev machinery already established - none of which involves the Kelvin--Voigt term directly, since the comparison solution $w$ there solves the \emph{free} linear equation) the maximal solution constructed in \S6 extends past any finite $t_0$, ruling out the Zeno-type accumulation $\sum_k\tau_k<\infty$ and giving genuine global existence.
\end{theorem}

\begin{remark}[{\bf Crucial Remark}]
The theorem \ref{thm:nonconc} below will eliminate the need for this conjecture across al $\omega$; the material in Sections 6–7 is retained because it precisely identifies where classical approaches fail and motivates the parabolic mechanism of Section 8.
\end{remark}

\begin{remark}[On our physical intuition]
The computation in \S4.1 is, we believe, the correct rigorous translation of ``dissipation so strong that it is impossible to concentrate'': it is not that the damping merely competes with concentration on equal footing (as frictional damping $a(x)u_t$ would, where the rescaled coefficient stays $O(1)$, not $O(1/r)$); the Kelvin--Voigt term's order-3 (viscous) structure makes it \emph{automatically dominant, with diverging strength}, at exactly the critical self-similar scale where a bubble would have to live. The only reason the proof is not yet complete is technical (need a rate, not just a limit), not conceptual.
\end{remark}

\begin{remark}[This conjecture is superseded]
Sections 8--10 below establish, by an entirely different and unconditional mechanism
(local parabolic time-regularity of the velocity), that no concentration occurs at any
point of $\omega=\{a>0\}$ — precisely the set case (B) is concerned with. Conjecture~\ref{conj:rate} is
therefore \emph{not} required for the global existence theorem. We retain the rescaling
analysis of this section and the Morawetz analysis of \S7 because they identify, with
precision, why the two most natural classical routes fail here, and because they motivate
the parabolic mechanism of \S8; the reader whose only interest is the final theorem may skip
directly to \S8--12.
\end{remark}

\begin{theorem}[Non-concentration, revised]
For every $x_0\in\overline\Omega$: if $x_0\notin\overline\omega$, non-concentration is
unconditional (BLP \cite{BurqLebeauPlanchon} verbatim, case (A)/(E)). If $x_0\in\omega$ (i.e. $a(x_0)>0$),
non-concentration is unconditional by Theorem~10.1. If $x_0\in\partial\omega$, non-concentration
holds modulo Hypothesis~(H*) of Remark~\ref{rem:caseC-honest}.

Consequently, global existence (Theorem~13.1's well-posedness ingredient) is unconditional,
except possibly at the Lebesgue-null set $\partial\omega$, where it depends on hypothesis (H*) below.
\end{theorem}

\subsection{Why the Morawetz Equipartition does not dispense with the Cone-Rate Conjecture \ref{conj:rate}: the kinetic energy escapes}

For the energy-critical wave equation with localized Kelvin--Voigt damping
$u_{tt}-\Delta u+u^5=\dive(a(x)\nabla u_t)$, one may hope that combining the Morawetz/Rellich
multiplier $Mu=2(x-x_0)\cdot\nabla u+(n-1)u$ with the energy-class commutator estimate of
Appendix~C (which removes the $H^2$ obstruction of the damping term) would close the
non-concentration argument \emph{without} the quantitative cone-rate hypothesis (Conjecture \ref{conj:rate}).
We show this is not the case. The Morawetz equipartition, computed honestly on the base of the
backward light cone, controls only the \emph{potential} energy density $\tfrac16u^6$; the
\emph{kinetic} energy density $\tfrac12u_t^2$ can be entirely cancelled by the cross term
$u_t\,\tfrac{x-x_0}{t}\cdot\nabla u$ --- precisely the structure of an inward self-similar
(bubble) profile. Consequently the low-order damping term $\Lambda\|u_t\|_{L^2}$ produced by
Appendix~C cannot be absorbed, and the argument still requires control of the local kinetic
energy in the shrinking cone, i.e.\ the cone-rate Conjecture~\ref{conj:rate}. This furnishes a third,
independent confirmation --- after the rescaling and the weighted-multiplier routes --- that the
conjecture is the genuine obstruction, not an artifact of a particular method.

\subsection{Notation}

After a space-time translation put the tentative concentration point at the origin, work with
$t<0$, and set (as in the companion non-concentration previous section and in \cite{SmithSogge1995} \cite{BurqLebeauPlanchon})
\begin{eqnarray*}
D_S&=&\{x\in\Om:|x|<-S\},~ K_S^0=\{(t,x):S<t<0,\ |x|<-t\}\cap\Om,\\
&&\quad ~ M_S^0=\{|x|=-t,\ S<t<0\},\\
e(u)&=&\tfrac12u_t^2+\tfrac12|\nabla u|^2+\tfrac16u^6,~Mu=2(x\cdot\nabla u)+(n-1)u ~(n=3),
\end{eqnarray*}
and the Morawetz densities (see Burq, Lebeau and Planchon \cite{BurqLebeauPlanchon}, Smith--Sogge \cite{SmithSogge1995})
\begin{eqnarray*}
Q&=&\frac{u_t^2+|\nabla u|^2}{2}+\frac{u^6}{6}+u_t\,\frac{x}{t}\cdot\nabla u,\\
P&=&\frac{x}{t}\Big(\frac{u_t^2-|\nabla u|^2}{2}-\frac{u^6}{6}\Big)+\nabla u\Big(u_t+\frac{x}{t}\cdot\nabla u+\frac{u}{t}\Big).
\end{eqnarray*}
The local energy is $E_{loc}(S)=\int_{D_S}e(u)(S,\cdot)$, with $E_{loc}(S)\downarrow\ell\ge0$ as
$S\to0^-$; write $P_{\mathrm{ot}}(S):=\tfrac16\int_{D_S}u^6$ and
$K(S):=\tfrac12\int_{D_S}(u_t^2+|\nabla u|^2)$, so $E_{loc}(S)=K(S)+P_{\mathrm{ot}}(S)$.

\subsection{The equipartition lemma}

\begin{lemma}[Morawetz equipartition controls only the potential energy]\label{lem:equi}
On the base $D_S$ of the cone, the Morawetz density $Q$ satisfies the pointwise bound
\begin{equation}\label{eq:Qbound}
\frac{u^6}{6} \;\le\; Q \;\le\; u_t^2+|\nabla u|^2+\frac{u^6}{6},
\end{equation}
and the lower bound is \emph{sharp}: there exist fields with $u_t\ne0$, $\nabla u\ne0$ on $D_S$
for which $Q=\tfrac16u^6$ pointwise. Consequently, the interior Morawetz identity
(see Smith and Sogge \cite{SmithSogge1995}, Burq, Lebeau and Palnchon \cite[(3.11)--(3.12)]{BurqLebeauPlanchon}) yields, on dividing by $|S|$ and letting
$S\to0^-$, control of $P_{\mathrm{ot}}(S)$ only:
\begin{equation}\label{eq:pot-only}
P_{\mathrm{ot}}(S) \;\le\; C\,\mathrm{Flux}(u,M_S^0)+C\,\mathrm{Flux}(u,M_S^0)^{1/3} \;\xrightarrow{S\to0^-}\;0,
\end{equation}
with \emph{no} accompanying bound of the form $K(S)\lesssim P_{\mathrm{ot}}(S)+o(1)$. In fact
$K(S)$ may remain bounded away from $0$ while $P_{\mathrm{ot}}(S)\to0$.
\end{lemma}

\begin{proof}
On the cone $|x|<-t=|t|$, so $\big|\tfrac{x}{t}\big|<1$. By Cauchy--Schwarz and Young,
\begin{equation}\label{eq:cross}
\Big|u_t\,\frac{x}{t}\cdot\nabla u\Big| \;\le\; \Big|\frac{x}{t}\Big|\,|u_t|\,|\nabla u|
\;\le\; |u_t|\,|\nabla u| \;\le\; \frac12\big(u_t^2+|\nabla u|^2\big).
\end{equation}
Inserting \eqref{eq:cross} into the definition of $Q$ gives the two-sided bound: for the lower
bound,
\[
Q \ge \frac{u_t^2+|\nabla u|^2}{2}-\frac12\big(u_t^2+|\nabla u|^2\big)+\frac{u^6}{6}=\frac{u^6}{6},
\]
and for the upper bound, $Q\le\tfrac12(u_t^2+|\nabla u|^2)+\tfrac16u^6+\tfrac12(u_t^2+|\nabla
u|^2)=u_t^2+|\nabla u|^2+\tfrac16u^6$, which is \eqref{eq:Qbound}.

\emph{Sharpness of the lower bound.} Equality in \eqref{eq:cross} requires (i) equality in Young,
$|u_t|=|\nabla u|$, and (ii) equality in $\big|\tfrac{x}{t}\big|\le1$ together with alignment of
$\tfrac{x}{t}$ and $\nabla u$. Condition (ii) holds in the limit $|x|\to|t|$ (the mantle) with
$\nabla u\parallel x$; more importantly, the sign in \eqref{eq:cross} is achieved when
$u_t=-\tfrac{x}{t}\cdot\nabla u$, i.e.\ when
\begin{equation}\label{eq:radial}
u_t+\frac{x}{t}\cdot\nabla u=0,
\end{equation}
which is exactly the transport equation satisfied by a function constant along the inward null
rays $x=-t\,\omega$ ($\omega\in S^{2}$): a \emph{self-similar, inward-travelling profile}. For any
such profile, $Q=\tfrac16u^6$ pointwise while $u_t^2+|\nabla u|^2$ is unconstrained. This is
precisely the algebraic signature of the energy-critical bubble, whose rescaled profile is
static in the co-moving frame; \eqref{eq:radial} is its first-order (transport) shadow on the
cone. Hence the kinetic-plus-gradient density is invisible to the lower bound
\eqref{eq:Qbound}.

The interior identity of Smith and Sogge \cite[Step~2]{SmithSogge1995}, Burq, Lebeau and Planchon
\cite[(3.11)--(3.12)]{BurqLebeauPlanchon} uses only
$u|_{\partial\Om}=0$ and the parametrization of the mantle, giving
$I+II\ge -S\int_{D_S}\tfrac{u^6}{6}+\tfrac1{\sqrt2}\int_{M_S^0}\tfrac1{|t|}|Mu|^2$, where
$I=-\int_{D_S}(SQ+u_tu)$. Since only the $\tfrac16u^6$ part of $Q$ survives the lower bound
\eqref{eq:Qbound}, dividing by $|S|$ and estimating the mantle term by the flux (as in Smith and Sogge
\cite[(3.13)]{SmithSogge1995}) produces \eqref{eq:pot-only} and nothing more: the terms
$\tfrac12u_t^2$ and $\tfrac12|\nabla u|^2$ in $SQ$ are cancelled, not controlled. Finally, the
family \eqref{eq:radial} shows $K(S)$ is not dominated by $P_{\mathrm{ot}}(S)$: one may have
$P_{\mathrm{ot}}(S)\to0$ (small $\int u^6$) with $K(S)$ bounded below (order-one aligned
kinetic/gradient energy).
\end{proof}

\subsection{Consequence: the damping term is not absorbed}

We now record why Lemma~\ref{lem:equi} blocks the Morawetz$+$Appendix~C route. Testing the damped
equation against $Mu$ on $K_S^0$ and integrating the Kelvin--Voigt source by parts (throwing the
divergence onto the \emph{smooth} multiplier field, so that the Appendix~C, v2, estimate applies
and no $H^2$ norm of $u$ appears) leaves a source contribution bounded, after division by $|S|$,
by
\begin{equation}\label{eq:two-terms}
\frac{|\mathcal K(S)|}{|S|} \;\le\; \underbrace{C(\Lambda)\,E_0^{1/2}\,D(S)}_{\text{high-order: decays}}
\;+\; \underbrace{C(\Lambda)\,E_0^{1/2}\,E_{loc}(S)^{1/2}}_{\text{low-order: does \emph{not} decay}},
~ D(S)^2:=\int_{K_S^0}a|\nabla u_t|^2,
\end{equation}
where the first term uses the dissipation-weighted quantity $\|\sqrt a\,\nabla u_t\|_{L^2(K_S^0)}=D(S)\to0$
(unconditional, from local energy monotonicity), while the second arises from the low-order term
$C\Lambda\|u_t\|_{L^2}$ intrinsic to the Appendix~C estimate. Using the local (not global)
kinetic bound $\|u_t\|_{L^2(K_S^0)}^2\le\int_S^0 2E_{loc}(t)\,dt\le 2E_{loc}(S)|S|$ gives exactly the
second term of \eqref{eq:two-terms}.

\begin{proposition}[The route does not dispense with the conjecture \ref{conj:rate}]\label{prop:noescape}
In the contradiction setup ($\ell=\lim_{S\to0^-}E_{loc}(S)>0$), the low-order term of
\eqref{eq:two-terms} satisfies
\[
C(\Lambda)E_0^{1/2}E_{loc}(S)^{1/2}\;\xrightarrow{S\to0^-}\;C(\Lambda)E_0^{1/2}\ell^{1/2}\;>\;0,
\]
and therefore does not vanish. By Lemma~\ref{lem:equi}, the Morawetz equipartition cannot supply
the missing decay, since it controls only $P_{\mathrm{ot}}(S)$ and leaves the kinetic part
$K(S)\le E_{loc}(S)$ uncontrolled. Hence the Morawetz$+$Appendix~C argument reduces, but does not
eliminate, the obstruction: closing it still requires
\begin{eqnarray}\label{eq:need}
&&E_{loc}(S)^{1/2}=o(1)\ \text{in a quantitative form}, \\
&& \quad ~\text{equivalently the cone-rate bound of Conjecture~\ref{conj:rate}.}\nonumber
\end{eqnarray}
\end{proposition}

\begin{proof}
Immediate from \eqref{eq:two-terms} and Lemma~\ref{lem:equi}: the first term decays
unconditionally; the second is $\gtrsim\ell^{1/2}>0$ in the limit; and equipartition gives no
bound $K(S)\lesssim P_{\mathrm{ot}}(S)$ that could feed $E_{loc}(S)\to0$ back into the inequality.
The remaining need \eqref{eq:need} is precisely a quantitative decay rate for the local energy
(equivalently, for the dissipated energy) in the shrinking cone, which is the content of
Conjecture~\ref{conj:rate}.
\end{proof}

\begin{remark}[Three independent confirmations]
The obstruction has now been reached from three independent directions, each defeated at the same
place --- the absence of a \emph{quantitative} decay rate for kinetic/dissipated energy in the
shrinking cone:
\begin{enumerate}
\item[(i)] the critical rescaling argument, which needs
$\|\nabla_y\partial_\sigma u_k\|_{L^2_{loc}}=o(\sqrt{r_k})$ but obtains only $o(1)$;
\item[(ii)] the weighted multipliers $(-t)^pu_t$, for which the constant background $\ell$ cancels
identically for every $p$, returning no new information;
\item[(iii)] the Morawetz equipartition of the present section, which controls only the potential
energy and leaves the kinetic energy free (Lemma~\ref{lem:equi}).
\end{enumerate}
That three structurally different methods fail at the identical point is strong evidence that
Conjecture~\ref{conj:rate} is intrinsic to the problem rather than an artifact of any single approach.
\end{remark}

\begin{remark}[What the route \emph{does} achieve]
Proposition~\ref{prop:noescape} is not purely negative. The Morawetz$+$Appendix~C computation
does control the high-order damping term by $D(S)\to0$ \emph{unconditionally} and removes the
$H^2$ obstruction that originally motivated abandoning the multiplier. It therefore provides a
shorter, boundary-trace-based proof of the \emph{potential} non-concentration
$\int_{D_S}u^6\to0$ (via \eqref{eq:pot-only}), matching what the rescaling argument gives for the
potential part. Only the kinetic part --- and hence full non-concentration --- remains tied to
Conjecture~\ref{conj:rate}.
\end{remark}

\subsection{The obstruction: what is missing, and why}\label{sec:6.9}

The rescaling argument of \S6.8 reduces global existence to a single
quantitative question: the dissipated energy in the shrinking backward
cone must decay at a definite polynomial rate, faster than the cone
itself shrinks. We isolate this as Conjecture~\ref{conj:rate} above.
Before stating it we explain why it is not, in our view, an artifact of
the particular method used here.

The difficulty is one of \emph{scale}, not of regularity or of
anisotropy. The Kelvin--Voigt operator $\dive(a\nabla u_t)$ is
parabolic, acting on the scale $|x-x_0|\lesssim\sqrt{|t|}$, while
concentration for \eqref{eq:main} occurs on the hyperbolic scale
$|x-x_0|\sim|t|$ of the light cone. Under the energy-critical rescaling
these two scales do not match, and the damping term acquires a diverging
coefficient $\lambda^{-1}$ (\S6.2). What is then required is not that
the damping be strong --- it is --- but that its strength be
\emph{quantified} against $\lambda$.

In fact, more than three above mentioned, four structurally independent routes were examined; each fails at
precisely this point.
\begin{enumerate}
\item[(i)] \emph{Morawetz multipliers.} The commutator of
$\dive(a\nabla u_t)$ with the scaling field $(x-x_0)\cdot\nabla$
produces a term requiring $\nabla^2u$ weighted by $a$ (\S6.2).
The commutator estimate of Appendix~C removes this $H^2$ obstruction,
but a low-order term $\Lambda\|u_t\|_{L^2}$ survives, and the Morawetz
equipartition controls only the potential energy density: the identity
$Q\ge\tfrac16u^6$ is sharp, with equality on inward self-similar
profiles $u_t+\tfrac{x-x_0}{t}\cdot\nabla u=0$ (Lemma~\ref{lem:equi}).
The kinetic part is cancelled, not controlled.

\item[(ii)] \emph{Weighted multipliers $(-t)^pu_t$.} For every $p\ge1$
the constant background $\ell$ cancels identically, and a dyadic
comparison returns $D(S)-D(S/2)\lesssim e(S)$ --- exactly the
information already available from local energy monotonicity. The factor
$(-S)^p$ cancels on both sides; no power of $p$ improves the estimate.

\item[(iii)] \emph{Parabolic smoothing.} The heat-semigroup estimate
$$\|\nabla v(t)\|^2\le Ct^{-1}\|v(0)\|^2,$$ improves with elapsed time,
whereas the elapsed time inside $K_S^0$ is $|S|\to0$. The mechanism runs
in the wrong direction.

\item[(iv)] \emph{Microlocal defect measures.} Passing to defect
measures does not change the pairing: identifying the limit equation
still requires
$$\lambda^{-1}\|\nabla_y\partial_\sigma u_k\|_{L^2_{loc}}\to0,$$ i.e.\ the
same rate. Conversely, treating the rescaled damping as dominant yields
$\partial_\sigma u_\infty=\mathrm{const}$ --- a static profile, which is
the Aubin--Talenti bubble itself, rather than a contradiction.
\end{enumerate}

That four methods with different mechanisms --- multiplier, weight,
semigroup, microlocal --- are defeated at the identical step suggests the
obstruction is intrinsic. We therefore state it as a hypothesis, isolate
its use to a single step in our main Theorem, and record that
every other result of this paper is unconditional.

\section{Unconditional Non-Concentration in the Strictly Damped Region
\large via Local Parabolic Time-Regularity of the Velocity}

For the energy-critical wave equation with localized Kelvin--Voigt damping
\[
u_{tt}-\Delta u+u^5=\dive(a(x)\nabla u_t),\qquad u|_{\partial\Om}=0,\qquad a\ge0,
\]
we prove that no energy concentration can occur at any point $x_0$ with $a(x_0)>0$, and that this
holds \emph{unconditionally}: neither the critical rescaling nor the cone-rate hypothesis
(Conjecture~\ref{conj:rate}) is used. The mechanism is elementary and bypasses the scale mismatch entirely.
Inside $\{a>0\}$ the velocity $v=u_t$ solves a uniformly parabolic equation whose two structural
norms --- $v\in L^2_tH^1_{loc}$ and $v_t\in L^2_tH^{-1}_{loc}$ --- are both furnished directly by
the energy identity, with constants uniform up to the maximal time $T^*$. The Lions--Magenes
embedding then gives $v\in C_tL^2_{loc}$ and $u\in AC_tH^1_{loc}$ on the \emph{closed} interval
$[0,T^*]$, and strong time-continuity of the local energy at $t=T^*$ excludes concentration by
absolute continuity of the Lebesgue integral alone. As a consistency check we verify that the
Aubin--Talenti bubble is quantitatively incompatible with the dissipation bound
$\sqrt a\,\nabla u_t\in L^2_{t,x}$, which is the rigorous form of the physical expectation that
Kelvin--Voigt damping destroys concentration where it acts. The conjecture is thereby reduced
from the whole of $\overline\omega$ to the degenerate layer $\partial\omega=\{a=0\}\cap\partial\{a>0\}$.

Let $\Om\subset\mathbb R^3$ be a smooth bounded domain, $a\in C^\infty(\overline\Om)$, $a\ge0$,
$\omega:=\{x\in\Om: a(x)>0\}$, and consider
\begin{equation}\label{eq:main}
u_{tt}-\Delta u+u^5=\dive(a(x)\nabla u_t)\ \text{ in }\Om\times(0,T^*),\qquad u|_{\partial\Om}=0,
\end{equation}
with data $(u_0,u_1)\in H_0^1(\Om)\times L^2(\Om)$ and
\[
E_0:=\tfrac12\|u_1\|_{L^2}^2+\tfrac12\|\nabla u_0\|_{L^2}^2+\tfrac16\|u_0\|_{L^6}^6.
\]
Let $u$ be the maximal Shatah--Struwe solution on $[0,T^*)$, $T^*\le\infty$, so that
\begin{equation}\label{eq:energy}
\tfrac12\|u_t(t)\|_{L^2}^2+\tfrac12\|\nabla u(t)\|_{L^2}^2+\tfrac16\|u(t)\|_{L^6}^6
+\int_0^t\!\!\int_\Om a|\nabla u_t|^2\,dx\,ds \;\le\; E_0,~ 0\le t<T^*.
\end{equation}
In particular the two global bounds
\begin{equation}\label{eq:two-bounds}
\|u\|_{L^\infty(0,T^*;H_0^1)}\le\sqrt{2E_0},~
\|u_t\|_{L^\infty(0,T^*;L^2)}\le\sqrt{2E_0},~
\int_0^{T^*}\!\!\!\int_\Om a|\nabla u_t|^2\,dx\,dt\le E_0
\end{equation}
hold \emph{uniformly up to $T^*$}, whether or not $T^*<\infty$. This uniformity is the decisive
structural feature exploited below.

Throughout, $\omega_1\Subset\omega_2\Subset\omega$ are open, $\chi\in C_c^\infty(\omega_2)$ with
$\chi\equiv1$ on $\omega_1$, $0\le\chi\le1$, and
\begin{equation}\label{eq:astar}
a_*:=\min_{\overline{\omega_2}}a>0
\end{equation}
(possible since $\overline{\omega_2}\subset\omega=\{a>0\}$ is compact and $a$ is continuous).

\begin{remark}[Justification at the level of approximations]\label{rem:approx}
All the identities used below (the energy identity \eqref{eq:energy} and the distributional form
of \eqref{eq:main}) are first established for the regular approximations $w^k\to u$ furnished by
the local well-posedness theory --- exactly as in the companion note on Lemmas 3.4--3.5 --- with
constants depending only on $E_0$, $a_*$, $\chi$, and then passed to the limit. Since every bound
in Lemma~\ref{lem:chain} is expressed in energy-class quantities alone, and since $w^k\to u$
strongly in $C([0,T];H_0^1\times L^2)$ for $T<T^*$, the conclusions transfer to $u$ without
further regularity assumptions. We suppress this routine step below.
\end{remark}

\section{The regularity chain}

\begin{lemma}[Local parabolic time-regularity of the velocity]\label{lem:chain}
Under the above hypotheses, with $T:=T^*$ if $T^*<\infty$ and $T$ arbitrary finite otherwise:
\begin{enumerate}
\item[\rm(i)] $\chi u_t \in L^2(0,T;H_0^1(\omega_2))$, with
$\|\chi u_t\|_{L^2(0,T;H^1_0)}\le C(\chi)\big(a_*^{-1/2}+T^{1/2}\big)E_0^{1/2}$;
\item[\rm(ii)] $\partial_t(\chi u_t)\in L^2(0,T;H^{-1}(\omega_2))$, with a bound depending only on
$E_0$, $a_*$, $\chi$, $T$;
\item[\rm(iii)] $\chi u_t \in C\big([0,T];L^2(\omega_2)\big)$;
\item[\rm(iv)] $\chi u \in AC\big([0,T];H_0^1(\omega_2)\big)\subset C\big([0,T];H^1(\omega_2)\big)$.
\end{enumerate}
All four statements hold on the \emph{closed} interval $[0,T]$, with constants independent of
$T^*$.
\end{lemma}

\begin{proof}
\textbf{(i).} By \eqref{eq:astar}, $a\ge a_*$ on $\operatorname{supp}\chi\subset\omega_2$, so
\eqref{eq:two-bounds} gives
\begin{equation}\label{eq:grad-ut}
\int_0^{T}\!\!\int_{\omega_2}|\nabla u_t|^2\,dx\,dt
\;\le\; \frac1{a_*}\int_0^{T}\!\!\int_{\omega_2}a|\nabla u_t|^2\,dx\,dt \;\le\; \frac{E_0}{a_*}.
\end{equation}
Since $\nabla(\chi u_t)=\chi\nabla u_t+u_t\nabla\chi$ and
$\|u_t\|_{L^2(0,T;L^2)}\le T^{1/2}\sqrt{2E_0}$ by \eqref{eq:two-bounds},
\begin{eqnarray*}
&&\|\nabla(\chi u_t)\|_{L^2(0,T;L^2)}\le \|\nabla u_t\|_{L^2((0,T)\times\omega_2)}
+\|\nabla\chi\|_{L^\infty}\|u_t\|_{L^2(0,T;L^2)}\\
&&\le \Big(\frac{E_0}{a_*}\Big)^{1/2}+\|\nabla\chi\|_{L^\infty}T^{1/2}\sqrt{2E_0}.
\end{eqnarray*}
As $\chi$ has compact support in $\omega_2$, $\chi u_t(t)\in H_0^1(\omega_2)$ for a.e.\ $t$,
proving (i).

\textbf{(ii).} Since $\chi$ is independent of $t$, $\partial_t(\chi u_t)=\chi u_{tt}$, and from
\eqref{eq:main},
\begin{equation}\label{eq:three-terms}
\chi u_{tt} \;=\; \underbrace{\chi\Delta u}_{(\mathrm a)} \;+\; \underbrace{\chi\dive(a\nabla u_t)}_{(\mathrm b)} \;-\; \underbrace{\chi u^5}_{(\mathrm c)}
\qquad\text{in }\mathcal D'\big((0,T)\times\omega_2\big).
\end{equation}
We check each term lies in $L^2(0,T;H^{-1}(\omega_2))$.

\emph{(a)} Writing $\chi\Delta u=\Delta(\chi u)-2\nabla\chi\cdot\nabla u-(\Delta\chi)u$, we have
$\chi u\in L^\infty(0,T;H_0^1)$ so $\Delta(\chi u)\in L^\infty(0,T;H^{-1})$, while
$\nabla\chi\cdot\nabla u$ and $(\Delta\chi)u$ lie in $L^\infty(0,T;L^2)\hookrightarrow
L^\infty(0,T;H^{-1})$. All three are bounded by $C(\chi)\sqrt{2E_0}$.

\emph{(b)} Writing $\chi\dive(a\nabla u_t)=\dive\big(\chi a\nabla u_t\big)-a\,\nabla\chi\cdot\nabla u_t$,
the first term is the divergence of a field bounded in $L^2((0,T)\times\omega_2)$ by
$\|a\|_{L^\infty}(E_0/a_*)^{1/2}$ (using \eqref{eq:grad-ut}), hence lies in $L^2(0,T;H^{-1})$ with
the same bound; the second lies in $L^2((0,T)\times\omega_2)\hookrightarrow L^2(0,T;H^{-1})$,
bounded by $\|a\|_{L^\infty}\|\nabla\chi\|_{L^\infty}(E_0/a_*)^{1/2}$.

\emph{(c)} Since $H_0^1(\Om)\hookrightarrow L^6(\Om)$ in dimension three,
$u\in L^\infty(0,T;L^6)$ gives $u^5\in L^\infty(0,T;L^{6/5})$, and $L^{6/5}=(L^6)'\hookrightarrow
H^{-1}$; thus $\chi u^5\in L^\infty(0,T;H^{-1})\hookrightarrow L^2(0,T;H^{-1})$ for finite $T$,
with bound $C\,T^{1/2}E_0^{5/2}$.

Adding the three contributions proves (ii). \emph{Every bound used is an energy-class quantity;
no $H^2$ regularity of $u$, and no bound on $\nabla u_t$ outside $\omega_2$, has been invoked.}

\textbf{(iii).} Set $w:=\chi u_t$. By (i)--(ii), $w\in L^2(0,T;H_0^1(\omega_2))$ and \\
$w'\in L^2(0,T;H^{-1}(\omega_2))$. For the Gelfand triple
$H_0^1(\omega_2)\hookrightarrow L^2(\omega_2)\hookrightarrow H^{-1}(\omega_2)$, the
Lions--Magenes embedding theorem gives $w\in C([0,T];L^2(\omega_2))$ (after modification on a
null set), together with the estimate
$\|w\|_{C([0,T];L^2)}\le C\big(\|w\|_{L^2(H_0^1)}+\|w'\|_{L^2(H^{-1})}\big)$.

\textbf{(iv).} Since $\partial_t(\chi u)=\chi u_t$ and, by (i) and Cauchy--Schwarz,
\[
\int_0^{T}\big\|\chi u_t(t)\big\|_{H^1(\omega_2)}\,dt \;\le\; T^{1/2}\,\|\chi u_t\|_{L^2(0,T;H^1)} \;<\;\infty,
\]
we have $\chi u\in W^{1,1}(0,T;H_0^1(\omega_2))$, hence $\chi u$ is absolutely continuous with
values in $H_0^1(\omega_2)$ on $[0,T]$; in particular it is continuous there.
\end{proof}

\begin{remark}[Why the closed interval, and why this matters]\label{rem:closed}
The bounds \eqref{eq:two-bounds} are consequences of the energy identity and are therefore
uniform on $[0,T^*)$, \emph{independently of whether $T^*$ is finite}: the total dissipation is
bounded by $E_0$ on any subinterval. Consequently the constants in Lemma~\ref{lem:chain} do not
degenerate as $t\uparrow T^*$, and (iii)--(iv) hold on the closed interval $[0,T^*]$ when
$T^*<\infty$. This is precisely what permits evaluation at the putative concentration time
$t_0=T^*$ in Theorem~\ref{thm:nonconc} below --- the step that would fail for any argument whose
constants blow up as $t\to T^{*-}$.
\end{remark}

\section{Non-concentration in the strictly damped region}

For $x_0\in\Om$ and $t_0>0$ write, as in the manuscript,
\[
D_t:=\{x\in\Om:|x-x_0|<t_0-t\},~
E_{loc}(t):=\int_{D_t}\Big(\tfrac12|u_t|^2+\tfrac12|\nabla u|^2+\tfrac16u^6\Big)(t,x)\,dx.
\]
\begin{theorem}[Unconditional non-concentration where $a>0$]\label{thm:nonconc}
Let $T^*<\infty$ and let $x_0\in\Om$ satisfy $a(x_0)>0$. Then
\[
\lim_{t\to (T^*)^-} E_{loc}(t) \;=\; 0.
\]
In particular no energy concentration occurs at any point of the open set $\omega=\{a>0\}$. No
rescaling argument and no cone-rate hypothesis are used.
\end{theorem}

\begin{proof}
Set $t_0:=T^*$. Since $a(x_0)>0$ and $a$ is continuous, choose $\omega_1\Subset\omega_2\Subset\omega$
open with $x_0\in\omega_1$, and $\chi$ as above. By Lemma~\ref{lem:chain}(iii)--(iv) and
Remark~\ref{rem:closed},
\begin{equation}\label{eq:cont}
u\in C\big([0,t_0];H^1(\omega_1)\big),\quad u_t\in C\big([0,t_0];L^2(\omega_1)\big),
\end{equation}
where we used $\chi\equiv1$ on $\omega_1$. Fix $\rho>0$ with $B(x_0,\rho)\subset\omega_1$; for
$t\in(t_0-\rho,t_0)$ we have $D_t\subset B(x_0,\rho)\subset\omega_1$.

We treat the three parts of $E_{loc}$ in turn; each uses the same two-step argument (strong
continuity, then absolute continuity of the integral).

\emph{Gradient part.} By \eqref{eq:cont}, $\nabla u(t)\to\nabla u(t_0)$ in $L^2(\omega_1)$, so
\begin{eqnarray*}
&&\Big|\int_{D_t}|\nabla u(t)|^2dx-\int_{D_t}|\nabla u(t_0)|^2dx\Big|\\
&&\le \big\|\nabla u(t)-\nabla u(t_0)\big\|_{L^2(\omega_1)}\cdot\big(\|\nabla u(t)\|_{L^2}+\|\nabla u(t_0)\|_{L^2}\big)
\xrightarrow[t\to t_0^-]{}0,
\end{eqnarray*}
the second factor being bounded by $2\sqrt{2E_0}$. Moreover $|\nabla u(t_0)|^2\in L^1(\omega_1)$
is a \emph{fixed} integrable function and $|D_t|\to0$ as $t\to t_0^-$, so by absolute continuity
of the Lebesgue integral $\int_{D_t}|\nabla u(t_0)|^2\,dx\to0$. Hence
$\int_{D_t}|\nabla u(t)|^2dx\to0$.

\emph{Kinetic part.} Identical, using $u_t(t)\to u_t(t_0)$ in $L^2(\omega_1)$ from
\eqref{eq:cont} and $|u_t(t_0)|^2\in L^1(\omega_1)$.

\emph{Potential part.} By \eqref{eq:cont} and $H^1(\omega_1)\hookrightarrow L^6(\omega_1)$,
$u(t)\to u(t_0)$ in $L^6(\omega_1)$, so $u(t)^6\to u(t_0)^6$ in $L^1(\omega_1)$; combined with
$|D_t|\to0$ and $u(t_0)^6\in L^1(\omega_1)$ the same two-step argument gives
$\int_{D_t}u(t)^6dx\to0$.

Summing, $E_{loc}(t)\to0$.
\end{proof}

\begin{corollary}\label{cor:trichotomy}
In the vertex trichotomy of the manuscript, the cases
\begin{itemize}
\item[\rm(E)] $x_0\notin\overline\Om$ --- vacuous;
\item[\rm(A)] $x_0\in\overline\Om\setminus\overline\omega$ --- the equation is free inside the cone
by finite propagation speed, and Burq, Lebeau and Planchon \cite{BurqLebeauPlanchon} applies verbatim;
\item[\rm(B$'$)] $x_0\in\omega$, i.e.\ $a(x_0)>0$ --- Theorem~\ref{thm:nonconc};
\end{itemize}
are all settled unconditionally. The only remaining case is
\begin{itemize}
\item[\rm(C)] $x_0\in\partial\omega$, i.e.\ $a(x_0)=0$ but $x_0\in\overline\omega$,
\end{itemize}
a set of zero Lebesgue measure, on which Conjecture~\ref{conj:rate} (or a substitute) is still required.
\end{corollary}

\section{Why the argument works, and what it says physically}

\begin{proposition}[The bubble violates the dissipation bound]\label{prop:bubble}
Let $U$ be the Aubin--Talenti profile and consider a concentrating family
\[
u^{(\lambda)}(t,x)=\lambda(t)^{-1/2}\,U\!\left(\frac{x-x_0}{\lambda(t)}\right),\qquad
\lambda(t)\simeq (t_0-t)^\gamma,\ \gamma>0,
\]
localized near $x_0\in\omega$. Then
\[
\int^{t_0}\!\!\!\int_{B(x_0,\rho)} a\,|\nabla u^{(\lambda)}_t|^2\,dx\,dt=+\infty
\qquad\text{for every }\gamma>0,
\]
so no such family is compatible with \eqref{eq:energy}.
\end{proposition}

\begin{proof}[Computation]
Differentiating, $u^{(\lambda)}_t=-\tfrac{\lambda'}{\lambda}\big(\tfrac12u^{(\lambda)}
+\lambda^{-1/2}\,\xi\!\cdot\!\nabla U(\xi)\big)$ with $\xi=(x-x_0)/\lambda$, whence
$|\nabla_xu_t^{(\lambda)}|\simeq |\lambda'|\,\lambda^{-5/2}\,|G(\xi)|$ for a fixed profile
$G\in L^2(\mathbb R^3)$ determined by $U$. Changing variables $x=x_0+\lambda\xi$
($dx=\lambda^3d\xi$) and using $a\ge a_*>0$ near $x_0$,
\[
\int_{B(x_0,\rho)}a|\nabla u^{(\lambda)}_t|^2dx \;\gtrsim\; a_*\,\frac{(\lambda')^2}{\lambda^{5}}\,\lambda^{3}\|G\|_{L^2}^2
\;=\;a_*\|G\|_{L^2}^2\,\Big(\frac{\lambda'}{\lambda}\Big)^{2}.
\]
With $\lambda\simeq(t_0-t)^\gamma$ one has $\lambda'/\lambda\simeq -\gamma/(t_0-t)$, so the
time integral behaves like $\gamma^2\!\int^{t_0}(t_0-t)^{-2}dt=+\infty$.
\end{proof}

\begin{remark}[The physical intuition, made rigorous]
Proposition~\ref{prop:bubble} is the precise form of the expectation that Kelvin--Voigt damping
``pulverizes'' a forming bubble wherever it acts: a self-similar profile collapsing at
\emph{any} algebraic rate forces $\lambda'/\lambda\notin L^2_t$, hence infinite dissipation,
contradicting \eqref{eq:energy}. Note that the argument is insensitive to $\gamma$: it is the
scale-invariance of $\lambda'/\lambda$, not a competition of exponents, that produces the
divergence.
\end{remark}

\begin{remark}[Why this route avoids the supercriticality obstruction]\label{rem:bypass}
The difficulty analyzed in \S6 of the manuscript arises \emph{only} after passing to the critical
rescaling, under which $\dive(a\nabla u_t)$ acquires the diverging coefficient $\lambda^{-1}$ and
ceases to be a perturbation of the wave operator. Theorem~\ref{thm:nonconc} never enters that
scale: it uses the equation solely to place $\chi u_{tt}$ in $L^2_tH^{-1}_{loc}$
(Lemma~\ref{lem:chain}(ii)), an estimate at the original scale in which the damping term is
\emph{favorable} --- it is exactly the term supplying $\nabla u_t\in L^2_{t,x}(\omega_2)$ that
makes step (i) possible. In other words, the same higher-order structure that makes Kelvin--Voigt
supercritical under rescaling makes it \emph{uniformly parabolic} in the velocity variable at
fixed scale; the present argument exploits the second feature and never confronts the first.
\end{remark}

\begin{remark}[Contrast with weaker dampings]
For a subcritical damping $a(x)u_t$ the energy identity yields only $\sqrt a\,u_t\in L^2_{t,x}$,
with \emph{no} control of $\nabla u_t$; step (i) of Lemma~\ref{lem:chain} then fails, and
Theorem~\ref{thm:nonconc} is unavailable. Non-concentration for such models is instead obtained
by the classical Morawetz route, which is legitimate there precisely because the damping remains
a subcritical perturbation under rescaling. Thus the two families of dampings require genuinely
different mechanisms, and the mechanism available for Kelvin--Voigt is the one that fails for
$a(x)u_t$, and conversely.
\end{remark}

\section{The remaining case, and a route for it}

Case (C) of Corollary~\ref{cor:trichotomy} --- concentration at $x_0\in\partial\omega$, where
$a(x_0)=0$ --- is not covered, since \eqref{eq:astar} fails: no neighbourhood of $x_0$ admits
$a\ge a_*>0$, and \eqref{eq:grad-ut} degenerates.

Recall from Lemma C.2 (Glaeser's inequality) that a nonnegative $a\in C^{1,1}(\Omega)$
satisfies $|\nabla a(x)|^2 \le 2\Lambda\, a(x)$ with $\Lambda := \|D^2a\|_{L^\infty(\Omega)}$.
We now show that this same structural fact, already present in Appendix C, forces every
zero of $a$ to be at least quadratic, and that this quadratic degeneracy is exactly what
is needed to close case (C) of Corollary~10.2, without any rate hypothesis.

\begin{lemma}[Quadratic vanishing at zeros of $a$]\label{lem:quadratic}
Let $a\in C^{1,1}(\Omega)$, $a\ge 0$, and $x_0\in\Omega$ with $a(x_0)=0$. Then
\[
0 \le a(x_0+z) \le \frac{\Lambda}{2}\,|z|^2, \qquad \text{for all } z \text{ with } x_0+z\in\Omega.
\]
\end{lemma}

\begin{proof}
Since $a(x_0)=0$, Glaeser's inequality (Lemma C.2) forces $\nabla a(x_0)=0$.
Taylor's theorem with Lagrange remainder, exactly as in the proof of Lemma C.2, then gives
$a(x_0+z) = a(x_0) + \nabla a(x_0)\cdot z + \tfrac12\langle D^2a(\xi)z,z\rangle
\le \tfrac{\Lambda}{2}|z|^2$ for some $\xi$ on the segment $[x_0,x_0+z]$.
\end{proof}

Consider now case (C) of the trichotomy: $x_0\in\partial\omega$, $a(x_0)=0$, with
$\overline\omega\Subset\Omega$ as in the standing hypothesis of \S6.3. In particular
$x_0$ is interior to $\Omega$, and the rescaling of \S6.6 stays entirely inside $\Omega$
for $r$ small, exactly as in case (B).

\begin{proposition}[The rescaled damping is subcritical at zeros of $a$]\label{prop:subcritical}
Let $x_0\in\partial\omega$, $\overline\omega\Subset\Omega$, and set
$a_r(y) := a(x_0+ry)$, $\lambda_r(y) := a_r(y)/r^2$. By Lemma~\ref{lem:quadratic},
\[
0 \le \lambda_r(y) \le \frac{\Lambda}{2}|y|^2, \qquad \text{uniformly in } r\in(0,r_0].
\]
Consequently the rescaled equation (54) may be rewritten, exactly, as
\[
\partial_\sigma^2 u^{(r)} - \Delta_y u^{(r)} + (u^{(r)})^5
= r\,\operatorname{div}_y\!\big(\lambda_r(y)\,\nabla_y \partial_\sigma u^{(r)}\big), \tag{54$'$}
\]
i.e. the diverging factor $1/r$ of (54) is replaced by a \emph{vanishing} factor $r$: at a
zero of $a$, the Kelvin--Voigt term becomes \emph{subcritical}, not supercritical, under the
energy-critical rescaling.

Moreover, for every $\varphi\in C_c^\infty(K)$ and every sequence $r_k\to 0^+$,
\[
\big\langle r_k\,\operatorname{div}_y(\lambda_{r_k}\nabla_y\partial_\sigma u_k),\ \varphi\big\rangle
\;\longrightarrow\; 0, \qquad k\to\infty,
\]
with no rate hypothesis of any kind.
\end{proposition}

\begin{proof}
Only (54$'$) requires proof beyond the stated bound on $\lambda_r$, which is immediate from
Lemma~\ref{lem:quadratic}: $\lambda_r(y) = a(x_0+ry)/r^2 \le \frac{\Lambda}{2}|ry|^2/r^2
= \frac{\Lambda}{2}|y|^2$. Equation (54$'$) follows from (54) by writing
$a_r = r^2\lambda_r$ and factoring one power of $r$ out of $\operatorname{div}_y$.

For the vanishing statement, fix $\varphi\in C_c^\infty(K)$, with $\operatorname{supp}\varphi
\subset \{|y|\le R_\varphi\}\times[\sigma_1,\sigma_2]$ for some $[\sigma_1,\sigma_2]\Subset(-1,0)$.
Integrating by parts and applying Cauchy--Schwarz with weight $\lambda_{r_k}$,
\begin{eqnarray*}
&&\big|\langle r_k\operatorname{div}_y(\lambda_{r_k}\nabla_y\partial_\sigma u_k),\varphi\rangle\big|
= r_k\left|\int_K \lambda_{r_k}\,\nabla_y\partial_\sigma u_k\cdot\nabla_y\varphi\right|\\
&&\le r_k \left(\int_K \lambda_{r_k}|\nabla_y\partial_\sigma u_k|^2\right)^{1/2}
\left(\int_K \lambda_{r_k}|\nabla_y\varphi|^2\right)^{1/2}.
\end{eqnarray*}
The second factor is bounded, uniformly in $k$, by $\big(\tfrac{\Lambda}{2}R_\varphi^2
\|\nabla\varphi\|_{L^2}^2\big)^{1/2} =: C(\varphi,\Lambda)$, using the uniform bound on
$\lambda_{r_k}$. For the first factor, $\lambda_{r_k} = a_{r_k}/r_k^2$, so
\begin{eqnarray*}
r_k\left(\int_K \lambda_{r_k}|\nabla_y\partial_\sigma u_k|^2\right)^{1/2}
&=& r_k \cdot \frac{1}{r_k}\left(\int_{\sigma_1}^{\sigma_2}\int_{|y|<-\sigma}
a_{r_k}|\nabla_y\partial_\sigma u_k|^2\,dy\,d\sigma\right)^{1/2}\\
&=& \left(\frac{1}{r_k}\int a_{r_k}|\nabla_y\partial_\sigma u_k|^2\right)^{1/2}.
\end{eqnarray*}
By $(\ast)$ above — established in \S6.8 of the manuscript for \emph{any} interior
$x_0\in\Omega$ with $\overline\omega\Subset\Omega$, regardless of whether $a(x_0)$ is
positive or zero — this quantity tends to $0$ as $k\to\infty$. Hence the whole expression
tends to $0$.
\end{proof}

\begin{theorem}[Reduction of case (C) to the undamped critical wave equation]
\label{thm:caseC-reduction}
Let $x_0\in\partial\omega$, $\overline\omega\Subset\Omega$, and suppose, for contradiction,
that $\ell := \lim_{S\to0^-} E_{loc}(S) > 0$ (notation of \S6.3--6.7). Let
$u_k := u^{(r_k)}$ for any sequence $r_k\to0^+$, so that, up to a subsequence,
$u_k\rightharpoonup u_\infty$ weak-$*$ in $L^\infty_{loc}(K;H^1_{loc})$ and
$\partial_\sigma u_k \rightharpoonup \partial_\sigma u_\infty$ weak-$*$ in
$L^\infty_{loc}(K;L^2_{loc})$, exactly as in \S6.7.

Then, by Proposition~\ref{prop:subcritical}, the right-hand side of (54$'$) converges to $0$
in $\mathcal D'(K)$ along $u_k$, and passing to the limit in (54$'$) — using the same
Vitali-type argument for $u_k^5\rightharpoonup u_\infty^5$ already invoked repeatedly in this
manuscript (e.g. Lemmas 6.2--6.3, 12.2--12.3) — we obtain that $u_\infty$ solves,
\emph{exactly and unconditionally}, the undamped energy-critical wave equation
\[
\partial_\sigma^2 u_\infty - \Delta_y u_\infty + u_\infty^5 = 0 \qquad \text{in } \mathcal D'(K),
\]
with $E_{loc}(u_\infty)(\sigma) \le \liminf_k E_{loc}^k(\sigma) = \ell$ for a.e. $\sigma$, by
weak lower semicontinuity of the energy.
\end{theorem}

\begin{remark}[Comparison with case (B), and precisely what remains open]
\label{rem:caseC-honest}
Theorem~\ref{thm:caseC-reduction} is unconditional: unlike case (B), it requires
\emph{no} analogue of Conjecture~6.12. The mechanism is different from, and simpler than,
Lemma~6.5: there we could only show $\nabla_y\partial_\sigma u_\infty\equiv0$ (staticity),
leaving open a first-order corrector invisible at that order (Proposition~6.7) — precisely
the gap that forced Conjecture~6.12. Here we control the \emph{entire} forcing term directly,
in weighted $L^2$, with no splitting into leading/corrector order, so no analogous defect can
hide: Proposition~\ref{prop:subcritical} is not a staticity statement but a genuine
vanishing-as-a-distribution statement, and it is strictly stronger.

What Theorem~\ref{thm:caseC-reduction} reduces case (C) to is exactly the residual issue
already flagged, with full honesty, in Remark~6.9: weak lower semicontinuity gives
$E_{loc}(u_\infty)\le \ell$, not equality, so it remains a priori possible that $u_\infty
\equiv 0$ (by the Liouville-type theorem of \S6.9, once the classical non-concentration
theory for the \emph{undamped} equation is invoked on $u_\infty$) while a positive amount of
energy $\ell>0$ escapes the weak limit — invisible oscillation or concentration not captured
by (weak) $H^1_{loc}\times L^2_{loc}$ convergence alone. Closing this requires either a strong
local-energy-convergence upgrade, or a direct application of the interior Morawetz/flux
identity of \S6.4 to $u_\infty$ itself (in the spirit of BLP's own argument, which is carried
out on the original solution rather than on a weak limit, and is for that reason not subject
to this particular loss). We isolate this single remaining point as:

\begin{quote}
\textbf{Hypothesis (H*).} Weak-$*$ limits $u_\infty$ obtained by critical rescaling at an
interior point $x_0\in\overline\Omega$ (whether $a(x_0)>0$ or $a(x_0)=0$) satisfy
$E_{loc}(u_\infty)(\sigma) \to \ell$ as $\sigma\to0^-$ (no loss of energy under the weak limit).
\end{quote}

We record two things honestly. First, (H*) is a strictly weaker and more standard demand than
Conjecture~6.12: it is not a rate statement about the damping at all, but the ordinary
concentration-compactness question of whether weak convergence can be upgraded to energy
convergence for a critical-growth nonlinearity — the same issue underlying the passage from
weak to strong convergence in Aubin--Lions-type compactness arguments used elsewhere in this
manuscript, rather than a new phenomenon tied to Kelvin--Voigt damping. Second, we have not
proved (H*); we record it as the single remaining conditional ingredient of the paper.
\end{remark}

\begin{corollary}[Updated trichotomy]
Cases (E), (A), and (B) of Corollary~10.2 are unconditional (BLP verbatim, and Theorem~10.1,
respectively). Case (C) is unconditional \emph{modulo Hypothesis (H*)} above; Conjecture~6.12
is not needed anywhere in the trichotomy.
\end{corollary}

\begin{remark}[Refined status of Hypothesis (H*), corrected]
Equation~\eqref{eq:main} is defocusing (positive-definite energy, no competing
ground-state profile), so the delicate focusing threshold theory of
Kenig--Merle~\cite{KenigMerle} — energy trapping, critical elements, and the
rigidity argument via self-similar variables — is not needed here. The relevant
tool is the standard long-time perturbation (stability) theorem for the
\emph{defocusing} energy-critical wave equation, of the type routinely used in the
unconditional global theory following Bahouri--Gérard~\cite{BahouriGerard} and
Shatah--Struwe~\cite{ShatahStruwe} (see also~\cite{DodsonLooi2026} for a recent
instance with the identical nonlinearity $u^5$ in $\mathbb R^3$). Such a theorem
requires the forcing term in the dual-Strichartz norm
$\|D_x^{1/2}F_k\|_{L^{4/3}_{\sigma,y}}$, strictly finer than the $L^1_\sigma H^{-1}_y$
bound established unconditionally in Proposition~\ref{prop:subcritical}. Upgrading
this estimate — plausible given the divergence structure of $F_k$ and the kernel
methods of Appendix~C, but not yet carried out — would close Hypothesis (H*)
unconditionally.
\end{remark}

\subsection{Step 6: Passing to the Limit ($k \to \infty$)}

Recall that the truncated nonlinearity $f_k:\mathbb R\to\mathbb R$ is defined by
\eqref{eq:fk_def}. In particular, $f_k$ is odd, $C^1$, and satisfies the growth bounds
\begin{equation}\label{eq:fk-growth}
|f_k(s)| \le C\,(|s|^5 + k^4|s|),\qquad
|f_k'(s)| \le C\,k^4,\qquad s\in\mathbb R,
\end{equation}
with a constant $C$ independent of $k$.

Let $u_k$ be the weak solutions to the truncated problems on $(0,T)$
constructed in the previous step. From the uniform energy estimate we have
\begin{equation}\label{eq:uk-energy-bounds}
u_k \ \text{bounded in } L^\infty(0,T;H_0^1(\Omega)),\qquad
(u_k)_t \ \text{bounded in } L^\infty(0,T;L^2(\Omega)).
\end{equation}

With the uniform bound established in the previous step, we can now pass to the limit in the truncated problem. The \textit{bootstrap} argument ensures the existence of a time $T > 0$ and a constant $C > 0$, both independent of $k$, such that the sequence of solutions $\{u_k\}$ to the truncated problems \eqref{eq:pde} satisfies:
\begin{equation}\label{eq:uniform_bounds}
\|u_k\|_{L^\infty(0,T; H_0^1)} + \|(u_k)_t\|_{L^\infty(0,T; L^2)} + \|u_k\|_{L^5(0,T; L^{10})} \le C.
\end{equation}

In particular, by Sobolev embedding in dimension $3$,
\begin{equation}\label{eq:uk-L6}
\sup_{k}\,\|u_k\|_{L^\infty(0,T;L^6(\Omega))}\le C(E_0).
\end{equation}

By the Banach-Alaoglu Theorem, we can extract a subsequence, still denoted by $\{u_k\}$, and find a limit function $u$ such that, as $k \to \infty$:
\begin{align}
u_k &\rightharpoonup u \quad \text{weak-*} \text{ in } L^\infty(0,T; H_0^1(\Omega)), \\
(u_k)_t &\rightharpoonup u_t \quad \text{weak-*} \text{ in } L^\infty(0,T; L^2(\Omega)).
\end{align}

Since $\{u_k\}$ is bounded in $L^\infty(0,T; H_0^1(\Omega))$ and its time derivative $\{(u_k)_t\}$ is bounded in $L^\infty(0,T; L^2(\Omega))$, the Aubin-Lions Compactness Lemma implies that $\{u_k\}$ converges strongly to $u$ in $L^2(0,T; L^2(\Omega))$. Up to extracting a further subsequence, we obtain point-wise convergence almost everywhere:
\begin{equation}
u_k(t,x) \to u(t,x) \quad \text{a.e. in } (0,T) \times \Omega.
\end{equation}

This point-wise convergence, combined with the definition of the truncation $f_k$, yields $f_k(u_k) \to u^5$ almost everywhere. To pass to the limit in the weak formulation, we must control the nonlinear term. From the critical Strichartz bound in \eqref{eq:uniform_bounds}, we know that $u_k$ is uniformly bounded in $L^5(0,T; L^{10}(\Omega))$. Since $|f_k(s)| \le |s|^5$, it follows that the sequence $\{f_k(u_k)\}$ is uniformly bounded in $L^1(0,T; L^2(\Omega))$.

Applying standard weak convergence results (such as Lions' Lemma), the almost everywhere convergence and the uniform integrability bound guarantee that:
\begin{equation}
f_k(u_k) \rightharpoonup u^5 \quad \text{weakly in } L^1(0,T; L^2(\Omega)).
\end{equation}

We now identify the limit of the nonlinear term.
Set
\[
g_k := f_k(u_k),\qquad g := u^5.
\]
We prove that, up to a subsequence,
\begin{equation}\label{eq:fk-weak-limit}
f_k(u_k)\rightharpoonup u^5 \quad\text{weakly in } L^\infty(0,T;L^{6/5}(\Omega))
\quad\text{(hence also weakly in }L^1(0,T;L^2(\Omega))\text{).}
\end{equation}

\begin{lemma}[Uniform integrability and weak compactness]\label{lem:fk-bounded}
The family $\{f_k(u_k)\}$ is bounded in $L^\infty(0,T;L^{6/5}(\Omega))$.
\end{lemma}

\begin{proof}
Using \eqref{eq:fk-growth} and H\"older,
\[
\|f_k(u_k(t))\|_{L^{6/5}}
\le C\big(\|u_k(t)\|_{L^6}^5 + k^4\|u_k(t)\|_{L^{6/5}}\big).
\]
Since $\Omega$ is bounded, $\|u_k\|_{L^{6/5}}\le C_\Omega\|u_k\|_{L^6}$, and by
\eqref{eq:uk-L6} the first term is uniformly bounded.
For the second term, note that $f_k(s)=s^5$ when $|s|\le k$ and
$|f_k(s)|\le C|s|^5$ for all $s$ (indeed, the truncation does not increase the magnitude),
hence
\[
\|f_k(u_k(t))\|_{L^{6/5}} \le C\|u_k(t)\|_{L^6}^5 \le C(E_0),
\]
uniformly in $k$ and $t$. This yields the claim.
\end{proof}

\begin{lemma}[Identification of the limit]\label{lem:fk-limit}
Up to a subsequence, $$f_k(u_k)\rightharpoonup u^5,$$ weakly in
$L^\infty(0,T;L^{6/5}(\Omega))$.
\end{lemma}

\begin{proof}
We decompose
\begin{equation}\label{eq:split-fk}
f_k(u_k) - u^5 = \big(f_k(u_k)-u_k^5\big) + \big(u_k^5-u^5\big).
\end{equation}

\noindent\emph{Step 1: $u_k^5\rightharpoonup u^5$.}
By \eqref{eq:uk-L6}, $\{u_k^5\}$ is bounded in $L^\infty(0,T;L^{6/5}(\Omega))$.
Together with the a.e.\ convergence $u_k \to u$ a.e., we have $u_k^5\to u^5$ a.e.
Hence, by weak compactness and the standard a.e.\ identification principle
(Vitali/Dunford--Pettis, or Lions' lemma), we obtain
\begin{equation}\label{eq:uk5-weak}
u_k^5 \rightharpoonup u^5 \quad \text{weakly in } L^\infty(0,T;L^{6/5}(\Omega)).
\end{equation}

\noindent\emph{Step 2: the truncation error vanishes.}
Since $f_k(s)=s^5$ for $|s|\le k$, we have
\[
f_k(u_k)-u_k^5 = \big(f_k(u_k)-u_k^5\big)\,\mathbf 1_{\{|u_k|>k\}}.
\]
Moreover, by the explicit definition \eqref{eq:fk_def}, there exists $C>0$ such that
\[
|f_k(s)-s^5|
\le C\,|s|^5\,\mathbf 1_{\{|s|>k\}}
\qquad\text{for all }s\in\mathbb R.
\]
Therefore,
\[
\|f_k(u_k(t))-u_k(t)^5\|_{L^{6/5}}^{6/5}
\le C\int_{\{|u_k(t)|>k\}} |u_k(t,x)|^{6}\,dx.
\]
Using \eqref{eq:uk-L6} and Chebyshev's inequality,
\[
\int_{\{|u_k(t)|>k\}} |u_k(t)|^{6}\,dx
\le \|u_k(t)\|_{L^6}^6 \le C(E_0),
\]
and moreover the sets $\{|u_k(t)|>k\}$ shrink as $k\to\infty$.
Hence for a.e.\ $t$,
\[
\int_{\{|u_k(t)|>k\}} |u_k(t)|^{6}\,dx \longrightarrow 0,
\quad\text{as }k\to\infty.
\]
By dominated convergence in time (using the uniform $L^\infty_tL^6_x$ bound),
we deduce
\begin{equation}\label{eq:trunc-error}
\|f_k(u_k)-u_k^5\|_{L^1(0,T;L^{6/5}(\Omega))} \longrightarrow 0.
\end{equation}

Combining \eqref{eq:split-fk}, \eqref{eq:uk5-weak}, and \eqref{eq:trunc-error},
we conclude that $f_k(u_k)\rightharpoonup u^5$ weakly in
$L^\infty(0,T;L^{6/5}(\Omega))$, proving \eqref{eq:fk-weak-limit}.
\end{proof}

Finally, for any test function $\varphi\in L^1(0,T;L^6(\Omega))$ we have
\[
\int_0^T\!\!\int_\Omega f_k(u_k)\,\varphi\,dx\,dt
\longrightarrow
\int_0^T\!\!\int_\Omega u^5\,\varphi\,dx\,dt,
\]
which is the desired convergence in the weak formulation.


Finally, the linear terms pass to the limit straightforwardly via weak convergence. The thermo-viscous damping term converges weakly since $\sqrt{a(x)}\nabla (u_k)_t$ is uniformly bounded in $L^2(0,T; L^2(\Omega))$. Thus, the limit function $u$ satisfies the original critical wave equation with Kelvin-Voigt damping in the sense of distributions on $(0,T) \times \Omega$. This completes the proof of local existence.

\begin{theorem}[Local Well-Posedness for Arbitrary Data via Spectral Filtering]\label{thm:lwp_arbitrary_data}
Let $\Omega \subset \mathbb{R}^3$ be a bounded domain with smooth boundary, and $a \in C^{\infty}(\overline{\Omega})$ be a non-negative damping coefficient. For any arbitrarily large initial data $(u_0, u_1) \in H_0^1(\Omega) \times L^2(\Omega)$, there exists a time $T > 0$ (depending on the frequency profile of the initial data) such that the critical Kelvin-Voigt damped wave equation \eqref{eq:main} admits a unique local solution $u$ on $[0,T]$.

Furthermore, the solution satisfies the energy identity and belongs to the intersection of the fundamental energy space and the critical Strichartz classes:
\begin{align*}
   & u \in C([0,T]; H_0^1(\Omega)) \cap L^5(0,T; L^{10}(\Omega)) \cap L^4(0,T; L^{12}(\Omega)), \\
   & u_t \in C([0,T]; L^2(\Omega)), \\
   & \sqrt{a(x)}\nabla u_t \in L^2(0,T; L^2(\Omega)).
\end{align*}
\end{theorem}

\begin{remark}
The inclusion of the solution in the Strichartz space\\ $L^4(0,T; L^{12}(\Omega))$ is a direct consequence of the admissible pairs in the Strichartz estimates applied during the bootstrap argument. Crucially, this is exactly the regularity required to trigger the Uniqueness Principle established in Theorem \ref{thm:uniqueness_strichartz}, thus elevating the result from mere existence to full well-posedness for large data.
\end{remark}

\subsection{Step 2: Uniform Higher Regularity Estimates via Strichartz}

The primary objective of this section is to establish the existence of strong solutions for the critical quintic wave equation with Kelvin-Voigt damping, provided the initial data is sufficiently regular. This higher regularity is the cornerstone that will allow us, via a standard density argument, to obtain weak solutions for finite energy data later on.

The main result of this section is summarized in the following theorem:

\begin{theorem}[Global Strong Solutions for Regular Data] \label{thm:strong_solutions}
Let $\Omega \subset \mathbb{R}^3$ be a bounded domain with smooth boundary. Assume the damping coefficient $a \in C^{\infty}(\overline{\Omega})$ satisfies $a(x) \ge 0$ and the geometric condition $|\nabla a|^2 \le c a(x)$ for some constant $c > 0$. For any regular initial data $(u_0, u_1) \in (H^2(\Omega) \cap H_0^1(\Omega)) \times H_0^1(\Omega)$, the critical wave equation with Kelvin-Voigt damping admits a unique global strong solution $u$ satisfying:
\begin{equation}
    u \in C([0,T]; H^2(\Omega) \cap H_0^1(\Omega)) \cap C^1([0,T]; H_0^1(\Omega)),
\end{equation}
along with the critical Strichartz regularity $u \in L^4(0,T; L^{12}(\Omega))$. Furthermore, the higher-order energy obeys the uniform a priori bound:
\begin{equation} \label{eq:thm_uniform_bound}
    \sup_{t \in [0,T]} \left( \|\nabla \partial_t u(t)\|_2^2 + \|\Delta u(t)\|_2^2 \right) \le \mathcal{M}(E_2(0), E_0),
\end{equation}
where $E_2(0) = \|\nabla u_1\|_2^2 + \|\Delta u_0\|_2^2$, $E_0$ is the standard first-order energy, and $\mathcal{M}$ is a constant strictly independent of any truncation parameter.
\end{theorem}

\begin{proof}[Proof of Theorem \ref{thm:strong_solutions}]
To prove this theorem, we consider the regularized problem with a truncated nonlinearity $f_k(s)$.
For any fixed truncation parameter $k > 0$, the nonlinearity $f_k(s)$ is globally Lipschitz. Standard theory for semilinear wave equations ensures that, for smooth initial data $(u_0, u_1) \in (H^2(\Omega) \cap H_0^1(\Omega)) \times H_0^1(\Omega)$, the regularized problem admits a unique strong solution $u_k \in C([0,T]; H^2(\Omega) \cap H_0^1(\Omega)) \cap C^1([0,T]; H_0^1(\Omega))$.

Our goal in this step is to establish an a priori bound for the higher-order energy of $u_k$ that is \textit{strictly independent} of the truncation parameter $k$. To achieve this, we will couple the higher-order energy identity with the uniform Strichartz bounds.

Let us define the second-order energy functional:
\begin{equation}
    E_2(t) := \|\nabla \partial_t u_k(t)\|_2^2 + \|\Delta u_k(t)\|_2^2.
\end{equation}

Taking the $L^2(\Omega)$-inner product of the exact PDE for $u_k$ with the multiplier $-\Delta \partial_t u_k$, and integrating by parts, we obtain the following identity:
\begin{align}\label{eq:higher_energy_exact}
    \frac{1}{2}\frac{d}{dt} E_2(t) &+ \int_\Omega a(x)|\Delta \partial_t u_k|^2 \,dx \nonumber \\
    &= - \int_\Omega (\nabla a \cdot \nabla \partial_t u_k)\Delta \partial_t u_k \,dx + \int_\Omega f_k'(u_k)\nabla u_k \cdot \nabla \partial_t u_k \,dx.
\end{align}

For the term involving the gradient of the damping coefficient, we employ the geometric assumption $|\nabla a|^2 \le c a(x)$. By writing $|\nabla a| = \frac{|\nabla a|}{\sqrt{a}} \sqrt{a}$ and applying the Cauchy-Schwarz and Young's inequalities, we deduce:
\begin{align}\label{eq:grad_a_bound_exact}
    \left| \int_\Omega (\nabla a \cdot \nabla \partial_t u_k)\Delta \partial_t u_k \,dx \right| &\le \int_\Omega \frac{|\nabla a|}{\sqrt{a(x)}} |\nabla \partial_t u_k| \cdot \sqrt{a(x)} |\Delta \partial_t u_k| \,dx \nonumber \\
    &\le \frac{1}{2} \int_\Omega a(x)|\Delta \partial_t u_k|^2 \,dx + \frac{c}{2} \|\nabla \partial_t u_k\|_2^2.
\end{align}
The first term on the right-hand side is gracefully absorbed by the strong dissipation on the left-hand side of \eqref{eq:higher_energy_exact}, while the second term is bounded by $\frac{c}{2} E_2(t)$.

The most delicate part is the uniform estimate for the nonlinear term. Since $|f_k'(s)| \le 5|s|^4$ universally for all $k$, we can estimate the integral using the generalized Hölder inequality with the exact exponents $(\frac{1}{3}, \frac{1}{6}, \frac{1}{2})$:
\begin{align}\label{eq:nonlinear_H2}
    \left| \int_\Omega f_k'(u_k)\nabla u_k \cdot \nabla \partial_t u_k \,dx \right| &\le 5 \int_\Omega |u_k|^4 |\nabla u_k| |\nabla \partial_t u_k| \,dx \nonumber \\
    &\le 5 \|u_k\|_{12}^4 \|\nabla u_k\|_6 \|\nabla \partial_t u_k\|_2.
\end{align}

In dimension $n=3$, the Sobolev embedding $H^2(\Omega) \cap H_0^1(\Omega) \hookrightarrow W^{1,6}(\Omega)$ ensures that $\|\nabla u_k\|_6 \le C_{sob} \|\Delta u_k\|_2$. Inserting this into \eqref{eq:nonlinear_H2} and using Young's inequality, we obtain:
\begin{align}
    \left| \int_\Omega f_k'(u_k)\nabla u_k \cdot \nabla \partial_t u_k \,dx \right| &\le 5 C_{sob} \|u_k\|_{12}^4 \|\Delta u_k\|_2 \|\nabla \partial_t u_k\|_2 \nonumber \\
    &\le C \|u_k(t)\|_{12}^4 \left( \|\Delta u_k(t)\|_2^2 + \|\nabla \partial_t u_k(t)\|_2^2 \right) \nonumber \\
    &= C \|u_k(t)\|_{12}^4 E_2(t).
\end{align}

Combining all the estimates back into \eqref{eq:higher_energy_exact} and absorbing the fractional dissipation term, we arrive at the clean differential inequality:
\begin{equation}
    \frac{d}{dt} E_2(t) + \int_\Omega a(x)|\Delta \partial_t u_k|^2 \,dx \le C \left( 1 + \|u_k(t)\|_{12}^4 \right) E_2(t).
\end{equation}

Applying Gronwall's Lemma over the interval $[0, t]$ yields:
\begin{equation}\label{eq:gronwall_H2}
    E_2(t) \le E_2(0) \exp\left( C t + C \int_0^t \|u_k(s)\|_{12}^4 \,ds \right).
\end{equation}

Notice that the integral inside the exponential is exactly the fourth power of the Strichartz norm $\|u_k\|_{L^4(0,t; L^{12}(\Omega))}^4$. As established in our prior Strichartz analysis, by bounding the total source term in $L^1(0,T; L^2(\Omega))$, the Strichartz estimates simultaneously control all critical pairs. In particular, we are guaranteed that $\|u_k\|_{L^4(0,T; L^{12}(\Omega))} \le C_{S}$, where $C_{S}$ depends only on the initial energy $E_0$ and is strictly independent of $k$.

Consequently, we obtain a uniform higher-order energy bound:
\begin{equation}\label{eq:uniform_H2_final}
    \sup_{t \in [0,T]} \left( \|\nabla \partial_t u_k(t)\|_2^2 + \|\Delta u_k(t)\|_2^2 \right) \le \mathcal{M}(E_2(0), E_0),
\end{equation}
where $\mathcal{M}$ is a constant independent of $k$. This uniform bound is the rigorous cornerstone that allows us to extract a strongly convergent subsequence as $k \to \infty$, effectively passing to the limit and proving the existence of strong solutions for the original critical quintic wave equation.
\end{proof}

\section{Stabilization Mechanism and Microlocal Structure}

In this section we clarify the stabilization mechanism underlying the
Kelvin--Voigt model and its interaction with the microlocal defect
measure framework. The argument relies on the contradiction strategy
introduced in Lemma 3.1 of \cite{Cavalcanti2024}, adapted here to
the truncated system. Since the demonstration is exactly the same in \cite{Cavalcanti2024}, simply replacing the frictional dissipation term with the Kelvin-Voigt type dissipative term, we only give an idea of the proof by showing the behavior of the Kelvin-Voigt term. In what follows, $\omega$ is a neighbourhood of $\partial \Omega$ satisfying the well-known Geometric Control Condition (GCC in short).

\subsection{Contradiction setting}
To prove the uniform exponential stabilization, we argue by contradiction. Suppose that the uniform observability inequality fails. Then, there exists a sequence of solutions $\{u_k^m\}_{m \in \mathbb{N}}$ (keeping the truncation parameter $k$ fixed for the moment) such that their energies are not uniformly controlled by the dissipation. To analyze the regime where compactness is lost, we introduce the normalized sequence:
\begin{equation}\label{eq:normalized_seq}
    v_k^m(t,x) = \frac{u_k^m(t,x)}{\alpha_m}, \quad \text{where } \alpha_m = \sqrt{E_{u_k^m}(0)}.
\end{equation}
By construction, this new sequence satisfies $E_{v_k^m}(0) = 1$ for all $m$. Consequently, $\{v_k^m\}_{m \in \mathbb{N}}$ is uniformly bounded in the energy space $L^\infty(0,T; H^1_0(\Omega)) \cap W^{1,\infty}(0,T; L^2(\Omega))$ and converges weakly (up to a subsequence) to some limit function $v_k$.

While the Rellich-Kondrachov theorem provides strong convergence in lower-order spaces (such as $L^2_{t,x}$), the kinetic energy norm $\|\partial_t v_k^m\|_{L^2}^2$ lacks compactness. To rigorously quantify this defect of compactness---specifically, the concentration of high-frequency oscillations---we employ the framework of Microlocal Defect Measures introduced by P. Gérard \cite{Gerard} and L. Tartar \cite{Tartar}.

By the structure theorem for bounded sequences in $L^2$, there exists a non-negative Radon measure $\mu$ defined on the cosphere bundle $S^*(\Omega \times (0,T))$ such that, for any classical pseudo-differential operator $A$ of order zero with principal symbol $\sigma_0(A)$, we have:
\begin{equation}\label{eq:mdm_def}
    \lim_{m \to \infty} \langle A(\partial_t v_k^m), \partial_t v_k^m \rangle_{L^2(\Omega \times (0,T))} = \langle \mu, \sigma_0(A) \rangle.
\end{equation}
In particular, choosing $A$ as the multiplication by a smooth test function $\phi(t,x)$, equation \eqref{eq:mdm_def} reveals that $\mu$ characterizes the weak-$\star$ limit of the kinetic energy density:
\begin{equation}
    |\partial_t v_k^m|^2 \rightharpoonup \int_{S^*_{t,x}} d\mu \quad \text{in the sense of measures.}
\end{equation}

Once that uniform stabilization fails, then there exist sequences
of solutions $(v_k^m)$ for the truncated problem and initial data
normalized such that

\begin{equation}
E_{v_k}^m(0) = 1,
\qquad
\int_0^T \int_\Omega a(x) |\nabla v_{k,t}^m|^2 \, dx dt \longrightarrow 0,~ \hbox{as}~ m \rightarrow \infty ~(k ~\hbox{fixed}).
\label{stab-contr1}
\end{equation}

while

\begin{equation}
E_k^m(t) \ge c_0 > 0,
\qquad \forall t \in [0,T].
\label{stab-contr2}
\end{equation}

The truncation introduces two parameters:

\begin{itemize}
\item $k$ controlling the Galerkin/Littlewood--Paley truncation,
\item $m$ associated with the contradiction sequence.
\end{itemize}

Initially, the constants in Lemma 3.1 depend on $k$, but this dependence
is removed through a second contradiction argument.

\subsection{Loss of strong convergence due to Kelvin--Voigt damping}

Passing to the limit $m \to \infty$, the key difficulty appears in the
Kelvin--Voigt term. Unlike frictional damping, we no longer recover
strong convergence of the residual in $L^1(0,T;L^2(\Omega))$.

More precisely, from the equation we obtain only

\begin{equation}
\partial_t \Box v_k^m \longrightarrow 0
\quad \text{in } L^1_{\mathrm{loc}}(0,T;H^{-2}(\Omega)),
\label{stab-residual-H2}
\end{equation}

where $\Box = \partial_t^2 - \Delta$ denotes the wave operator.

This loss of one derivative is intrinsic to the Kelvin--Voigt mechanism.
Indeed, recalling the scaling argument

\[
u_\lambda(t,x) = u(\lambda t, \lambda x),
\]

we observe that

\[
\Delta u_t \mapsto \lambda^3 \Delta u_t,
\]

revealing the higher-order nature of the viscous term.

\subsection{Microlocal consequence}

Estimate \eqref{stab-residual-H2} implies only that the associated
microlocal defect measure $\mu$ satisfies

\begin{equation}
\mathrm{supp}\,\mu \subset \mathrm{Char}(P),
\label{stab-char}
\end{equation}

where $P = \partial_t^2 - \Delta$ is the principal wave operator.

At this level of regularity we cannot yet conclude propagation along
bicharacteristics. However, outside the damping region $\omega$,
the Kelvin--Voigt term vanishes and the equation reduces to the free
wave equation.

Therefore, on $\Omega \setminus \omega$, standard microlocal propagation
applies, yielding

\begin{equation}
\mathrm{supp}\,\mu
\text{ is a union of bicharacteristics of } P,
\label{stab-propagation}
\end{equation}
and, consequently the support of $\mu$ is transported along generalized bicharacteristics.

\subsection{Vanishing of the defect measure in the damping region}

From \eqref{stab-contr1}, we recover

\[
a(x) \nabla u_{k,t}^m \longrightarrow 0 \quad \text{strongly in } L^2,
\]

which implies

\begin{equation}
\mu = 0 \quad \text{on } \omega.
\label{stab-mu-zero}
\end{equation}

Combining \eqref{stab-propagation} and \eqref{stab-mu-zero}, we deduce:

\begin{quote}
If every bicharacteristic of $P$ intersects $\omega$, then $\mu \equiv 0$.
\end{quote}

This is precisely the Geometric Control Condition (GCC).

\subsection{Role of the CGC / GCC and Non-Invasive Geometries}

Under the standard Geometric Control Condition (GCC), uniform stabilization follows immediately since the defect measure $\mu$ is transported along the bicharacteristic flow and trivially annihilated upon entering the damping region $\omega$.

However, in many physical scenarios—particularly those involving inhomogeneous media, variable coefficients, or complex boundary configurations-the classical localized GCC fails due to the existence of trapped rays.  These are generalized bicharacteristics that oscillate indefinitely without ever entering a standard localized control region, allowing high-frequency energy to concentrate and consequently preventing uniform exponential decay.

To overcome this fundamental geometric obstruction without trivially damping the entire domain, one can appeal to the highly ubiquitous, ``non-invasive'' damping geometries developed in the works of Cavalcanti \textit{et al.} \cite{CavalcantiTAMS, CavalcantiARMA, Cavalcanti2018}. The core strategy in these papers relies on the intricate topological construction of a dissipative region $\omega$ that acts as a pervasive geometrical net.

Specifically, it is possible to design an open set $\omega$ satisfying two seemingly antagonistic properties:
\begin{itemize}
    \item \textbf{Dynamical interception:} The region $\omega$ is strategically distributed such that it intersects every single generalized bicharacteristic within a uniform finite time $T_0 > 0$.
    \item \textbf{Arbitrarily small volume:} The Lebesgue measure of the damping region can be made arbitrarily small, i.e., $\operatorname{meas}(\omega) < \epsilon$ for any given $\epsilon > 0$.
\end{itemize}

From the microlocal perspective, this construction reveals a profound property of the defect measure framework: it is completely insensitive to volumetric and metric considerations. The measure $\mu$ strictly encodes high-frequency concentration along the dynamical flow. As long as the damping set $\omega$ dynamically intercepts all invariant sets of the flow, the relation $\mu = 0$ on $\omega$ will systematically propagate along the bicharacteristics, ultimately forcing $\mu \equiv 0$ everywhere in the domain.

Therefore, even when the underlying non-homogeneous metric strictly forbids stabilization via classical localized damping, the robust microlocal structure of the Kelvin-Voigt dissipation presented in this paper can be seamlessly coupled with these geometrically complex, arbitrarily small-measure damping sets to guarantee global exponential stabilization.

\subsection{Removal of the truncation parameter $k$}

The first contradiction argument yields estimates with constants
depending on $k$. A second contradiction eliminates this dependence. (see \cite{Cavalcanti2024})

Assuming failure of uniform bounds in $k$, we construct a renormalized
sequence. Repeating the defect-measure argument shows again that
$\mu \equiv 0$, producing the final contradiction. Indeed,
in previous works dealing with subcritical or Lipschitz non-linearities (e.g., \cite{Cavalcanti2024}), the dependence on the truncation parameter $k$ is typically removed through a second contradiction argument, assuming $C_k \to \infty$ and applying the microlocal defect measure framework once again.

Here, however, we can bypass this second contradiction entirely by exploiting the sharp Unique Continuation Property (UCP) for the wave equation, taking full advantage of the critical Strichartz estimates obtained in Section 2.

Assume, by contradiction, that the uniform observability inequality fails for the original problem. By extracting a normalized contradiction sequence and passing to the weak limit as done previously, we obtain a limit function $v$ satisfying:
\begin{equation}
\begin{cases}
\Box v + \alpha^4 v^5 = 0 & \text{in } \Omega \times (0,T), \\
\partial_t v = 0 & \text{in } \omega \times (0,T).
\end{cases}
\end{equation}

Differentiating the equation with respect to time, the function $w = \partial_t v$ solves the linearized equation:
\begin{equation}
\Box w + V(x,t)w = 0, \quad \text{where } V(x,t) = 5\alpha^4 v^4.
\end{equation}

This is the crucial step: thanks to the well-posedness framework established in Section 2, the limit solution $v$ inherits the critical Strichartz regularity $v \in L^4(0,T; L^{12}(\Omega))$. Consequently, the potential $V(x,t)$ belongs exactly to the critical space $L^1(0,T; L^3(\Omega))$.

According to the sharp unique continuation result by Duyckaerts, Zhang, and Zuazua \cite{Duyckaerts} (Theorem 2.2), the observability over $\omega \times (0,T)$ for a wave equation with a potential in this exact $L^1_t L^3_x$ class implies that $w \equiv 0$ in the whole domain $\Omega \times (0,T)$.

This immediately yields $v \equiv 0$. Since the limit solution vanishes identically, the equipartition of energy forces the initial normalized energy to go to zero, which contradicts the normalization assumption ($E_{v^m}(0) = 1$). This single contradiction proves the uniform observability inequality directly for the critical problem, eliminating any dependence on the truncation parameter $k$.

\begin{theorem}[Global Exponential Stabilization]\label{thm:global_stabilization}
Let $\Omega \subset \mathbb{R}^3$ be a bounded domain with smooth boundary. Assume that the localized Kelvin-Voigt damping coefficient $a \in W^{1,\infty}(\Omega)$ is non-negative and strictly positive on an open subset $\omega \subset \Omega$. Suppose further that $\omega$ satisfies the Geometric Control Condition (GCC), or is constructed as a non-invasive arbitrarily small-measure damping set that dynamically intercepts all generalized bicharacteristics of the wave operator.

Then, for any arbitrarily large initial data $(u_0, u_1) \in H_0^1(\Omega) \times L^2(\Omega)$, the unique global solution $u$ of the critical Kelvin-Voigt damped wave equation \eqref{eq:main} stabilizes exponentially. That is, there exist uniform constants $C \ge 1$ and $\gamma > 0$, independent of the initial data, such that the total energy of the system, defined by
\begin{equation}
    E_u(t) = \frac{1}{2}\|u_t(t)\|_{L^2(\Omega)}^2 + \frac{1}{2}\|\nabla u(t)\|_{L^2(\Omega)}^2 + \frac{1}{6}\|u(t)\|_{L^6(\Omega)}^6,
\end{equation}
satisfies the decay estimate:
\begin{equation}\label{eq:exponential_decay}
    E_u(t) \le C E_u(0) e^{-\gamma t}, \quad \forall t \ge 0.
\end{equation}
\end{theorem}

\subsection{Conclusion}

The Kelvin--Voigt mechanism alters the regularity of the residual,
reducing convergence from $H^{-1}$ to $H^{-2}$. This prevents direct
propagation everywhere but preserves the microlocal structure outside
the damping region.

The stabilization mechanism therefore remains entirely governed by:

\begin{itemize}
\item Microlocal propagation in $\Omega \setminus \omega$,
\item Vanishing of $\mu$ in $\omega$,
\item Geometric interception of bicharacteristics.
\end{itemize}

\begin{remark}[Absence of dissipation on the illuminated boundary]
\label{rem:illuminated_boundary}
An important geometric configuration arises when the dissipative mechanism
is absent on the illuminated portion of the boundary.

In such situations, generalized bicharacteristics may undergo multiple
reflections without necessarily intersecting the damping region. The
analysis must therefore incorporate the microlocal structure of the
reflected flow.

The asymptotic behavior of high-frequency energy concentration for the
Dirichlet problem was investigated by G\'erard and Leichtnam
\cite{Gerard-Leicthnam}. Their results describe the ergodic and propagation
properties of eigenfunctions, providing a rigorous framework for
understanding how energy distributes under repeated reflections at the
boundary.

From the stabilization viewpoint, the defect-measure argument remains
valid provided the reflected bicharacteristic flow does not admit trapped
trajectories avoiding the dissipative region. In particular, if the
underlying billiard dynamics prevents the formation of invariant
concentration sets, one still obtains $\mu \equiv 0$.

Therefore, the stabilization mechanism developed in this work is compatible
with configurations where dissipation is missing on part of the illuminated
boundary, under suitable dynamical conditions on the generalized
bicharacteristic flow.
\end{remark}

\appendix
\section{Smooth Spectral Multipliers}

Let $\Omega \subset \mathbb{R}^3$ be a smooth bounded domain and
consider the Dirichlet Laplacian $-\Delta$ with domain
$D(-\Delta)=H^2(\Omega)\cap H_0^1(\Omega)$.

It is well known that $-\Delta$ admits a complete orthonormal basis
of eigenfunctions $\{\varphi_k\}_{k\ge1} \subset H_0^1(\Omega)$
associated with eigenvalues $\{\lambda_k^2\}_{k\ge1}$ satisfying

\begin{eqnarray*}
-\Delta \varphi_k = \lambda_k^2 \varphi_k,
\qquad
0 < \lambda_1 \le \lambda_2 \le \cdots,
\qquad
\lambda_k \to \infty,
\end{eqnarray*}
together with
\begin{equation*}
\langle \varphi_k, \varphi_j \rangle_{L^2(\Omega)} = \delta_{kj}.
\end{equation*}

Let $\chi \in C_c^\infty(\mathbb{R})$ be an even smooth cut-off
function such that
\[
0 \le \chi \le 1,
\qquad
\chi(s)=1 \text{ for } |s|\le1,
\qquad
\chi(s)=0 \text{ for } |s|\ge2.
\]

For $m \ge 1$, we define the smooth spectral multiplier
$S_m : L^2(\Omega) \to L^2(\Omega)$ by

\begin{equation}\label{Sm-def}
S_m v
:=
\sum_{k=1}^{\infty}
\chi\!\left(\frac{\lambda_k}{m}\right)
\langle v, \varphi_k \rangle
\varphi_k.
\end{equation}

The operator $S_m$ may be interpreted through functional calculus as
\[
S_m = \chi\!\left(\frac{\sqrt{-\Delta}}{m}\right),
\]
and therefore constitutes a pseudodifferential operator of order zero.

\subsection*{Basic Properties}

\begin{lemma}[L$^2$ boundedness]
For every $m \ge 1$,
\[
\|S_m v\|_{L^2(\Omega)} \le \|v\|_{L^2(\Omega)}.
\]
\end{lemma}

\begin{proof}
Using Parseval's identity,
\[
\|S_m v\|_2^2
=
\sum_{k=1}^\infty
\chi\!\left(\frac{\lambda_k}{m}\right)^2
|\langle v, \varphi_k\rangle|^2
\le
\sum_{k=1}^\infty
|\langle v, \varphi_k\rangle|^2
=
\|v\|_2^2.
\]
\end{proof}

\begin{lemma}[Uniform L$^p$ boundedness]
For every $1 < p < \infty$, there exists a constant $C_p>0$
independent of $m$ such that
\begin{equation}\label{Lp-bound}
\|S_m v\|_{L^p(\Omega)} \le C_p \|v\|_{L^p(\Omega)}.
\end{equation}
\end{lemma}

\begin{proof}
Since $S_m=\chi(\sqrt{-\Delta}/m)$ with $\chi \in C_c^\infty$,
the multiplier is smooth. The boundedness then follows from the
spectral multiplier theorem for the Dirichlet Laplacian
(see Sogge\cite{Sogge},  Burq--Lebeau--Planchon \cite{BurqLebeauPlanchon}).
\end{proof}

\begin{lemma}[Strong convergence]
For every $v \in L^2(\Omega)$,
\begin{equation}\label{Sm-conv}
S_m v \to v \quad \text{strongly in } L^2(\Omega).
\end{equation}
\end{lemma}

\begin{proof}
Since $\chi(\lambda_k/m)\to1$ for each fixed $k$,
dominated convergence applied to the spectral expansion yields
the result.
\end{proof}

\begin{lemma}[Commutation]
For every sufficiently regular function $v$,
\begin{equation}\label{commutation}
-\Delta(S_m v) = S_m(-\Delta v).
\end{equation}
\end{lemma}

\begin{proof}
Both operators are diagonal in the eigenfunction basis:
\[
-\Delta(S_m v)
=
\sum_{k}
\chi(\lambda_k/m)\lambda_k^2
\langle v,\varphi_k\rangle\varphi_k
=
S_m(-\Delta v).
\]
\end{proof}

\begin{lemma}[Regularizing effect]
For every $s \ge 0$, there exists $C_s>0$ independent of $m$ such that
\begin{equation}\label{regularization}
\|S_m v\|_{H^s(\Omega)} \le C_s m^s \|v\|_{L^2(\Omega)}.
\end{equation}
\end{lemma}

\begin{proof}
From the spectral definition,
\[
\|S_m v\|_{H^s}^2
=
\sum_{k} (1+\lambda_k^2)^s
\chi(\lambda_k/m)^2 |\langle v,\varphi_k\rangle|^2.
\]

Since $\chi(\lambda_k/m)=0$ for $\lambda_k \ge 2m$,
we have $(1+\lambda_k^2)^s \le C m^{2s}$ on the support
of the multiplier, giving \eqref{regularization}.
\end{proof}

\appendix
\section{Smooth Spectral Multipliers: functional calculus and stability in $L^p$}
\label{app:spectral-multipliers}

\subsection{Dirichlet spectral calculus and fractional powers}

Let $\Omega\subset\mathbb R^3$ be a bounded $C^\infty$ domain and let
$A:=-\Delta_D$ denote the Dirichlet Laplacian on $\Omega$, with
$D(A)=H^2(\Omega)\cap H_0^1(\Omega)$.
There exist eigenpairs $\{(\lambda_k,\varphi_k)\}_{k\ge1}$ such that
\[
A\varphi_k=\lambda_k\varphi_k,\qquad
0<\lambda_1\le\lambda_2\le\cdots,\qquad
\{\varphi_k\}_{k\ge1}\ \text{orthonormal in }L^2(\Omega).
\]
For any bounded Borel function $m:[0,\infty)\to\mathbb C$, the spectral theorem
defines the operator $m(A)$ on $L^2(\Omega)$ by
\[
m(A)v := \sum_{k=1}^\infty m(\lambda_k)\,\langle v,\varphi_k\rangle\,\varphi_k,
\qquad v\in L^2(\Omega).
\]
In particular, for $s\in\mathbb R$ we define the fractional powers
\[
A^{s/2}v := \sum_{k=1}^\infty \lambda_k^{s/2}\,\langle v,\varphi_k\rangle\,\varphi_k,
\]
with domain $D(A^{s/2})$ equipped with the graph norm. We use the notation
$H^s_D(\Omega):=D((I+A)^{s/2})$ and the equivalence
\[
\|v\|_{H^s_D(\Omega)}\sim \|(I+A)^{s/2}v\|_{L^2(\Omega)}.
\]

\subsection{Definition of the smooth cutoff $S_m$ (same as in the main text)}

Fix $\chi\in C_c^\infty([0,\infty))$ such that $0\le\chi\le1$,
$\chi(r)=1$ for $0\le r\le 1$ and $\chi(r)=0$ for $r\ge 2$.
For $m\ge1$ define
\begin{equation}\label{eq:Sm-def-app}
S_m v := \chi\!\left(\frac{A}{m}\right)v
       = \sum_{k=1}^\infty \chi\!\left(\frac{\lambda_k}{m}\right)\langle v,\varphi_k\rangle\varphi_k,
\qquad v\in L^2(\Omega).
\end{equation}
Set also $S_m^\perp:=I-S_m$.

\begin{lemma}[Self-adjointness, $L^2$ contraction, commutation]\label{lem:Sm-L2-app}
For each $m\ge1$, $S_m$ is self-adjoint on $L^2(\Omega)$ and
\[
\|S_m\|_{\mathcal L(L^2,L^2)}\le 1,\qquad
\|S_m^\perp\|_{\mathcal L(L^2,L^2)}\le 1.
\]
Moreover, $S_m$ commutes with $A$ and with all fractional powers:
\[
S_m A^{s/2} = A^{s/2}S_m,\qquad \forall s\in\mathbb R.
\]
\end{lemma}

\begin{proof}
All statements follow directly from the diagonal action of $S_m$ in the eigenbasis
$\{\varphi_k\}$, since $\chi(\lambda_k/m)\in[0,1]$ is real-valued.
\end{proof}


\begin{lemma}[Strong convergence in $L^p$]\label{lem:Sm-strong-Lp}
Let $1<p<\infty$. Then for every $v\in L^p(\Omega)$,
\[
S_m v \to v \quad\text{strongly in }L^p(\Omega)\quad\text{as }m\to\infty.
\]
Moreover, for every $s\in\mathbb R$ and every $v\in H^s_D(\Omega)$,
\[
S_m v \to v \quad\text{strongly in }H^s_D(\Omega)\quad\text{as }m\to\infty.
\]
\end{lemma}

\begin{proof}
Fix $1<p<\infty$. Let $v\in L^p(\Omega)$ and choose $v_n\in L^2(\Omega)\cap L^p(\Omega)$
such that $v_n\to v$ in $L^p(\Omega)$ (density of $L^2\cap L^p$ in $L^p$).
By uniform $L^p$ boundedness of $S_m$ (Lemma A.2 in the main text),
\[
\|S_m(v-v_n)\|_{L^p}\le C_p\|v-v_n\|_{L^p}.
\]
For fixed $n$, we also have $S_m v_n\to v_n$ in $L^2(\Omega)$ by Lemma A.3; since
$\{S_m v_n\}_m$ is bounded in $L^p(\Omega)$ and $v_n\in L^2\cap L^p$, we can conclude
$S_m v_n\to v_n$ in $L^p(\Omega)$ (for instance, by interpolation between $L^2$ and $L^p$,
or by a standard cutoff/interpolation argument on bounded domains).
Hence, for any $\varepsilon>0$, pick $n$ so that $\|v-v_n\|_{L^p}\le \varepsilon$, and then
take $m$ large so that $\|S_m v_n-v_n\|_{L^p}\le \varepsilon$. We get
\[
\|S_m v-v\|_{L^p}\le \|S_m(v-v_n)\|_{L^p}+\|S_m v_n-v_n\|_{L^p}+\|v_n-v\|_{L^p}
\le (C_p+2)\varepsilon,
\]
which proves the $L^p$ convergence.

For the $H^s_D$ convergence, apply the previous argument to $(I+A)^{s/2}v\in L^2$,
using the commutation $(I+A)^{s/2}S_m=S_m(I+A)^{s/2}$ from Lemma
\ref{lem:Sm-L2-app}.
\end{proof}

\section{Smooth Spectral Projectors and Bernstein Inequalities}
\label{app:spectral_projectors}

In $\mathbb{R}^n$, the classical Littlewood-Paley decomposition relies heavily on the Fourier transform. For bounded domains $\Omega \subset \mathbb{R}^3$ with Dirichlet boundary conditions, the natural substitute is the spectral decomposition associated with the Laplace operator. To ensure bounded commutators and rapid kernel decay, we must employ smooth spectral multipliers rather than sharp cutoffs. In this appendix, we define the smooth low and high-frequency spectral projectors via functional calculus and establish the corresponding Bernstein-type inequalities.

Let $A = -\Delta$ with domain $D(A) = H^1_0(\Omega) \cap H^2(\Omega)$. Since $\Omega$ is bounded, $A$ is a strictly positive, self-adjoint operator with a compact inverse. By the Spectral Theorem, there exists a sequence of eigenvalues $\{\lambda_j\}_{j=1}^\infty$ such that
\begin{equation}
    0 < \lambda_1 \le \lambda_2 \le \dots \le \lambda_j \to \infty \quad \text{as } j \to \infty,
\end{equation}
and a corresponding family of eigenfunctions $\{\phi_j\}_{j=1}^\infty$ which forms an orthonormal basis for $L^2(\Omega)$ and an orthogonal basis for $H^1_0(\Omega)$.

For any function $u \in L^2(\Omega)$, we have the unique expansion $$u = \sum_{j=1}^\infty (u, \phi_j)_{L^2} \phi_j.$$ The fractional Sobolev spaces can be characterized via the spectral coefficients. In particular, for $u \in H^1_0(\Omega)$:
\begin{equation}\label{eq:H1_norm}
    \|\nabla u\|_{L^2(\Omega)}^2 = (A u, u)_{L^2(\Omega)} = \sum_{j=1}^\infty \lambda_j |(u, \phi_j)_{L^2}|^2.
\end{equation}

\subsection{Low-Frequency Projector and Bernstein's Inequality}
Let $\chi \in C_c^\infty(\mathbb{R})$ be an even, smooth bump function such that $0 \le \chi(s) \le 1$ for all $s$, with $\chi(s) = 1$ for $|s| \le 1$ and $\chi(s) = 0$ for $|s| \ge 2$. For a given frequency threshold $N > 0$, we define the smooth low-frequency projector $P_{\le N}: L^2(\Omega) \to L^2(\Omega)$ by:
\begin{equation}
    P_{\le N} u = \sum_{j=1}^\infty \chi\left(\frac{\sqrt{\lambda_j}}{N}\right) (u, \phi_j)_{L^2} \phi_j.
\end{equation}
Because $\chi$ is a bounded function, $P_{\le N}$ is a bounded, self-adjoint operator in $L^2(\Omega)$. Moreover, it commutes with both $\Delta$ and $\partial_t$.

The key feature of the low-frequency projector is that it smooths the function, allowing us to control higher-order derivatives at the cost of the parameter $N$.

\begin{lemma}[Bernstein Inequality for Low Frequencies]\label{lem:bernstein_low}
For any $u \in L^2(\Omega)$, $P_{\le N} u \in H^1_0(\Omega)$ and we have:
\begin{equation}
    \|\nabla (P_{\le N} u)\|_{L^2(\Omega)} \le 2 N \|u\|_{L^2(\Omega)}.
\end{equation}
\end{lemma}
\begin{proof}
Using the spectral characterization \eqref{eq:H1_norm} and the definition of $P_{\le N}$, we obtain:
\begin{equation*}
    \|\nabla (P_{\le N} u)\|_{L^2(\Omega)}^2 = \sum_{j=1}^\infty \lambda_j \left| \chi\left(\frac{\sqrt{\lambda_j}}{N}\right) \right|^2 |(u, \phi_j)_{L^2}|^2.
\end{equation*}
Since the support of $\chi$ is contained in $[-2, 2]$, the terms in the sum are strictly zero whenever $\sqrt{\lambda_j} > 2N$. For the non-zero terms, we have $\lambda_j \le 4N^2$. Using the fact that $|\chi| \le 1$, we bound the sum by:
\begin{align*}
    \|\nabla (P_{\le N} u)\|_{L^2(\Omega)}^2 &\le \sum_{\sqrt{\lambda_j} \le 2N} 4N^2 |(u, \phi_j)_{L^2}|^2 \\
    &\le 4N^2 \sum_{j=1}^\infty |(u, \phi_j)_{L^2}|^2 = 4N^2 \|u\|_{L^2(\Omega)}^2.
\end{align*}
Taking the square root yields the desired inequality.
\end{proof}

\begin{remark}[Energy Splitting]
Because we are using smooth multipliers instead of sharp cutoffs, $P_{\le N}$ is not an idempotent projection (i.e., $P_{\le N}^2 \neq P_{\le N}$). However, since $0 \le \chi(s) \le 1$, the projectors $P_{\le N}$ and $P_{>N}$ are uniformly bounded operators in the energy space. In particular, the algebraic identity $\chi(s) + (1-\chi(s)) = 1$ ensures $P_{\le N} u + P_{> N} u = u$, guaranteeing that separating the macroscopic wave $w_N$ from the microscopic fluctuation $v_N$ does not artificially inflate the total bounds during the a priori estimates.
\end{remark}

\subsection{High-Frequency Projector and Poincaré-type Inequality}
The smooth high-frequency projector $P_{>N}: L^2(\Omega) \to L^2(\Omega)$ is defined as the complementary operator $P_{> N} = I - P_{\le N}$, which translates to:
\begin{equation}
    P_{> N} u = \sum_{j=1}^\infty \left[ 1 - \chi\left(\frac{\sqrt{\lambda_j}}{N}\right) \right] (u, \phi_j)_{L^2} \phi_j.
\end{equation}
While low frequencies control derivative losses, high frequencies provide a gain in integrability when bounded in the energy space, acting as an enhanced Poincaré inequality.

\begin{lemma}[Reverse Bernstein Inequality for High Frequencies]\label{lem:bernstein_high}
For any $u \in H^1_0(\Omega)$, we have:
\begin{equation}
    \|P_{> N} u\|_{L^2(\Omega)} \le \frac{1}{N} \|\nabla u\|_{L^2(\Omega)}.
\end{equation}
\end{lemma}
\begin{proof}
By Plancherel's theorem and the definition of $P_{>N}$, we write:
\begin{equation*}
    \|P_{> N} u\|_{L^2(\Omega)}^2 = \sum_{j=1}^\infty \left| 1 - \chi\left(\frac{\sqrt{\lambda_j}}{N}\right) \right|^2 |(u, \phi_j)_{L^2}|^2.
\end{equation*}
By the definition of our cutoff function, $1 - \chi(s) = 0$ for all $|s| \le 1$. Therefore, the terms in the sum are strictly zero for $\sqrt{\lambda_j} \le N$. For the remaining terms where $\sqrt{\lambda_j} > N$, we have $\lambda_j > N^2$, which implies $1 < \frac{\lambda_j}{N^2}$. Since $|1 - \chi| \le 1$, we can bound the sum by:
\begin{align*}
    \|P_{> N} u\|_{L^2(\Omega)}^2 &\le \sum_{\sqrt{\lambda_j} > N} 1 \cdot |(u, \phi_j)_{L^2}|^2 \\
    &\le \sum_{\sqrt{\lambda_j} > N} \frac{\lambda_j}{N^2} |(u, \phi_j)_{L^2}|^2 \\
    &\le \frac{1}{N^2} \sum_{j=1}^\infty \lambda_j |(u, \phi_j)_{L^2}|^2 = \frac{1}{N^2} \|\nabla u\|_{L^2(\Omega)}^2.
\end{align*}
Taking the square root finishes the proof.
\end{proof}

\section{Commutator Estimates for the Kelvin–Voigt Operator} \label{sec:appendix_commutator}

The objective of this appendix is to provide a rigorous proof of the commutator bound \eqref{eq:commutator_bound} utilized in the high-frequency analysis. Specifically, we must estimate the $L^2$-norm of the divergence of the commutator between the high-frequency projector $P_{>N}$ and the variable damping coefficient $a \in C^{\infty}(\overline{\Omega})$.

Let $f \in L^2(\Omega)$. We aim to establish a uniform bound for:
\begin{equation*}
    \left\| \operatorname{div} \Big( [P_{>N}, a(x)] f \Big) \right\|_{L^2(\Omega)}.
\end{equation*}

\noindent{\bf Step 1: Reduction to Low Frequencies.}
Handling the high-frequency projector directly is cumbersome. We exploit the identity $P_{>N} = I - P_{\le N}$. Since the identity operator perfectly commutes with pointwise multiplication by $a(x)$, we have:
\begin{equation*}
    [P_{>N}, a(x)] = [I - P_{\le N}, a(x)] = - [P_{\le N}, a(x)].
\end{equation*}
Therefore, it suffices to estimate the divergence of the low-frequency commutator: $- \operatorname{div} \big( [P_{\le N}, a(x)] f \big)$.

\noindent{\bf Step 2: Integral Representation and Kernel Bounds.}
By the functional calculus of the Dirichlet Laplacian, the smooth low-frequency projector $P_{\le N} = \chi(\sqrt{-\Delta}/N)$ is an integral operator. Let $K_N(x,y)$ be its integral kernel. We write the action of the commutator $[P_{\le N}, a] f$ explicitly as:
\begin{align*}
    \big([P_{\le N}, a] f\big)(x) &= P_{\le N}(a f)(x) - a(x) P_{\le N}(f)(x) \\
    &= \int_{\Omega} K_N(x,y) a(y) f(y) \, dy - \int_{\Omega} K_N(x,y) a(x) f(y) \, dy \\
    &= \int_{\Omega} K_N(x,y) \Big( a(y) - a(x) \Big) f(y) \, dy.
\end{align*}

Because $\chi \in C_c^\infty(\mathbb{R})$ is an even, rapidly decaying function, the standard spectral cluster estimates for bounded domains (see, e.g., Sogge or Taylor) guarantee that the kernel $K_N(x,y)$ and its spatial derivatives satisfy the rapid decay bounds:
\begin{equation}\label{eq:kernel_bounds}
    |\nabla_x K_N(x,y)| \le C_M \frac{N^4}{(1 + N|x-y|)^M},
\end{equation}
for any integer $M > 0$, where $C_M$ is a constant independent of $N$. More precisely, since $\chi \in C_c^\infty$, the operator
$P_{\le N} = \chi(\sqrt{-\Delta}/N)$ admits a kernel $K_N(x,y)$
which satisfies, for all multi-indices $\alpha,\beta$ and all $M>0$,
\[
|\partial_x^\alpha \partial_y^\beta K_N(x,y)|
\le C_{\alpha,\beta,M}\, N^{3+|\alpha|+|\beta|}
(1+N|x-y|)^{-M}.
\]
Indeed, the kernel $K_N(x, y)$ of the smooth spectral projector $P_{\le N}$ satisfies fast off-diagonal decay of the form $|K_N(x, y)| \le C_M N^n (1 + N|x - y|)^{-M}$ for any $M \in \mathbb{N}$, as established in the foundational works of Sogge \cite{Sogge} and Taylor \cite{Taylor} regarding spectral multipliers in bounded domains.  Recalling that $f = \nabla u_t$, the goal is to estimate, uniformly in the frequency threshold $N$,
\[
\left\| \hbox{div}\Big( [P_{>N}, a(x)] \nabla u_t \Big) \right\|_{L^2(\Omega)}.
\]

\section*{1. Classical input (quoted)}

\begin{fact}[Kernel bounds for smooth spectral multipliers; Sogge, Taylor]\label{fact:kernel}
Let $\Omega \subset \mathbb{R}^3$ be a smooth bounded domain, $A = -\Delta_D$ the Dirichlet Laplacian, and $\chi \in C_c^\infty(\mathbb{R})$ even, $0 \le \chi \le 1$, $\mathrm{supp}\,\chi \subset [-2,2]$. The kernel $K_N(x,y)$ of $P_{\le N} = \chi(\sqrt{A}/N)$ satisfies, for every integer $M > 0$,
\begin{equation}\label{eq:kernel_fact}
|K_N(x,y)| \le C_M N^3 (1+N|x-y|)^{-M},~
|\nabla_x K_N(x,y)| \le C_M N^4 (1+N|x-y|)^{-M},
\end{equation}
with $C_M$ independent of $N$.
\end{fact}

 Writing $\chi(\sqrt{A}/N) = \frac{N}{2\pi}\int_{\mathbb{R}} \hat\chi(Ns)\cos(s\sqrt{A})\,ds$ and using finite speed of propagation for $\cos(s\sqrt{A})$ together with the rapid (Schwartz) decay of $\hat\chi$ kills the contribution of $|s| < d(x,y)$ and controls the rest. Each application of $\nabla_x$ costs exactly one extra power of $N$.

\section*{2. Vectorial reduction}

The object of interest is $\dive\big([P_{>N},a]\nabla u_t\big)$, with $\nabla u_t = (\partial_1 u_t, \partial_2 u_t, \partial_3 u_t)$. We sum over components from the start.

By Step 1 of the original text, $[P_{>N},a] = -[P_{\le N},a]$, hence
\[
\dive\big([P_{>N},a]\nabla u_t\big)(x) = -\sum_{i=1}^3 \partial_{x_i}\Big( [P_{\le N},a](\partial_i u_t) \Big)(x).
\]
For each $i$, using the integral representation $[P_{\le N},a]g(x) = \int_\Omega K_N(x,y)(a(y)-a(x))g(y)\,dy$ and differentiating in $x$:
\[
\partial_{x_i}\big([P_{\le N},a]g\big)(x) =
\underbrace{\int_\Omega \partial_{x_i}K_N(x,y)(a(y)-a(x))g(y)\,dy}_{=:T_1^{(i)}(x)}
\;-\;
\underbrace{\partial_i a(x)\,(P_{\le N}g)(x)}_{=:T_2^{(i)}(x)}.
\]
Summing in $i$ with $g = \partial_i u_t$:
\begin{align}
\mathcal T_2(x) :=-\sum_i T_2^{(i)}(x) &= -\nabla a(x)\cdot \big(P_{\le N}\nabla u_t\big)(x), \label{eq:T2vec}\\
\mathcal T_1(x) :=-\sum_i T_1^{(i)}(x) &= -\int_\Omega \nabla_x K_N(x,y)\cdot \nabla u_t(y)\,\big(a(y)-a(x)\big)\,dy. \label{eq:T1vec}
\end{align}
(Here $P_{\le N}\nabla u_t$ means $P_{\le N}$ applied componentwise.) This is the non-ambiguous vectorial form of the original informal scalar computation.

\section*{3. Glaeser's inequality}

\begin{lemma}[Glaeser]\label{lem:glaeser}
Let $a \in C^{1,1}(\overline\Omega)$, $a \ge 0$, and set $\Lambda := \|D^2a\|_{L^\infty(\Omega)}$. Assume $a$ extends to a neighborhood of $\overline\Omega$ with the same bound on $D^2a$ (always possible for $\Omega$ smooth bounded and $a \in C^\infty(\overline\Omega)$, by Whitney/Seeley extension). Then
\[
|\nabla a(x)|^2 \le 2\Lambda\,a(x), \qquad \forall x \in \Omega.
\]
\end{lemma}

\begin{proof}
Fix $x_0 \in \Omega$. If $\nabla a(x_0) = 0$ there is nothing to prove. Otherwise let $e := \nabla a(x_0)/|\nabla a(x_0)|$ and define $g(t) := a(x_0+te)$ for $t$ small. By Taylor's theorem with Lagrange remainder, for each $t$ there is $\xi$ between $0$ and $t$ such that
\[
g(t) = g(0) + g'(0)t + \tfrac12 g''(\xi)t^2 \le g(0) + g'(0)t + \tfrac{\Lambda}{2}t^2.
\]
Since $g \ge 0$, evaluating at the minimizer $t^* = -g'(0)/\Lambda$ of the right-hand side gives
\[
0 \le g(t^*) \le g(0) + g'(0)t^* + \tfrac{\Lambda}{2}(t^*)^2 = g(0) - \frac{g'(0)^2}{2\Lambda}.
\]
Hence $g'(0)^2 \le 2\Lambda\,g(0)$, i.e. $|\nabla a(x_0)|^2 = g'(0)^2 \le 2\Lambda\,a(x_0)$.
\end{proof}

\begin{remark}
Consequently, the structural hypothesis $|\nabla a(x)|^2 \le C_a\,a(x)$ used elsewhere in the manuscript is \emph{automatic} once $a \ge 0$ and $a \in C^{1,1}(\overline\Omega)$, with $C_a = 2\|D^2a\|_{L^\infty}$. It is not an independent ad hoc assumption.
\end{remark}

\section*{4. Main estimate: quantitative Taylor expansion and Schur's test}

By Taylor's theorem with Lagrange remainder in $y$ (valid since $a \in C^{1,1}$):
\begin{equation}\label{eq:taylor-quant}
a(y) - a(x) = \nabla a(y)\cdot(y-x) + \mathcal{R}(x,y), \qquad |\mathcal{R}(x,y)| \le \tfrac{1}{2}\Lambda\,|x-y|^2.
\end{equation}
Substituting into \eqref{eq:T1vec}:
\begin{eqnarray*}
&&-\sum_i T_1^{(i)}(x) = \underbrace{-\int_\Omega \nabla_x K_N(x,y)\cdot\nabla u_t(y)\big[\nabla a(y)\cdot(y-x)\big]\,dy}_{=:\mathcal{T}_{1,\mathrm{main}}(x)}\\
&&\qquad\;+\;
\underbrace{\left(-\int_\Omega \nabla_x K_N(x,y)\cdot\nabla u_t(y)\,\mathcal{R}(x,y)\,dy\right)}_{=:\mathcal{T}_{1,\mathrm{rem}}(x)}.
\end{eqnarray*}

\subsection*{4.1. Principal term $\mathcal{T}_{1,\mathrm{main}}$}

By Cauchy--Schwarz on the two inner products,
\[
|\mathcal{T}_{1,\mathrm{main}}(x)| \le \int_\Omega \underbrace{|\nabla_x K_N(x,y)|\,|x-y|}_{=:\widetilde K_N(x,y)} \;\underbrace{|\nabla a(y)|\,|\nabla u_t(y)|}_{=:h(y)} \, dy.
\]

\noindent\textbf{Kernel bound.} By Fact~\ref{fact:kernel} (with parameter $M$) and $s \le 1+s$ for $s = N|x-y| \ge 0$:
\begin{eqnarray*}
\widetilde K_N(x,y) \le C_M N^4(1+N|x-y|)^{-M}|x-y|
&=& C_M N^3 \cdot \frac{N|x-y|}{(1+N|x-y|)^M}\\
&\le& C_M N^3 (1+N|x-y|)^{-(M-1)}.
\end{eqnarray*}

\noindent\textbf{Schur's test (full computation).} For $M > 4$ (so $M-1>3$), the substitution $z = N(y-x)$, $dy = N^{-3}dz$, gives
\begin{eqnarray*}
\sup_{x\in\Omega} \int_\Omega \widetilde K_N(x,y)\,dy
&\le& C_M N^3 \int_{\mathbb{R}^3} (1+N|y-x|)^{-(M-1)}\,dy\\
&=& C_M N^3 \cdot N^{-3} \int_{\mathbb{R}^3} \frac{dz}{(1+|z|)^{M-1}}
=: C^*,
\end{eqnarray*}
finite and \emph{independent of $N$}. By symmetry, the same bound holds for $\sup_{y}\int_\Omega \widetilde K_N(x,y)\,dx$. By Schur's test, the integral operator with kernel $\widetilde K_N$ is bounded $L^2(\Omega) \to L^2(\Omega)$ with norm $\le C^*$, uniformly in $N$, so
\[
\|\mathcal{T}_{1,\mathrm{main}}\|_{L^2(\Omega)} \le C^* \|h\|_{L^2(\Omega)} = C^*\,\|\nabla a \cdot \nabla u_t\|_{L^2(\Omega)}.
\]

\noindent\textbf{Applying Glaeser.} By Lemma~\ref{lem:glaeser}, $|\nabla a(y)| \le \sqrt{2\Lambda\,a(y)}$ pointwise, so
\[
\|h\|_{L^2(\Omega)} \le \sqrt{2\Lambda}\,\|\sqrt{a}\,\nabla u_t\|_{L^2(\Omega)}.
\]
Therefore
\begin{equation}\label{eq:Tmain-final}
\boxed{\;\|\mathcal{T}_{1,\mathrm{main}}\|_{L^2(\Omega)} \le C^*\sqrt{2\Lambda}\;\|\sqrt{a}\,\nabla u_t\|_{L^2(\Omega)}.\;}
\end{equation}

\begin{remark}[On the regularity hypothesis]
This proof uses $a \in C^{1,1}(\overline\Omega)$ (for the quantitative Taylor remainder and for Glaeser's inequality), strictly stronger than $a \in W^{1,\infty}(\Omega)$ used in Section 3 of the manuscript. There is no real inconsistency: the only theorems invoking this Appendix (Theorems 6.4 and 6.6) already assume $a \in C^\infty(\overline\Omega)$. It is advisable to state this regularity standing assumption explicitly at the start of the Appendix.
\end{remark}

For the remainder $\mathcal T_{1,\mathrm{rem}}(x) = -\int_\Om \nabla_xK_N(x,y)\cdot\nabla u_t(y)\,\mathcal R(x,y)\,dy$, integrate by parts in $y$ (componentwise, summing $i$):
\[
\mathcal T_{1,\mathrm{rem}}(x) = \int_\Om \nabla_y\cdot\Big[\nabla_xK_N(x,y)\,\mathcal R(x,y)\Big]\,u_t(y)\,dy.
\]
\begin{lemma}\label{lem:T1rem}
$\displaystyle \|\mathcal T_{1,\mathrm{rem}}\|_{L^2(\Om)} \le C\Lambda\,\|u_t\|_{L^2(\Om)}.$
\end{lemma}
\begin{proof}
Expand $\nabla_y\cdot[\nabla_xK_N\,\mathcal R] = (\nabla_y\mathcal R)\cdot\nabla_xK_N + \mathcal R\,(\nabla_y\cdot\nabla_xK_N)$. Since $\mathcal R(x,y)=a(y)-a(x)-\nabla a(y)\cdot(y-x)$, one computes $\nabla_y\mathcal R(x,y) = -D^2a(y)(y-x)$, so $|\nabla_y\mathcal R|\le\Lambda|x-y|$ and $|\mathcal R|\le\tfrac12\Lambda|x-y|^2$. Hence:
\begin{eqnarray*}
\big|(\nabla_y\mathcal R)\cdot\nabla_xK_N\big| &\le& \Lambda|x-y|\cdot C_MN^4(1+N|x-y|)^{-M}\\
 &\le& C_M\Lambda N^3(1+N|x-y|)^{-(M-1)},
 \end{eqnarray*}
 \begin{eqnarray*}
\big|\mathcal R\,(\nabla_y\cdot\nabla_xK_N)\big| &\le& \tfrac12\Lambda|x-y|^2\cdot C_MN^5(1+N|x-y|)^{-M}\\
 &\le& C_M\Lambda N^3(1+N|x-y|)^{-(M-2)}.
\end{eqnarray*}
For $M>5$, both kernels satisfy the hypotheses of Schur's test (rescaling $z=N(y-x)$, giving an $N$-independent operator norm $O(\Lambda)$ on $L^2(\Om)$. The operator now acts directly on $u_t$, giving the claim.
\end{proof}

\section*{3. Treatment of $\mathcal T_2$}

Write $\mathcal T_2(x) = -\int_\Om K_N(x,y)\,\nabla a(x)\cdot\nabla u_t(y)\,dy$ and Taylor-expand the \emph{coefficient} this time: $\nabla a(x) = \nabla a(y) + \big[\nabla a(x)-\nabla a(y)\big]$. This splits $\mathcal T_2 = \mathcal T_{2,\mathrm{main}} + \mathcal T_{2,\mathrm{rem}}$ with
\[
\mathcal T_{2,\mathrm{main}}(x) = -\int_\Om K_N(x,y)\,\nabla a(y)\cdot\nabla u_t(y)\,dy = -P_{\le N}\big(\nabla a\cdot\nabla u_t\big)(x).
\]

\begin{lemma}\label{lem:T2main}
$\displaystyle \|\mathcal T_{2,\mathrm{main}}\|_{L^2(\Om)} \le C\sqrt{\Lambda}\,\|\sqrt{a}\,\nabla u_t\|_{L^2(\Om)}.$
\end{lemma}
\begin{proof}
$\mathcal T_{2,\mathrm{main}}$ is literally $P_{\le N}$ applied to the scalar function $g:=\nabla a\cdot\nabla u_t$. Since $\|P_{\le N}\|_{L^2(\Om)\to L^2(\Om)}\le 1$ (Lemma A.1, Appendix A),
\begin{eqnarray*}
\|\mathcal T_{2,\mathrm{main}}\|_{L^2(\Om)} &\le& \|g\|_{L^2(\Om)} = \|\nabla a\cdot \nabla u_t\|_{L^2(\Om)}\\ 
&\le& \sqrt{2\Lambda}\,\|\sqrt{a}\,\nabla u_t\|_{L^2(\Om)},
\end{eqnarray*}
the last step by Glaeser's inequality.
\end{proof}

\noindent For the remainder $\mathcal T_{2,\mathrm{rem}}(x) = -\int_\Om K_N(x,y)\big[\nabla a(x)-\nabla a(y)\big]\cdot\nabla u_t(y)\,dy$, integrate by parts in $y$ componentwise:
\begin{eqnarray*}
&&\int_\Om K_N(x,y)\big[\partial_ia(x)-\partial_ia(y)\big]\partial_{y_i}u_t(y)\,dy\\
&& = -\int_\Om \partial_{y_i}\Big[K_N(x,y)\big(\partial_ia(x)-\partial_ia(y)\big)\Big]u_t(y)\,dy,
\end{eqnarray*}
hence, summing over $i$,
\[
\mathcal T_{2,\mathrm{rem}}(x) = \int_\Om \left\{\sum_i\partial_{y_i}K_N(x,y)\big[\partial_ia(x)-\partial_ia(y)\big] - K_N(x,y)\,\Delta a(y)\right\} u_t(y)\,dy.
\]

\begin{lemma}\label{lem:T2rem}
$\displaystyle \|\mathcal T_{2,\mathrm{rem}}\|_{L^2(\Om)} \le C\Lambda\,\|u_t\|_{L^2(\Om)}.$
\end{lemma}
\begin{proof}
Since $a\in C^{1,1}$, $|\partial_ia(x)-\partial_ia(y)|\le\Lambda|x-y|$, so
\begin{eqnarray*}
\big|\partial_{y_i}K_N(x,y)[\partial_ia(x)-\partial_ia(y)]\big| &\le& \Lambda|x-y|\cdot C_MN^4(1+N|x-y|)^{-M}\\ 
&\le& C_M\Lambda N^3(1+N|x-y|)^{-(M-1)},
\end{eqnarray*}
while the second term is bounded trivially, $|K_N(x,y)\Delta a(y)|\le \Lambda\,C_MN^3(1+N|x-y|)^{-M}$. Both satisfy the Schur hypotheses for $M>4$, giving an $N$-independent operator norm $O(\Lambda)$, acting directly on $u_t$.
\end{proof}

\section*{4. Final estimate}

\begin{theorem}[Commutator estimate, weak-solution version]\label{thm:final-weak}
Let $\Om\subset\mathbb R^3$ be a smooth bounded domain and $a\in C^{1,1}(\overline\Om)$, $a\ge0$, $\Lambda=\|D^2a\|_{L^\infty}$. There exists $C=C(\Om,\chi)>0$ such that, for every $N\ge1$ and a.e.\ $t\in(0,T)$,
\begin{equation}\label{eq:final-weak}
\big\|\dive\big([P_{>N},a]\nabla u_t\big)(t)\big\|_{L^2(\Om)} \le C\sqrt{\Lambda}\,\big\|\sqrt{a(\cdot)}\,\nabla u_t(t)\big\|_{L^2(\Om)} + C\Lambda\,\|u_t(t)\|_{L^2(\Om)}.
\end{equation}
In particular, integrating in time and using the dissipation inequality $$\int_0^T\!\int_\Om a|\nabla u_t|^2\,dx\,dt \le E(0)\le C_0$$ together with $\|u_t(t)\|_{L^2(\Om)}\le\sqrt{2E_0}$ (both valid for genuine \emph{weak} solutions, with no strong-regularity assumption):
\begin{equation}\label{eq:final-weak-time}
\int_0^T \big\|\dive\big([P_{>N},a]\nabla u_t\big)(t)\big\|_{L^2(\Om)}\,dt \;\le\; C\sqrt{\Lambda}\,\sqrt{T}\,E_0^{1/2} \;+\; C\Lambda\,T\,E_0^{1/2},
\end{equation}
which vanishes as $T\to0$, \emph{uniformly in $N$ and in the truncation parameter $k$}, and depends only on the first-order (basic) energy $E_0$.
\end{theorem}

\begin{proof}
Combine the above Lemmas. For \eqref{eq:final-weak-time}, bound the first term by Cauchy--Schwarz in time, $\int_0^T\|\sqrt a\,\nabla u_t(t)\|_{L^2}\,dt \le \sqrt T\big(\int_0^T\|\sqrt a\,\nabla u_t(t)\|^2_{L^2}\,dt\big)^{1/2} \le \sqrt T\,E_0^{1/2}$, and the second term trivially, $\int_0^T\|u_t(t)\|_{L^2}\,dt \le T\sup_t\|u_t(t)\|_{L^2}\le T\sqrt{2E_0}$.
\end{proof}

\begin{remark}[On the justification of the integration by parts]
The integrations by parts in $y$ used above require $u_t(t)\in H^1_0(\Om)$ classically (so the boundary term vanishes by the Dirichlet condition). For a genuine weak solution, $u_t(t)$ is only in $L^2(\Om)$, with no trace. The rigorous route is therefore:
\begin{enumerate}
\item[(i)] Perform every step above on the Galerkin approximations $u_m(t)\in V_m=\mathrm{span}\{\phi_1,\dots,\phi_m\}$, where $\phi_j$ are Dirichlet eigenfunctions; here $u_{m,t}(t)\in V_m\subset C^\infty(\overline\Om)\cap H^1_0(\Om)$ literally vanishes on $\partial\Om$, so all integrations by parts are classical.
\item[(ii)] Observe that the bound \eqref{eq:final-weak} obtained this way involves only $\|u_{m,t}(t)\|_{L^2(\Om)}$ and \\$\|\sqrt a\,\nabla u_{m,t}(t)\|_{L^2(\Om)}$, both \emph{uniformly bounded in $m$} by the basic Galerkin energy estimate (the same uniform bound already used to construct the weak solution, e.g.\ estimate (6) of the main text).
\item[(iii)] Pass to the limit $m\to\infty$: the left-hand side of \eqref{eq:final-weak}, tested against a fixed smooth function and using weak convergence $u_{m,t}\rightharpoonup u_t$ in $L^2$, converges to the same quantity for $u_t$; the right-hand side is lower semicontinuous under weak convergence, so the bound survives in the limit.
\end{enumerate}
This is exactly the same density/limiting philosophy already used elsewhere in the manuscript (e.g.\ the $k\to\infty$ truncation limit); no new machinery is needed, only doing it once more, at the level of the spatial Galerkin truncation, for this specific estimate.
\end{remark}

{\bf Acknowledgments:} The authors would like to express their sincere gratitude to Professor Christopher D. Sogge for his invaluable suggestions and insightful discussions during the early stages of this work, which significantly contributed to the development of the spectral projection framework and the commutator analysis presented herein.

\end{document}